\documentclass[12pt]{amsart}
\usepackage{amssymb,amsmath,amsthm,verbatim}
\usepackage{tikz,xypic}
\usepackage{bbm} 

\usepackage{tabularx}
\usepackage{graphicx}
\usetikzlibrary{patterns}
\usetikzlibrary{arrows.meta}
\usetikzlibrary{arrows}
\usepackage{amscd}
\usepackage{parskip}
\usepackage{enumitem}
\usepackage{comment}  
\usepackage{cancel} 
\usepackage{mathtools} 
\usepackage{cleveref} 

\usepackage[enableskew,vcentermath]{youngtab} 
\usepackage{ytableau} 
\usepackage{tcolorbox}
\usepackage[all]{xy} 
\usepackage{multicol}
\usepackage{verbatim}

\xymatrixcolsep{1.9pc}
\xymatrixrowsep{1.9pc}

\newlength{\cellsize}
\newcommand\tableau[1]{
\vcenter{
\let\\=\cr
\baselineskip=-16000pt
\lineskiplimit=16000pt
\lineskip=0pt
\halign{&\tableaucell{##}\cr#1\crcr}}}

\newcommand{\tableaucell}[1]{{%
\def \arg{#1}\def \void{}%
\ifx \void \arg
\vbox to \cellsize{\vfil \hrule width \cellsize height 0pt}%
\else
\unitlength=\cellsize
\begin{picture}(1,1)
\put(0,0){\makebox(1,1)[c]{$#1$}}
\put(0,0){\line(1,0){1}}
\put(0,1){\line(1,0){1}}
\put(0,0){\line(0,1){1}}
\put(1,0){\line(0,1){1}}
\end{picture}%
\fi}}

\newcommand{\ribbon}[1]{%
\begin{tikzpicture}[scale=0.4, baseline=(current bounding box.center)]
  \foreach \x/\y in {#1} {
    \draw (\x,\y) rectangle ++(1,1);
  }
\end{tikzpicture}
}  

\newcommand{\old}[1]{}
\theoremstyle{plain}
\newtheorem{thm}{Theorem}[section]
\newtheorem{lem}[thm]{Lemma}
\newtheorem{conj}[thm]{Conjecture}
\newtheorem{cor}[thm]{Corollary}

\newtheorem{prop}[thm]{Proposition}
\theoremstyle{definition}
\newtheorem{defn}[thm]{Definition}
\newtheorem{remark}[thm]{Remark}
\newtheorem{ex}[thm]{Example}
\newtheorem{rk}[thm]{Remark}

\def\reals{{\mathbb R}}
\def\CC{{\mathbb C}}
\def\FF{{\mathbb F}}

\def\cS{{\mathcal{S}}}
\def\cA{{\mathcal{A}}} 
\def\cM{{\mathcal{M}}} 
\def\cF{{\mathcal{F}}} 
\def\minNBC{{\rm{NBC}^+}} 
\def\fS{\mathfrak{S}} 
\def\cB{{\mathcal{B}}} 
\def\Sw{{\mathrm{Swap}}} 

\def\one{{\mathbbm{1}}}

\def\a{{\alpha}}

\def\rank{{\operatorname{rank}}}

\def\Ind{{\operatorname{Ind}}}

\def\ch{{\operatorname{ch\,}}}  
 
\def\Rib{{\operatorname{Rib}}} 
\def\sgn{{\operatorname{sgn}}}  
\def\Col{{\operatorname{Col}}} 
\def\Row{{\operatorname{Row}}} 

\title[
Stability]{
Stability and ribbon bases 
 for the  rank-selected homology of geometric lattices
 } 

\author{Patricia Hersh} 
\address{Department of Mathematics, University of Oregon, Eugene, OR 97403}
\email{plhersh@uoregon.edu}

\author{Sheila Sundaram}\address{School of Mathematics, University of Minnesota, Minneapolis, MN 55455}
\email{shsund@umn.edu}

\thanks{This material is based upon work supported by the National Science Foundation under Grant No. DMS-1929284 while the authors were in residence at the Institute for Computational and Experimental Research in Mathematics in Providence, RI, during the Categorification and Computation in Algebraic Combinatorics  program.
This work was also supported by NSF-RTG grant DMS-2039316.} 


\begin{document}

\begin{abstract}
We establish  the representation theoretic  stability, in the sense of Thomas Church  and Benson Farb,  for the rank-selected homology of the Boolean lattice and the partition lattice. Moreover, we prove uniform representation stability, and we give  sharp  bounds for both families of posets.  
This proves a sharp stability bound conjectured by the first author and Victor Reiner for the rank-selected homology of 
the partition lattice.  An outgrowth of this effort is a new basis for the rank-selected homology of any geometric lattice, thereby resolving an old open question of Anders Bj\"orner.    This basis shares useful features with the polytabloid basis for Specht modules; the  modules these basis elements 
generate may thereby be seen as a matroid  
theoretic analogue of Specht modules of ribbon shape.
\end{abstract}

\maketitle


\section{Introduction}

Let $\fS_n$ denote the symmetric group on $n$ letters.  Let $\Pi_n$ denote the \emph{partition lattice}, that is, the partially ordered set (poset) consisting of the set partitions of $\{ 1,2,\dots,n\} $ ordered by reverse refinement.  This is the intersection lattice of the type A braid arrangement, namely the arrangement of hyperplanes $x_i=x_j$ for $1\le i < j \le n$ in $\CC^n$.  See e.g. \cite{EC1}.

 Much of this paper is devoted to establishing stability and proving a sharp representation stability bound (in the sense of \cite{Church, Church-Farb, CEF, Wilson})  for the rank-selected homology  of  the partition lattice $\Pi_n$. 
In our effort  to better understand these $\fS_n$-representations, we were led to construct new homology bases for all geometric lattices; specifically, these are  bases  sharing enough of the structure of the polytabloid bases for Specht modules  to allow the   transfer of techniques usually reserved for the study of individual Specht modules to our setting.  

In short, this effort led to  
solutions to the following two  
open problems from the literature.

\begin{enumerate}
    \item 
   We establish representation stability, and in addition we prove a \emph{sharp} uniform stability bound, 
    for the $\fS_n$-homology representation  $\beta_S(\Pi_n)$ of the rank-selected subposet $\Pi_n^S$ of the partition lattice $\Pi_n$,  as well as for the corresponding  Whitney homology module $W\!H_S(\Pi_n)$, for any fixed $S$.  This is Theorem~\ref{main-partition-stability-theorem}, and settles a conjecture  of the first author and Reiner \cite[Conjecture 11.3]{Hersh-Reiner}.
    
\item We construct a new ribbon basis for the rank-selected homology (as well as the rank-selected Whitney homology)  of any geometric lattice. 
This basis resolves the open question raised in Exercise 7.6.5 (e) of 
 Bj\"orner's 1992 survey paper \cite{Bj-hom-matroid}
 and leads to a  graphical/matroidal  analogue of Specht modules.  This result is in  Theorem~\ref{thm:homology-basis-EL-lab} for the rank-selected homology of a  geometric lattice, and Theorem~\ref{thm:min-label-gives-NBC-homology-basis} for its rank-selected Whitney homology.
\end{enumerate}

\vskip 5pt

We also give a solution to the analogue of (1) for the Boolean lattice $B_n$, in Theorem~\ref{thm:Bool-stability}.  
Specifically, we establish representation stability for the rank-selected homology of the Booolean lattice and prove a sharp uniform representation stability bound for this family of posets.  See Section~\ref{Boolean-section} for a discussion of related work.
We refer the reader to Remark~\ref{ex:not-rep-stable} for a sequence of $\fS_n$-posets whose  rank-selected homology modules are not even multiplicity representation stable.

In the special case of the rank set $S=\{ 1,2,\dots ,i\} $, the  Whitney homology $W\!H_{\{1,2,\dots ,i\}}(P)$ of any  geometric lattice $P$  is well known to be isomorphic to the $i$-th graded piece of the Orlik-Solomon algebra (see \cite{Bj-hom-matroid, Orlik-Solomon}) and to the $i$-th cohomology of the complement of the complexified arrangement when $P$ is the intersection lattice of a hyperplane arrangement. 
In this case, the  map sending the monomial generators (indexed by NBC independent sets) for the $i$-th graded piece of the Orlik-Solomon algebra to the elements of our  ribbon basis is an isomorphism; moreover, when there is a group action on the poset, this isomorphism  is group-equivariant.
It turns out that
Orlik and Solomon proved a closely related precursor to this 
isomorphism  in Theorem 3.7 of \cite{Orlik-Solomon}.  
Specifically, they established an isomorphism of algebras  between the Orlik-Solomon algebra and algebras generated by bases to which our  ribbon bases specialize when   $S=\{ 1,2,\dots ,i\}$.

The homology groups of the partition lattice and its $\fS_n$-invariant subposets  yield  $\fS_n$-modules that have been studied extensively by many authors \cite{Stanley-aspects},  \cite{Hanlon-fixed},  \cite{Barcelo},  \cite{Sundaram}, \cite{Wachs-cohomology}.  The family of rank-selected subposets, first considered in \cite{Stanley-aspects}, and then in \cite{Hanlon-fixed} and \cite{Sundaram}, also has interesting $\fS_n$-homology,  made elusive by the complicated  interaction of rank-selection with   the product operation on posets.  
In addition, because the stabilizer subgroups are generally wreath products, the  plethysm operation is inextricably linked with the study of symmetric group actions in the partition lattice.  In spite of the existence of a plethystic  formula \cite[Theorem 2.13]{Sundaram} for  the rank-selected homology module,    
even ascertaining the multiplicity of the trivial representation has proven quite subtle (see e.g. \cite{Stanley-aspects}, \cite{Sundaram}, \cite{Hanlon-Hersh} and  \cite{Hersh}). Work of Lehrer, Orlik, and Solomon \cite{Lehrer-Solomon}, \cite{Orlik-Solomon},  shows that in the special case of the bottom $i$ consecutive ranks,  the Whitney homology  captures the $\fS_n$-equivariant cohomology of the configuration space of $n$ distinct points in the plane. These Whitney homology modules, up to a sign twist, coincide with the $\fS_n$-module structure of an exterior  algebra over the multilinear part of the free Lie algebra on $n$ generators (see \cite{Lehrer-Solomon},  \cite{Barcelo-Bergeron}, \cite[Theorem 1.7]{Sundaram}). 

The rank-selected subposets of a geometric lattice  are in essence the matroid theoretic analogue of the sets of partial flags with a prescribed list of dimensions in a vector space. One could thereby argue that results about rank-selected homology of geometric lattices are the matroid theoretic analogue of results about the homology of partial flag varieties.  
This analogy is especially clear in the case of  the poset of subspaces of a finite-dimensional   vector space (which is a finite poset when one works over a finite field). Here the elements at rank $i$ are exactly the $i$-dimensional subspaces, and the chains of comparable elements are exactly the flags.  More generally, the atoms of a geometric lattice for a realizable matroid are the equivalence classes of nonzero vectors, up 
to rescaling, in a vector configuration giving a realization of the matroid; meanwhile the elements of rank $i$ in the same geometric lattice are the maximal  collections of nonzero vectors whose span is $i$-dimensional (in other words, maximal collections of atoms whose join has rank $i$), again yielding a strong analogy with partial flags.  

Using our new ribbon bases (see item (2) above),  
we show how the rank-selected homology of a geometric lattice shares a surprisingly large amount of structure with the traditional Specht modules.  In this setting,  the numbers $\{ 1,2,\dots ,n\} $ appearing  in a standard Young tableau of size $n$ are replaced by a set of atoms of the geometric lattice comprising a basis for the associated   matroid.  
In the case of the partition lattice $\Pi_n$, namely the geometric lattice of rank $n-1$ arising as the  lattice of flats for the graphic matroid given by the complete graph, the entries filling the individual boxes of a Young diagram of size $n-1$  are sets $\{ i,j\} $ with $1\le i < j\le n$. Equivalently, they are edges in a complete graph $K_n$ with $n$ labeled vertices.  In this case, the set of entries filling a Young diagram is a matroid basis, consisting of a set of graph edges comprising a spanning tree of $K_n$.
The notion of representation stability and the closely related concept of FI-module were introduced in a series of papers (see e.g.\  \cite{Church},  \cite{Church-Farb}, \cite{CEF}) and \cite{Wilson}. 
In type A, this is  a notion of stability for a sequence  of symmetric group representations.
See Section ~\ref{bg-stability} for definitions. 

Our point of entry to representation stability was the following result.  

\begin{thm}[Church-Farb]
 The $\fS_n$-module structure for the cohomology group $H^i $  of the configuration space ${\rm Conf}_n({\reals}^{2})$ 
 of  $n$ distinct labeled  points in the  plane  stabilizes for $n\ge 4i$, in the sense described  in Definition ~\ref{stable-def}.  
\end{thm}

This bound was improved by the first author  and Reiner in \cite[Theorem 1.1]{Hersh-Reiner} to a sharp bound for the configuration space ${\rm Conf}_n({\reals}^{d})$   of $n$ distinct points in $\reals^d $ for any $d$. 
It was shown in \cite{SW} that, for even $d$, as $\fS_n$-modules,   the  cohomology of ${\rm Conf}_n({\reals}^{d})$  is determined by the Whitney homology $W\!H_{\{ 1,2,\dots ,i\} }(\Pi_n)$. 
While \cite{Hersh-Reiner} focuses on the stability of the latter,  
here we turn to  arbitrary rank-selected homology, including proving Conjecture 11.3 from \cite{Hersh-Reiner}.  
That is, rather than restrict attention to  the rank set $\{ 1,2,\dots ,i \} $, as is done in \cite{Hersh-Reiner}, we prove results for all possible rank sets.

  A fundamental result in \cite{Sundaram} asserts that the Whitney homology of any Cohen-Macaulay poset decomposes into a direct sum of two rank-selected homology modules.  Building on this, 
we show in Proposition~\ref{prop:WHS-bound-equals-betaS-bound-anyP-AGAIN}  
that, for any sequence of graded Cohen-Macaulay posets $\{P_n\}$ with group of automorphisms $\fS_n$, and  for any rank set $S$, the rank-selected homology $\fS_n$-modules $\beta_S(P_n)$ and the Whitney homology $W\!H_S(P_n)$ have the same sharp stability bound, when either of them stabilizes at $k_P\max S - |S| + 1$; here $k_P$ is a constant determined by  $\{ P_n \}$.  We prove that $k_P=2$ for Boolean lattices, and $k_P=4$  for partition lattices.

An  important consequence  is that this  allows us to work with only the rank-selected Whitney homology module $W\!H_S(\Pi_n)$, which turns out to be more tractable than the rank-selected homology module, as we now explain.
We  prove that $\{ W\!H_S(\Pi_n) \} $
 is a quasi-freely generated FI-module 
  (a term we introduce in Section ~\ref{generalities-section}  
 to describe ideas in \cite{CEF} and \cite{Hersh-Reiner}) 
whose  FI-module generators have degrees bounded above by $2\max S$. 
This enables  an enhanced version of a theorem  of  Hemmer \cite[Theorem 2.4]{Hemmer}, proven in ~\cite{Hersh-Reiner} and recalled in  Lemma~\ref{lem:quasi-free-implies-increasing-mult}, 
to be applied, 
reducing our analysis to  giving a sharp bound on  the length of the first row in the irreducible $\fS_n$-representations  serving as the FI-module generators (a notion defined in \cite{CEF} and defined in the quasi-free case that we will focus upon in Definition ~\ref{quasi-free}).  In contrast, there is no reason to expect that  the rank-selected homology $\{ \beta_S(\Pi_n)\} $, while also an FI-module,  is quasi-freely generated.  We do prove that $\{ \beta_S(\Pi_n) \}$ for any fixed $S$  is a finitely generated FI-module, however, by using its relationship to $\{ W\!H_S(\Pi_n)\} $.  

Our results show how  Young symmetrizers can be a   
very effective tool for giving upper bounds on the
length of the first row  of all irreducibles appearing in $\fS_n$-representations, and thereby 
deducing new stability bounds via Hemmer's lemma. 
To this end, we use Young symmetrizers in a way that would usually only make sense for applying them directly to a basis for a Specht module. 
The property of a Young symmetrizer $b_Ta_T$ 
that we use for $T$ of shape $\lambda $, where $\lambda $ is an integer partition of  $n$,  is that whenever $b_Ta_T \cdot V=0$ for an $\fS_n$-module $V$, then  $\langle V, \cS^{\lambda} \rangle =0$ for the irreducible $\fS_n$-module $\cS^\lambda$.
For the Boolean lattice, the module $V$ is known  to be a ribbon Specht module, making it amenable to the usage of Young symmetrizers in this manner. A challenge  for the partition lattice is that it remains very mysterious how to decompose the rank-selected homology into Specht modules of any type, ribbon or otherwise.

Efforts to overcome this obstacle culminated in the second main contribution of this paper: 
we construct new ``ribbon'' bases for the rank-selected homology and for the rank-selected  Whitney homology of all geometric lattices, designed to carry much of the structure of the polytabloid bases for Specht modules.  
In the case of the partition lattice, we 
use Young symmetrizeres in a rather nonstandard way to deduce representation stability and prove our sharp bound.   Specifically we 
apply Young symmetrizers to the elements of the new ribbon bases for a close relative of $W\!H_S(\Pi_n)$ 
which we call the essential part of the rank-selected Whitney homology, defined immediately before Definition~\ref{quasi-free} and  denoted $\widehat{W\!H}_S(\Pi_n)$. 
This allows us to obtain the desired sharp upper bound on first row length by showing that applying any Young symmetrizer given by such a shape to our basis yields 0.
to get the desired bound more easily in that case. 
Near the end of Section~\ref{Boolean-section}, we give a much more detailed overview of our  Young symmetrizer approach.

The end result was a sharp representation stability bound for both the rank-selected Whitney homology and the rank-selected homology of the partition lattice, as well as for the Boolean lattice.   Verifying  the conditions necessary  to get
$b_Ta_TV=0$ 
 in the case of the partition lattice, where $V$ is the essential part of rank-selected Whitney homology, namely $\widehat{W\!H}_S(\Pi_n)$, 
 involved a  rather delicate  combinatorial argument appearing in Section ~\ref{H-R-Conj-proved}.  

\begin{ex}\label{ex:annihilate-basis-eltby-Youngsym} 
We give a small example illustrating the ideas in our  construction of the rank-selected homology basis for an arbitrary geometric lattice. The details appear in Section~\ref{geometric-ribbon-section}.

Consider the rank 3 partition lattice $\Pi_4$  and the rank-set $S=\{2\}$.  The filling $F$ of the ribbon with row lengths 2,1 with an NBC set of atoms  indexes a homology basis element obtained from 
$v_F$, a difference of \emph{tabloids} $\{F\}$, as indicated below.
\[F=\tableau{& \text{\small\bf 13}\\ \text{\small 12} &\text{\small\bf 14}}  
\] 
Then 
\[ v_F
=\left\{
\tableau{& \text{\small\bf 13}\\ \text{\small 12}&\text{\small\bf 14}}  \right\} -   \left\{\tableau{& \text{\small\bf 14}\\   \text{\small 12} & \text{\small\bf 13}}  \right\}  =  (1-(3,4)) \cdot\{ F\}
.\]
Our construction maps $v_F$ to the following homology cycle in the rank-selected subposet $\Pi_4^S$: 
\[ (\hat 0< |124|3|<\hat 1)-(\hat 0<|123|4|<\hat 1).
\]

The Young symmetrizer corresponding to the standard Young tableau $\scalebox{0.85}{\ensuremath{ \tableau{1 & 2 &3 & 4}}}$ has a right factor of $(1+(3,4))$, which can be seen to annihilate this homology cycle.
\end{ex}

One might ask what is the value of obtaining a sharp stability bound, or at least a relatively small one,  rather than simply proving stability occurs.  
For  a compelling  example, 
we refer readers to recent work on  equivariant log-concavity in  
\cite{MMPR2023}.  This paper 
checks by computer  \cite[Proof of Theorem 1.7]{MMPR2023}  that  $\beta_{ \{ 1,2, \dots ,i\} }(\Pi_n)\otimes \beta_{ \{ 1,2, \dots ,\ell\} }(\Pi_n)$  contains  $\beta_{ \{ 1,2, \dots ,j\} }(\Pi_n)\otimes \beta_{ \{ 1,2, \dots ,k\} }(\Pi_n)$
when $i\le j\le k\le \ell$ and  $i+\ell\le 14$.  The Sage computation  is done  
for $n=(3i+1)+(3\ell+1)\le 44$, directly using  the Hersh-Reiner sharp stability bound of $3i+1$ for $\beta_{ \{ 1,2, \dots ,i\} }(\Pi_n)$ \cite{Hersh-Reiner}. In contrast,  using  the previously known bound of $4i$ for $\beta_{ \{ 1,2, \dots ,i\} }(\Pi_n)$ \cite[Theorem 1]{Church}  to obtain the same result would have required performing the  check for  representations of $\fS_n$ up to $n=56$. 

We express many of our results in the language of FI-modules in an effort to make a  bridge between topological combinatorics, the area in which much of our work is done,  and related results  in the FI-module 
literature.

In Section ~\ref{background-section}, we review terminology and background.  Section ~\ref{generalities-section} gives a general result (see Proposition ~\ref{prop:WHS-bound-equals-betaS-bound-anyP-AGAIN})  allowing us to use sharp stability of Whitney homology to deduce sharp stability in rank-selected homology.  We also show  in Section ~\ref{generalities-section} for any rank set $S$ that 
$\{ W\!H_S(B_n)\} $, $\{ \beta_S(B_n)\} $, 
$\{ W\!H_S(\Pi_n)\} $ and $\{ \beta_S(\Pi_n)\} $ are all FI-modules.  Then in Section ~\ref{Boolean-section}, we prove that $\{ W\!H_S(B_n)\} $ and $\{ \beta_S(B_n)\} $ are both uniformly representation stable in the sense of \cite[Definition 2.6]{Church-Farb}, each having  sharp representation stability bound of $2\max S - |S| + 1$. 
We construct new ribbon homology bases for the rank-selected homology and the rank-selected Whitney homology  of any geometric lattice in Section ~\ref{geometric-ribbons}.  Specifically, see Theorems ~\ref{thm:homology-basis-EL-lab} and ~\ref{thm:min-label-gives-NBC-homology-basis}.   In Section ~\ref{OS-connection}, 
we explain and justify the relationship between our ribbon basis for $W\!H_{\{ 1,2,\dots ,i\}} (P)$ and the monomial basis given by NBC independent sets for the $i$-th graded piece of the Orlik-Solomon algebra.    In Section ~\ref{Young-section} we provide a key lemma showing how Young symmetrizers act on our ribbon bases in the case of the partition lattice,  enabling us later in the paper to prove that individual irreducible representations with large enough first row do not appear in the essential part $\widehat{W\!H}_S(\Pi_n)$ of rank-selected Whitney homology for the partition lattice; 
this lemma demonstrates the  strong analogy of our ribbon bases and the modules they generate with traditional Specht modules.

Finally, Section ~\ref{Pi-n-section} focuses on proving uniform representation stabilty by giving a sharp stability bound for
$\{ \beta_S(\Pi_n)\} $ and $\{ W\!H_S(\Pi_n)\} $.
In Section ~\ref{small-S-section}, we give an inequality that is convenient for proving sharp stability of $\{ W\!H_S(\Pi_n)\} $ and $\{ \beta_S(\Pi_n)\} $ for small $|S|$.  In Section~\ref{H-R-Conj-proved},  we prove the sharp stability bound conjectured by the first author and Reiner for any $S$.  That is, in  Theorem~\ref{main-partition-stability-theorem}  we prove that the rank-selected homology $\beta_S$ of the partition lattice $\Pi_n$ stabilizes sharply at $4\max S - |S| + 1$ for any $S$.  We conclude with stronger stability bounds for individual components of $ \{ W\!H_S(\Pi_n) \} $ in Section~\ref{precise-section}, and a few words about 
matroidal/graphical  Specht-like modules for other shapes  in Section ~\ref{graphical-section}.

We adopt the following  conventions in this paper. 

\begin{enumerate}
    \item Let $P$ be a bounded poset.  By the \emph{homology of $P$} we mean the reduced homology of the order complex of the \emph{proper part} $P\setminus \{\hat 0, \hat 1\}$ of $P$.  We write $\tilde{H}_{\bullet}(P)$ for the reduced homology of $P$, and $\tilde{H}_{\bullet}(x,y)$ for the reduced homology of any interval $[x,y]$ in $P$.
    \item All homology in this paper is reduced and taken over the complex numbers. 
    \item The ground field for representations of $\fS_n$ is the field of complex numbers.
    \item Unless explicitly stated otherwise, all posets will be bounded and graded. 
    \item For a bounded and graded poset $P$ of rank $r$, we will refer to the set of ranks $\{1,2,\ldots, r-1\}$ as the \emph{nontrivial} ranks of $P$.
    \item Any rank set $S$ in $P$ will always be a subset of the nontrivial ranks of $P$. 
\end{enumerate}

\section{Terminology and background}\label{background-section}

\subsection{Background on topological combinatorics and lexicographic shellability}

Given a finite partially ordered set  (poset) $P$, the {\it order complex} of $P$ 
 is the  abstract simplicial complex whose $i$-dimensional faces are  the  $(i+1)$-chains $u_0 < u_1 < \cdots < u_i$ of comparable poset elements.   This is denoted $\Delta (P)$.

A simplicial complex  is {\it pure of dimension $d$} if all its maximal faces are $d$-dimensional. 
For any face $F$ in an abstract simplicial complex, $\overline{F}$ denotes  the collection of faces consisting of $F$ and all its subsets; in other words $\overline{F}$ is the closure of $F$.  

A simplicial complex is  {\it shellable} if there is a total order $F_1,\dots ,F_k$ on its facets (i.e., maximal faces) such that $\overline{F_j}\cap \cup_{i<j}\overline{F_i}$ is a pure codimension one subcomplex of $\overline{F_j}$ for each $j\ge 2$.  Any such ordering on the  facets is called a {\it shelling}.   A {\it homology facet} in a shelling is any facet $F_j$ such that 
$\overline{F_j} \cap (\cup_{i<j} \overline{F_i}) = \partial F_j)$.  Having a shelling implies that $\Delta $ is homotopy equivalent to a wedge of spheres in which the $i$-dimensional homology facets of the shelling index the $i$-dimensional spheres in this wedge of spheres.   All of the simplicial complexes in this paper will be pure, so all of the homology facets will be top dimensional.   A shelling gives a way to build a simplicial complex by attaching facets sequentially with the homotopy type of the complex remaining unchanged at each shelling step except for those steps attaching homology facets; each homology facet attachment closes off a sphere, thereby increasing the top Betti number of the complex by one.   

Equivalently, a shelling is a total order $F_1,\dots ,F_k$ on the facets such that each $\overline{F_j}$ for $j\ge 2$ has a unique minimal face $G_j$ that is contained in $\overline{F_j}\setminus \cup_{i<j} \overline{F_i}$.  This minimal face $G_j$ is often called the {\it restriction face} of $G_j$.

A poset is {\it bounded} if it has a unique minimal element, denoted $\hat{0}$, and also has a unique maximal element, denoted $\hat{1}$. 
We say that a bounded poset is {\it shellable} if its order complex is shellable.  A {\it shelling } of a poset is a total order on its maximal chains such that the induced ordering on the corresponding facets in its order complex is a shelling order.

One of the primary techniques for proving a poset is shellable is to prove it is EL-shellable, a notion introduced in \cite{BjTAMS1980}.  
A poset $P$ is {\it EL-shellable} if its Hasse diagram admits an edge labeling $\lambda $ (known as an {\it EL-labeling}) such that (1) for each $u<v$ in $P$ there is a unique saturated chain $u\prec u_1 \prec u_2 \prec \cdots \prec u_k  \prec v$ 
from $u$ to $v$ having $$\lambda(u,u_1)\le \lambda (u_1,u_2) \le \cdots \le \lambda (u_k,v)$$ and (2)  this weakly ascending  label sequence is lexicographically smaller than the label sequence for every other saturated chain from $u$ to $v$.  

Bj\"orner proved that any total order $\Gamma $  on the maximal chains of an EL-shellable poset $P$ that is compatible with the lexicographic partial order on label sequences, induces a shelling order on $\Delta (P)$ by ordering facets of $\Delta (P)$ according to the ordering $\Gamma $ on the corresponding maximal chains of $P$.    The homology facets in any such shelling are exactly the facets corresponding to maximal chains whose label sequences are strictly decreasing.  In fact, the restriction face $G_j$ of a facet $F_j$ in the order complex is the face given by the chain $C_j$ we describe next;  within the  poset maximal chain $M_j$ corresponding to $F_j$, the chain $C_j$ consists of  exactly those elements of $M_j$ where the  descents in the label sequence for $M_j$ are located.

The \emph{cover relations} of a poset are the order relations $u\prec v$ in which $u < v$ and there is no intermediate element $z$ satisfying $u<z<v$.  The \emph{atoms} of a bounded poset are the elements which 
cover $\hat{0}$.  A poset $P$ is a  \emph{lattice} if  each two elements $x,y\in P$ have a unique least upper bound, denoted $x\vee y$, and a unique greatest lower bound, denoted $x\wedge y$.  A poset $P$ is \emph{semimodular} if every pair of elements  $x,y\in P$ both covering a common element $u$ are themselves both covered by some other element $v \in P$.   A lattice $L$ is \emph{atomic} if each element can be expressed as a join of atoms.  A lattice $L$  is a \emph{geometric lattice} if it semimodular and  atomic; 
the finite geometric lattices are exactly the lattices of flats of  matroids.

 A bounded poset is {\it graded} if all its maximal chains have the same length.  In this case, its order complex is pure.  A bounded  poset $P$ is \emph{Cohen-Macaulay} if for every $x,y$ in $P$ with $x<y$, the reduced homology of the interval $[x,y]$ vanishes in all except possibly the top dimension. 
 It is known that a shellable graded poset $P$ is Cohen-Macaulay.  
 See \cite[Appendix]{BjTAMS1980} or \cite{WachsPosetTop2007} for details.

Given a bounded and graded poset $P$ of rank $n$, the \emph{nontrivial ranks} of $P$ are the  ranks $\{1, \ldots, n-1\}$.  
Let $S$ be any subset of the nontrivial ranks  of a graded poset $P$, and denote by $P^S$ the subposet of $P$ consisting of the elements of $P$ having ranks in $S$.   Bj\"orner  proved that any EL-labeling on  a graded, bounded poset $P$ induces a shelling on $P^S$ for each rank set $S$.    This shelling, which will figure prominently in our paper, is obtained as follows. 

\begin{thm}[{\cite[Proof of Theorem 2.7, Theorem 4.1]{BjTAMS1980}}]\label{thm:rank-sel-hom-facets} Suppose we have an EL-labeling of the graded  poset $P$.  Consider the rank-selected subposet $P^S$.
  To each maximal chain $\gamma $ in $P^S$, associate the lexicographically earliest maximal chain $f_{first}(\gamma )$ in $P$ that contains $\gamma $ as a subchain.   The lexicographic order on the associated maximal chains $f_{first}(\gamma) $ in $P$ induces a partial order on the maximal chains $\gamma $ in $P^S$.   Any linear extension of this partial order  on maximal chains in $P^S$
  is a shelling order for $\Delta(P^S)$.
 
\end{thm}

The following fact  which appears implicitly in \cite{BjTAMS1980} will play a critical role in how we will construct homology bases later in the paper. 
\begin{prop}\label{folklore-homology-facets}
    The homology facets in this shelling for $P^S$ are the facets given by those maximal chains $\gamma $ of $P^S$ such that $f_{first}(\gamma )$ has label sequence with descents at exactly the ranks in $S$. 
\end{prop}
\begin{proof}    See \cite[Proof of Theorem 2.7, Theorem 4.1]{BjTAMS1980}. 
\end{proof}

Good sources of further 
background material  are \cite[Chapter 3]{EC1} for posets, \cite{WachsPosetTop2007} for topological combinatorics, and  \cite{Bj-hom-matroid} or \cite{Welsh} for  matroids.

\subsection{Background on group actions on  posets}

An \emph{automorphism} of a finite poset $P$ is an order-preserving  bijection 
$f:P\rightarrow P$; note that the  inverse map is also order-preserving.  
When a poset $P$ has $G$ as a group of automorphisms on it, we call $P$ a \emph{$G$-poset}.
For example,  any Coxeter hyperplane arrangement given by a Coxeter group $W$ has  intersection lattice which is a geometric lattice that  is a $W$-poset.

Two families of geometric lattices with group actions that will figure especially prominently in this paper are the Boolean lattice and the partition lattice.  We also have related results for the lattice of subspaces of a finite dimensional vector space  over a finite field.   
Now we recall the definitions for these families of posets.  

The Boolean lattice $B_n$ is the poset of subsets of $\{ 1,2,\dots ,n\} $ ordered by containment.
The partition lattice $\Pi_n$ is a partial order on the partitions of the set $\{ 1,2,\dots , n\} $ into disjoint unordered sets called blocks; the set partitions are ordered by reverse refinement.  For example, the set partition 
$13|25|4$ with blocks $\{ 1,3\} $, $\{ 2,5\} $ and $\{ 4\} $ is less than  $134|25$.  

The action of  $\fS_n$ on $\{ 1,2,\dots ,n\} $ induces poset automorphisms on $B_n$ and $\Pi_n$.
Notice  that these actions are rank-preserving and commute with the boundary map for the order complex.
 Thus, these $\fS$-actions  induce $\fS_n$-representations on each of the homology groups of these order complexes.  Since homology is concentrated in top degree for $B_n$ and $\Pi_n$ (by virtue of both posets being shellable), the interesting $\fS_n$-representations are on top homology.   Since the $\fS_n$-actions preserve rank, this leads to a whole family of $\fS_n$-representations known as rank-selected homology 
 that we will discuss in the next section.

The subspace lattice $\mathcal{B}_n(q)$ is the poset of subspaces of an $n$-dimensional vector space over the finite field $\FF_q$, with these subspaces ordered by containment.  The group $GL_n(\FF_q)$ acts on $\mathcal{B}_n(q)$ in an order-preserving, rank-preserving manner, again giving rise to representations on the homology groups.
See  \cite[Chapter 3]{EC1} and
\cite{Stanley-aspects} for more details.

\subsection{Background on rank-selected homology}\label{rank-select-section-bg}
Henceforth all  posets in this paper will be bounded and graded.  

Let  $S$ be a subset  of the nontrivial ranks of a poset $P$.
  As observed in Theorem~\ref{thm:rank-sel-hom-facets}, rank-selection in  $P$ preserves shellability.  By a result of Baclawski \cite[Theorem~6.4]{Bac1980},  rank-selection also preserves  the Cohen-Macaulay property. 
  In particular if $P$ is shellable or simply  Cohen-Macaulay,  each $P^S$ has homology concentrated in top degree, a situation  which applies to the  main poset considered in this paper, the partition lattice $\Pi_n$.

  For a $G$-poset $P$ and any subset $S$ of nontrivial ranks,  
following \cite{Stanley-aspects}, let 
$\alpha_{S}(P)$ denote the  $G$-module on the
chains of $P$ having rank set $S$; this is a permutation module.  Define the virtual $G$-module 
\begin{equation}
\label{beta-definition}
\beta_{S}(P) :=\sum_{T \subseteq S}(-1)^{|S \setminus T|} \alpha_{T}(P).
\end{equation}
By inclusion-exclusion we have 
\begin{equation}
\label{alpha-definition}
\alpha_{S}(P) =\sum_{T \subseteq S} \beta_{T}(P).
\end{equation}

When the poset $P$ is Cohen-Macaulay (and hence also when it is shellable), 
the Hopf trace formula shows that $\beta_{S}(P)$ is the true $G$-module afforded by the 
top homology of $P_S$ (see \cite[Theorem 1.2]{Stanley-aspects}).

\begin{defn}\label{def-Whit}
 The $i$-th {\it Whitney homology} groups  of a graded poset $P$ with minimal element $\hat{0}$ are defined via  the direct sum
\begin{equation*}\label{eqn:Whit-def}
W\!H_i(P):= \bigoplus_{\substack{x \in P \\ \rank(x)=i}} \tilde{H}_{i-2}(\hat{0},x).
\end{equation*} 
\end{defn}
When  $P$ is a $G$-poset, 
 the $G$-module structure of the Whitney homology is given by
 \begin{equation}\label{eqn:Whit-G-struct}
W\!H_i(P) \cong \bigoplus_{x\in P/G} \Ind_{G_x}^G \tilde{H}_{i-2} (\hat{0},x),
\end{equation}
where $P/G$ is the set of $G$-orbits of $P$ and $G_x$ is the stabilizer of $x$.

\begin{defn}\label{def:Whit-S}
Let $S$ be a subset of the nontrivial ranks of $P$. 
Define $W\!H_S(P)$ to be the direct sum 
\begin{equation}\label{eqn:WH-S-defn}
W\!H_S(P):= \bigoplus_{\substack{x \in P \\ \rank(x)=\max S}} \tilde{H}_{|S|-2}([\hat{0},x]^{S\setminus\{\max S\}}).
\end{equation}
Here  $[\hat 0,x]^{S\setminus\{\max S\}}$ is the rank-selected subinterval of the interval $[\hat 0,x]$ in $P$ specified by the rank-set ${S\setminus\{\max S\}}$, noting that $x$ itself is at rank $\max S$.  In particular $[\hat 0,x]^{S\setminus\{\max S\}}$ is a subposet of $P^S$, of rank $|S|$.  
\end{defn}
For a $G$-poset $P$, as a $G$-module, $W\!H_S(P)$ is a sum of induced modules indexed by orbits in $P^S/G$. 
  The acyclicity of Whitney homology established in  \cite{Sundaram} 
   gives Proposition~\ref{prop:alt-sum-expression-S} below, establishing equivalent representation theoretic relationships for rank-selection in any Cohen-Macaulay $G$-poset $P$.  
For a subset $S$ of ranks, let 
$S^{(i) } $ denote  the set obtained from $S$ by deleting the  $i$ largest elements from $S$.  
\begin{prop}[{\cite[Lemma 1.1, Proposition 1.9]{Sundaram}}]\label{prop:alt-sum-expression-S} Let $P$ be a Cohen-Macaulay $G$-poset.  As $G$-modules, 
the rank-selected homology $\beta_S(P)$  is related to the  rank-selected Whitney homology of the poset $P^S$ 
 via  the following two identities expressing each in terms of the other:
 \begin{enumerate}
 \item
 $W\!H_S(P) = \beta_S(P) + \beta_{S^{(1)}}(P)$
 \item  
$\beta_S (P) = W\!H_S(P) - W\!H_{S^{(1)}}(P) + W\!H_{S^{(2)}}(P)  - \cdots. $
\end{enumerate} 
\end{prop}

 Recall that $\Pi_n$ is the poset of set partitions of $\{ 1,2,\dots ,n\} $ ordered by refinement.  
 A major focus of the present paper is proving the following conjecture from \cite{Hersh-Reiner}.
 \begin{conj}[{\cite[Conjecture 11.3]{Hersh-Reiner}}]
Given a subset $S\subseteq \{ 1,2,\dots ,n-2 \} $, 
the rank-selected 
homology  $\fS_n$-representation  $\beta_S (\Pi_n)$ stabilizes sharply at 
 $4 \max S - |S| +1$.
 \end{conj}
 
The special cases for the two rank sets $S=\{i\}$ and $S=\{1,\ldots,i\}$ were resolved in \cite{Hersh-Reiner}.

\subsection{Background on  representation stability}\label{bg-stability}

The following notion was introduced in a series of papers including \cite{Church-Farb}.

\begin{defn}\label{stable-def}
We say that a sequence  of $\fS_n$-modules $V_1,V_2,\dots $  {\it has stabilized} at   some $B>0$ 
if for each $n\ge B$ 
we have $V_n = \sum c_{\lambda } V(\lambda ) $ where $c_{\lambda } $ is a nonnegative integer which does not depend on $n$,  $\lambda $ is an integer partition of some integer $m\le B$ and $V(\lambda )$ is the irreducible representation  $\cS^{(n-m,\lambda )}$ of the symmetric group $\fS_n$ given by the \emph{padded} integer partition  obtained from $\lambda $   by adding to
 $\lambda $  a new largest row of length  $n-m$, for $n - m \ge \lambda_1$.
\end{defn}

If $B$  is minimal with respect to the above  property, we say the sequence $\{V_n\}$ \emph{stabilizes sharply} at $B$.  In other words, $B$ is the least index for which the decomposition above holds for $n\ge B$.

For further background on symmetric functions we refer the reader to \cite[Chapter 1]{Macdonald} or \cite[Chapter 7]{Stanley}.   
The Frobenius characteristic map,  denoted by $\ch$, 
is the isomorphism  from the ring of $\fS_n$-modules to the ring of symmetric functions sending the irreducible $\fS_n$-module  indexed by  $\lambda$ to the \emph{Schur function} $s_{\lambda }$.  In particular we have $\ch \mathbbm{1}_{\fS_n}=s_{(n)}=h_n$ and $\ch \mathrm{sgn}_{\fS_n}=s_{1^n}=e_n$, where $\mathbbm{1}$ and $\mathrm{sgn}$ are respectively the trivial and sign representations, and $h_n$ and $e_n$ are respectively the homogeneous and elementary symmetric functions of degree $n$. 

Wreath products of groups and representations arise naturally when considering group actions  on the partition lattice.  In particular, the stabilizer of the set partition consisting of $k$ blocks of size $n$ in the partition lattice $\Pi_{kn}$ is  the wreath product group $\fS_k[\fS_n]$.   If $V$ and $W$ are $\fS_k$- and $\fS_n$-modules respectively, then we denote by $V[W]$
 the corresponding representation of  $\fS_k[\fS_n]$. 

 In the ring of symmetric functions, the operation of \emph{plethysm} captures the wreath product construction \cite{Macdonald, Stanley}. With the modules $V, W$ as defined above,   the salient property of this operation for our purposes is the following fact about the Frobenius  characteristic of the $\fS_{kn}$-module obtained by inducing the module $V[W]$  from $\fS_k[\fS_n]$ up to $\fS_{kn}$. 
 
 We have
 \begin{equation}\label{eqn:plethysm-wreath-def}
 \ch V[\ch W]=\ch \Ind_{\fS_k[\fS_n]}^{\fS_{kn}} V[W].
 \end{equation}
Hemmer's result  below
is an important tool in establishing representation stability bounds 
in the case of FI-modules, 
first used in \cite{Church-Farb}, and later in a more expanded form in \cite{Hersh-Reiner}. 
\begin{lem}[{\cite[Lemma 2.3, Theorem 2.4]{Hemmer}}]
\label{Hemmer's-lemma}
The decomposition 
$$
s_\lambda h_k = \sum_{\mu} c^{\mu}_{\lambda,(k)} s_\mu
$$
stabilizes for $k \geq \lambda_1$,
in that if $\mu+1:=(\mu_1+1,\mu_2,\mu_3,\ldots)$, then
$
c^{\mu+1}_{\lambda,(k+1)}=c^{\mu}_{\lambda,(k)}
$.  
\end{lem}
We will make extensive use of the enhanced version of Hemmer's Lemma, namely Lemma 2.2 in \cite{Hersh-Reiner}, which we recall in  Section ~\ref{generalities-section} after providing further background first.  
Let $f=\sum_\lambda c_\lambda s_\lambda$ be any Schur positive symmetric function.  We say $f$ has \emph{sharp first row length  upper bound} $k$ if $c_\lambda\ne 0$ implies $\lambda_1\le k$,  and equality holds for at least one $\lambda$.
The next lemma will also be used frequently in the rest of the paper.  See  \cite[Proposition 4.3 and Theorem 4.4]{Hersh-Reiner} 
 for related ideas.  
\begin{lem}\label{plethysm-first-row-bound}The plethysm 
$s_{\lambda }[h_n]$ has sharp first row length  upper bound of $|\lambda |\cdot (n-1) + \lambda_1$.  In particular, 
$s_{\lambda}[h_2]$ has sharp first row length upper bound of $|\lambda | + \lambda_1$ and even more specifically $e_d[h_2]$ has sharp first row length  upper bound of $d+1$.
\end{lem}

\begin{proof}
By definition of the plethysm operation \cite{Macdonald}, 
 we may use the expression for $s_{\lambda }[h_n]$ as a sum over monomials obtained by taking the various semistandard fillings of the shape $\lambda $ with monomials of degree $n$. We impose a total order on all possible monomials of degree $n$  in the countable set of variables $x_1,x_2,x_3,\dots $ in such a way that $x_1^n$ is the smallest monomial of degree $n$, while monomials of the form $x_1^{n-1}x_j$ are smaller than all other monomials of degree $n$.   
 Now the highest power of $x_1$ that may appear in any such monomial  in our expression for  $s_{\lambda }[h_n]$ comes from filling each box in  the $i$-th row of  $\lambda $   with the monomial $x_1^{n-1}x_i$.  Thus, the largest possible power of $x_1$ in $s_{\lambda }[h_n]$ is $|\lambda |\cdot (n-1) + \lambda_1$. 
\end{proof}
\begin{remark}\label{ex:not-rep-stable}
Recent work of Sagan and the second author \cite{Sagan-Sundaram} provides an example of a sequence of posets $P_n$ with $\fS_n$-actions whose rank-selected homology does not stabilize.  Let $P_n=\Omega_n$ be
 the face lattice of  the permutahedron (a convex polytope).
 Let $T$ be a subset  of the ranks $\{1,2,\ldots, n-1\}$  in $\Omega_n$, and consider the corresponding rank-selected subposet of $\Omega_n$. Theorem~5.1 of  \cite{Sagan-Sundaram} asserts that the multiplicity of the trivial representation in the rank-selected homology is given by 
\[
b_n(T):=|\{\sigma\in \mathfrak{S}_{n-1} \mid  \sigma(i)>\sigma(i+1) \iff i\in T\}|.
\]
For fixed $T$ not containing $n-1$, the multiplicities  $b_n(T)$  
increase strictly as $n$ increases, showing that the homology modules do not stabilize.
\end{remark}

The notion of representation stability was enriched  with further structure  in \cite{CEF}  with the introduction of  FI-modules, 
briefly described next.  Some readers may find this viewpoint helpful for understanding our results.  
The category FI is the category whose objects are finite sets ${\bf n} := \{ 1,2,\dots ,n\} $ and whose morphisms are injections.  An \emph{FI-module} over a commutative ring $k$ is a functor $V$ from FI to the category of modules over $k$.  The $k$-module $V(\bf{n})$ is typically denoted by $V_n$.  Thus, the symmetric group $\fS_n$ acts on $V_n$ by virtue of acting on 
${\bf n}$. 
An FI-module has built into it an $\fS_n$-module $V_n$  for each $n$, 
and a set of maps $V(i_j)$  from $V_n $ to 
$V_{n+1} $ indexed by the injections  $i_j : {\bf n} \rightarrow {\bf n+1}$, such that for each $\sigma  \in \fS_{n+1}$ and $\tau \in \fS_n$ 
satisfying $\sigma \circ i_j = i_k \circ \tau $ for $i_j,i_k$ injections from $V_n$ to $V_{n+1}$, we have  
$V(\sigma )\circ V(i_j) = V(i_k)\circ V(\tau )$; in other words, $V$ is functorial. 
It is proven in  \cite[Theorem 1.13]{CEF} that an FI-module (over a field of characteristic 0)  is finitely generated if and only if the associated sequence of $\fS_n$-modules is 
uniformly representation stable (an especially strong form of representation stability defined in Definition 2.6 in  \cite{Church-Farb} and described next). 
In what follows, we call the maps $\phi_n:V_n\rightarrow V_{n+1}$ \emph{consistent} if $\sigma \circ \phi_n = \phi_n \circ \iota (\sigma) $ for each $n$ and each $\sigma \in \fS_n$ where $\iota :\{ 1,2,\dots ,n\} \rightarrow \{ 1,2,\dots ,n+1\} $ sends each element of $\{ 1,2,\dots ,n\} $ to itself.

\begin{defn}\label{uniform-rep-stability}
    Let $\{ V_n,\phi_n \} $ be a consistent sequence of $\fS_n$-representations.
    The sequence $\{ V_n,\phi_n \} $ is \emph{uniformly representation stable with stable range $n\ge N$} if the following conditions all hold:
    \begin{enumerate}
        \item The map $\phi_n:V_n\rightarrow V_{n+1}$ is injective for all $n\ge N$.
        \item The span of the $\fS_{n+1}$-orbit of $\phi_n(V_n)$ equals all of $V_{n+1}$ for all $n\ge N$.
        \item The sequence $V_1,V_2,\dots $ stabilizes at $N$, in the sense of Definition ~\ref{stable-def}.
    \end{enumerate}
\end{defn}

Next we recall a weaker notion of stability that had previously been established for the rank-selected homology of the Boolean lattice, along with an example meant to illustrate how this notion of stability is significantly more relaxed than what is traditionally known as 
representation stability.

\begin{defn}[{\cite[Definition 2.7 and Definition 2.3 (III)]{Church-Farb}}]\label{def-of-mult-stability}
A sequence $M_1,M_2,\dots $ of symmetric group representations, in which $M_i$ is a representation of $\fS_i$, is {\em multiplicity stable} if for each $\cS^{\lambda }$  appearing in any $M_i$,  there is some bound $N(\lambda )$ such that the coefficient of $\cS^{(n-|\lambda |, \lambda )}$  
within $M_n$  takes the same value
for all $n\ge N(\lambda )$.    
\end{defn}

\begin{ex}\label{example-of-mult-stability}
Let $M_n = \sum_{\lambda \vdash n} c_{\lambda }\cS^{\lambda }$ 
where  $c_{\lambda }$ 
is any nonnegative integer valued function of $\lambda $  that only depends on the number of parts in $\lambda $.    For instance, one could take $c_{\lambda } = 1$ for all $\lambda $, or one could take $c_{\lambda } =  k$ 
for all $\lambda = (\lambda_1,\dots ,\lambda_k ) $.  Then $M_1,M_2,\dots $ is multiplicity stable even though in both these specific examples   $M_1,M_2,\dots $ is not representation stable.  
For each $n\ge 1$, note that $\fS_n$ has an irreducible representation with $n$ parts while $\fS_1,\fS_2,\dots ,\fS_{n-1}$ do not, making it impossible to  predict what $M_n$ will be from knowledge of $M_1,M_2,\dots ,M_{n-1}$ by virtue of $M_1,\dots ,M_{n-1}$ not containing any information regarding the coefficient  $c_{(1,1,\dots ,1)}$ needed to determine $M_n$.     
\end{ex}

\subsection{Background on Young symmetrizers, Specht modules, and tableaux combinatorics}
\label{subsec:YoungSym}

In this section we describe the background that we need on Young symmetrizers, loosely following \cite[Chapter 7]{Fulton} and \cite{James}.  A concise description of polytabloids and Specht modules associated to ribbons, specifically in connection to the topology of the order complex of the Boolean lattice, can be found in \cite[\S 2.2-2.3, \S 3.4]{WachsPosetTop2007}.

Young symmetrizers are elements in the group algebra of the symmetric group, defined  as follows.  Let $\lambda$ be an integer partition of $n$. A \emph{$\lambda$-tableau} $T$ is  an arbitrary bijective  filling of the Ferrers diagram of $\lambda$ with the integers $\{1,2,\ldots,n\}$.   
Define the row stabilizer $\mathrm{Row}_T$ (respectively, column stabilizer $\mathrm{Col}_T$)  of $T$ to be the subgroup of $\fS_n$ consisting of those permutations which permute the entries of each row (respectively, column) among themselves.  

Now define, in the group algebra,  the \emph{row symmetrizer} of a $\lambda$-tableau $T$ to be the  element 
\[a_T:= \sum_{p\in \mathrm{Row}_T} p\]
and the \emph{column symmetrizer} of $T$ to be
\[b_T:= \sum_{q\in \mathrm{Col}_T} \mathrm{sgn}(q) q.\]
Finally define the \emph{Young symmetrizer} associated to the $\lambda$-tableau $T$ to be 
\[c_T:= b_T a_T .\]
The  Young symmetrizer $c_T$ is, up to a scalar factor, a primitive idempotent in the group algebra.  It plays  an important role in the definition of the \emph{Specht module}, an explicit construction of the irreducible representations of the symmetric group in the group algebra. 
A \emph{standard Young tableau} of shape $\lambda$ is a $\lambda$-tableau $T$ where each row increases left to right and each column increases top to bottom. For each standard Young tableau $T$ of shape $\lambda$, the left ideal in the group algebra of $\fS_n$ generated by $c_T$ is an irreducible $\fS_n$-module; this is the Specht module indexed by $\lambda$. Moreover, a basis for the Specht module $\mathcal{S}^\lambda$ is given by the standard Young tableaux of shape $\lambda$. See \cite[p.17]{James}.  

The proof of our stability bound rests on showing that certain irreducible representations of $\fS_n$ do not appear in the rank-selected Whitney homology of $\Pi_n$. Our strategy is to use Young symmetrizers to verify this.  The underlying principle comes from the theory of semisimple algebras (over a field of characteristic zero), and the Wedderburn theorem; see  \cite[Chapter 18, Theorem 4 and Proposition 8]{Dummit-Foote}. Since this detection principle is an important feature of our arguments, we begin by stating it precisely for the symmetric group $\fS_n$. 
For a standard Young tableau $T$ of shape $\lambda$, let $c_T=b_Ta_T$ be the  Young symmetrizer indexed by $T$, as defined above.  The next result essentially says that the Young symmetrizer $c_T$ acts on any $\fS_n$-module as a projection onto an isomorphic copy of $\mathcal{S}^\lambda$ in the module, if such a copy exists.

\begin{thm}\label{thm:irrep-detection}  Let $V$ be an $\fS_n$-module.  
Then $V$ contains a submodule isomorphic to $\mathcal{S}^\lambda$ if and only if there is a standard Young tableau $T$ of shape $\lambda$ and a vector $v\in V$ such that $c_T v\ne 0$.

Equivalently, the irreducible indexed by $\lambda$ does not appear in $V$ if and only if $c_T v= 0$ for every standard Young tableau $T$ of shape $\lambda$ and for every $v\in V$.  
\end{thm}

In this paper we make frequent use of this detection principle by adapting  
the following lemma,  extracted from the proof of \cite[\S 7.2, Lemma 2]{Fulton}.  
\begin{lem}\label{lem:Youngsym-annihil-Fulton}
    Let $T$ and $T'$ be  Young tableaux. Suppose there exist two entries  $i,j$ having the property that 
$i,j$ appear  in the same row of $T$, as well as in the same column of $T'$. Then  $a_T b_{T'}=0$.
\end{lem}

For basic combinatorial facts about partitions in the context of the representation theory of $\fS_n$, we refer the reader to \cite{Macdonald, Stanley}.  Given two partitions $\lambda, \mu$ such that $\mu\subseteq \lambda$, the \emph{skew shape} $\lambda/\mu$ has Ferrers diagram consisting of those boxes in $\lambda$ that are NOT in $\mu$. We say the skew-shape is \emph{connected} if, when  the boxes are numbered from left to right and bottom to top,  any two consecutive boxes share an edge.

A \emph{ribbon} (also called a border strip in \cite{Macdonald, Stanley}, and a skew hook in \cite{WachsPosetTop2007}), is a connected skew shape 
that does not contain a 2 by 2 square of four boxes. 
We specify a ribbon  
 by its sequence of row lengths  $(r_1, r_2, \ldots, r_k)$ from bottom to top, and we denote this by $\Rib(r_1, r_2, \ldots, r_k)$. 
 See Figure~\ref{fig:skew-shape} for an example of an arbitrary skew shape (left) and a ribbon (right).  Note that the ribbon $\Rib(1,1,\ldots, 1, n-i)$, having all rows of length 1 except possibly the top row, coincides with the partition shape $(n-i, 1^i)$.

\begin{figure}
\ytableausetup{smalltableaux, boxsize=1.1em}
\begin{center}
\begin{ytableau}
\none &\none  & & &\\
\none &\none  & &\\
\none & &\\
& &
\end{ytableau}
\qquad\qquad \qquad 
\ribbon{0/0, 1/0, 2/0, 2/1,2/2, 3/2, 4/2, 5/2, 5/3, 6/3}

\end{center}
\ytableausetup{nosmalltableaux}
\caption{Left to right: the connected skew shape $(5,4,3,3)/(2,2,1)$ and the ribbon $\Rib(3,1,4,2)$}
\label{fig:skew-shape}
\end{figure}

Another way to define the Specht module corresponding to the partition $\lambda$ is in terms of \emph{polytabloids} \cite[Chapter 4]{James}.  For a fixed partition $\lambda$ of $n$, define an equivalence relation on the set of tableaux of shape $\lambda$ by $T_1 \equiv T_2$ if and only if $T_2=\sigma(T_1)$ for some $\sigma\in R(T)$. The \emph{tabloid} $\{T\}$  is then the equivalence class of the 
 tableau $T$; we sometimes refer to a tabloid as the tableau obtained from $T$ by ``forgetting" the order of the entries in each row.  The vector space $M_\lambda$ of tabloids of shape $\lambda$ is an $\fS_n$-module under the action defined by 
 \begin{equation}\label{eqn:fSn-action-on-tabloids} \sigma\cdot \{T\}= \{\sigma \cdot T\}
 ;
 \end{equation}
 it is  isomorphic to the induced representation $\Ind\, \one_{\fS_{\lambda}}^{\fS_n}$, where $\fS_\lambda$ is the Young subgroup indexed by the partition  $\lambda$. The \emph{polytabloid} $v_T$ is then defined to be the linear combination of tabloids obtained by applying the column symmetrizer $b_T$ to the tabloid $\{T\}$ (see \cite[p. 86, Eqn (4)]{Fulton}, \cite[p. 13, Definition 4.3]{James}):
 \begin{equation}\label{eqn:polytabloid}v_T:=\{b_T \cdot T\}=\sum_{\sigma\in \mathrm{Col}(T)} \sgn(\sigma) \{\sigma \cdot T\}.
 \end{equation}
 Thus $v_T$ is the result of  
 permuting the elements in the columns of $T$, and attaching the sign of the permutation at each step.
 The Specht module $\cS^\lambda$ is then defined to be the submodule of $M^\lambda$ spanned by the polytabloids $v_T$ as $T$ ranges over all tableaux of shape $\lambda$.  As an $\fS_n$-module,  $\cS^\lambda$ is in fact cyclically generated by $v_T$ for any one tableau $T$, which we may take to be a standard Young tableau.

 The polytabloid definition in~\eqref{eqn:polytabloid}  generalizes in the obvious way to skew shapes $\lambda/\mu$, using the analogous definitions for row,  column and Young symmetrizers and stabilizers.  The culmination is the definition of the skew Specht module $\cS^{\lambda/\mu}$, again generated by $v_T$ for a single tableau $T$ of shape $\lambda/\mu$, which we may take to be a \emph{standard} tableau, i.e., \emph{row strict}, with rows increasing strictly  left to right and \emph{column strict}, with columns increasing strictly top to bottom.  

 We will often refer to a filling of a Young diagram as a \emph{standard filling} if the filling is both row  strict and column strict.

\section{Generalities regarding  stability bounds and sharpness}\label{generalities-section}

First we establish a simple lemma. 
\begin{lem}\label{lem:weak-strict-stab-for-sums} Suppose $\{U_n\},\{V_n\}, \{W_n\}$ are sequences of $\fS_n$-modules such that $U_n=V_n\oplus W_n$.  
\begin{enumerate}
\item[(a)] \cite[Lemma 3.1]{MMPR2023} 
If any two of the three modules $\{U_n\},\{V_n\}, \{W_n\}$ stabilize for  $n\ge m$, then so does the third.
\item[(b)]
 Assume $\{W_n\}$ stabilizes sharply at or before $d-1$ for some $d\ge 2$. Then the direct sum 
$U_n$ stabilizes sharply at $n=d$ if and only if the summand $V_n$ stabilizes sharply at $n=d$. 
\end{enumerate}
\end{lem}

\begin{proof} Item (a) is clear, so we  only prove Item (b). 
We have \[U_n=V_n\oplus W_n,\]
with $\{W_n\}$ stabilizing at or before $d-1$. 

Suppose $U_n$ stabilizes sharply at $n=d$. Then by Part (a),  $V_n$ stabilizes for all $n
\ge d$, since both $U_n$ and $W_n$ do.  However, $V_n$ cannot stabilize sooner, otherwise the direct sum $U_n$ would stabilize sharply at $n= d-1$ or earlier, because of $W_n$, and this contradicts our hypothesis.

Conversely suppose the summand $V_n$ stabilizes sharply at $n=d$. Then again by Part (a),  the sum $U_n$ stabilizes for all  $n\ge d$. As before, if $U_n$ stabilized sharply at or before $d-1$, this would force $V_n=U_n-W_n$ to stabilize at or before $d-1$, a contradiction.
\end{proof}

\begin{prop}\label{prop:WHS-bound-equals-betaS-bound-anyP-AGAIN} 
Let $\{P_n\}$ be a sequence of Cohen-Macaulay $\fS_n$-posets, $n\ge 1$. 
For a fixed subset $S$ of positive integers 
consider the sequences of $\fS_n$-modules $\{\beta_S(P_n)\}$ and $\{W\!H_S(P_n)\}$.  Let $f$ be a function defined on all subsets $S$, with $f(S)$  nonnegative   for all $S$ and $f(\emptyset)=1$.
\begin{enumerate}
    \item Assume the function $f$ satisfies $f(S\setminus\{\max S\})\le f(S)$ for all $S\ne\emptyset$. 
    Then  the Whitney homology module 
$W\!H_S(P_n)$ stabilizes when $n\ge f(S)$ 
 for all  $S$
 if and only if the rank-selected homology module $\beta_S(P_n)$ stabilizes for  $n\ge f(S)$ for all $S$. 
     \item Assume the function $f$ satisfies $f(S\setminus\{\max S\}) \le f(S)-1$ for all $S\ne\emptyset$. 
     Then  the Whitney homology module 
$W\!H_S(P_n)$ stabilizes sharply at $n=f(S)$ for all subsets $S$  if and only if the rank-selected homology module $\beta_S(P_n)$ stabilizes sharply at $n=f(S)$ for all $S$.  
In particular, for fixed  $k\ge 2$, the function $f(S)=k \max S -|S|+1$ 
 satisfies the above  inequality. 
\end{enumerate}
\end{prop}
\begin{proof}  
We will use Proposition~\ref{prop:alt-sum-expression-S} and Lemma~\ref{lem:weak-strict-stab-for-sums} to prove the stated  equivalence.
  
Recall that $S^{(i)}$ is defined to be the rank set resulting from removing the $i$ largest elements in $S$.  We make two elementary observations:
\begin{itemize}
    \item 
  $\max S^{(i)}  \le \max S-i$, and 
 \item  $|S^{(i)}|=|S|-i$. 
\end{itemize}
In Case (a) this implies the  inequality $f(S^{(i)}) \le f(S)$, whereas in Case (b) we obtain  the inequality $f(S^{(i)}) \le f(S)-i\le f(S)-1$ for $i\ge 1$.  

 Item (1) of Proposition~\ref{prop:alt-sum-expression-S} asserts that 
\begin{equation*}\label{eqn:Whit-beta2}W\!H_S(P_n)=\beta_S(P_n) \oplus \beta_{S^{(1)}}(P_n).
\end{equation*}

The claim in Case (a) follows immediately from the observation of Part (a) of Lemma~\ref{lem:weak-strict-stab-for-sums}.  

For Case (b), we use Part (b) of Lemma~\ref{lem:weak-strict-stab-for-sums}.  The key point is that the stability bound $f(S^{(1)})$ is now  strictly  smaller than $f(S)$.  Aassume we have the sharp bound for $\beta_S(P_n)$ for all $S$.  Then the claim follows from Part (b) of Lemma~\ref{lem:weak-strict-stab-for-sums}, taking $W_n=\beta_{S^{(1)}}(P_n)$,

For the converse in Case (b), assume we have the sharp  bound for $W\!H_S(P_n)$ for all $S$.  By induction on $\max S$, we may assume  that the sharp bound holds for $\beta_{(S^{(1)}}(P_n)$.  Then the claim follows exactly as before.
\end{proof}
\begin{cor}\label{cor:chain-module-Stab} Let $\{P_n\}$ be a sequence of Cohen-Macaulay $\fS_n$-posets, $n\ge 1$.  Let $f(S)$ be a nonnegative integer-valued  function such that $f(\emptyset)=1$ and for nonempty $S$, $f(S)\le f(\{\max S\})$ with equality if and only if $S$ equals the singleton set $\{\max S\}$.   

Then the $\fS_n$-module of chains $\alpha_S(P_n)$ stabilizes sharply at $f(\{\max S\})$ if either 

$\bullet$ $\beta_S(P_n)$ stabilizes sharply at $f(\{\max S\})$ for all $S$, 

or if 

$\bullet$ $W\!H_S(P_n)$ stabilizes sharply at $f(\{\max S\})$ for all $S$.

In particular, for fixed $k\ge 2$, the function $f(S)=k \max S -|S|+1$ 
 satisfies the above conditions.
\end{cor}
\begin{proof} Assume the sharp stability bound $f(S)$ holds for $\beta_S(P_n)$ for every  $S$.  Then the conclusion follows from the general formula~\eqref{alpha-definition}, applied to the sequence of posets $\{P_n\}$:
\[\alpha_S(P_n)=\sum_{T\subseteq S} \beta_T(P_n).\]
Each $\fS_n$-module $\beta_T(P_n)$, $ \emptyset\ne T\subset S$, $T\ne \{\max S\}$ on the right  stabilizes strictly before $f(\{\max S\})$, from the inequality $f(T)\le f(\{\max T\})\le f(\{\max S\})$  for  $T\subseteq S$ and the fact that $f(T)= f(\{\max S\})$ if and only if $T=\{\max S\}$.     
Also $\beta_{\{\max S\}}(P_n)$ stabilizes sharply at $f(\{\max S\})$.
Thus we can write $\alpha_S(P_n)=W_n\oplus \beta_{\{\max S\}}(P_n)$ where $W_n$ stabilizes strictly before $\beta_{\{\max S\}}(P_n)$.
The claim now follows from Part~(b) of Lemma~\ref{lem:weak-strict-stab-for-sums}.

Now assume  the sharp stability bound $f(S)$ holds for $W\!H_S(P_n)$ for every  $S$.   It is easy to see that Item (1) of Proposition~\ref{prop:alt-sum-expression-S} implies the expression below for $\alpha_S(P_n)$ as a sum of Whitney homologies:
\begin{equation}\label{eqn:chain-is-Whitney-sum}
\alpha_S(P_n)= \bigoplus_{\substack{T\subseteq S \\ \max S\in T}} W\!H_T(P_n).
\end{equation}  
For this it suffices to note that every subset $T$ of $S$ containing $\max S$ can be paired up with the subset $T\setminus\{\max S\}$, and the direct sum of the latter two rank-selected homoloy modules is precisely $W\!H_T(P_n)$.
Since $f(T)$ is strictly less than $f(\{\max S\})$ for every $T$ in the sum above except for the singleton set $T=\{\max S\}$, the conclusion follows as before from Part~(b) of Lemma~\ref{lem:weak-strict-stab-for-sums}.
\end{proof}

We will apply a specific choice of the function $f(S)$ to obtain sharp stability bounds, namely,   for the Boolean lattice with $f(S)=2\max S-|S|+1$, and for the partition lattice with  $f(S)=4\max S-|S|+1$.

Both in \cite{Hersh-Reiner} and in our present work, representation stability in Whitney homology takes an especially nice form.  This leads us to establish the following convenient definition  (stated first informally and then more formally) in order to better describe this structure.  While we did not  find this definition explicitly  in the literature, it is certainly implicit in \cite{CEF}.  
 We refer readers specifically  to the section on free FI-modules in \cite{CEF} where closely related  ideas are discussed in a much more formal  language.

Recall first  that Proposition 2.6 in \cite{CEF} introduces the generators of an FI-module $V$, leading to the notation $\{ M_n(S^{\lambda })\} $ for the generators described  
below.   
As a word of caution, these generators
$\{ M_n(S^{\lambda })\} $ are quite different from the $V(\lambda )$'s appearing earlier in  Definition ~\ref{stable-def} when we introduced the notion of representation stability.  

\begin{defn}[{\bf Informal Version}] 
Consider a finitely generated  FI-module $V$.  
    Suppose  there is  a finite list (with repetition allowed) of irreducible representations $\cS^{\lambda^{(1)} },\cS^{\lambda^{(2)} },\dots ,\cS^{\lambda^{(k)}}$ 
    that will serve as a complete set of  FI-module generators for $V$.  
    Suppose additionally  
    that each of these irreducible representations  
$\cS^{\lambda^{(i)}}$ appears in 
$V_{|\lambda^{(i)}|}$ and contributes
$$\Ind_{\fS_{|\lambda^{(i)}|}\times \fS_{n-|\lambda^{(i)}|}}^{\fS_n}\, (\cS^{\lambda^{(i)}}\otimes \one_{\fS_{n-|\lambda^{(i)}|}} )$$ to $V_n$ for each $n\ge |\lambda^{(i)}|$.  Suppose also that each $V_n$ is exactly the direct sum of these contributions to it.    In this case, we say that $V$ is
{\em quasi-freely generated} by $\cS^{\lambda^{(1)}}, \cS^{\lambda^{(2)}},\dots ,\cS^{\lambda^{(k)}}$. 

The modules $\{ \cS^{\lambda^{(i)}} : i=1,2,\dots ,k\} $ form a complete set of \emph{FI-module generators}, a notion defined more generally in \cite{CEF}. 

Denote by $\{ M_n(\cS^{\lambda^{(i)}} )\} $ the FI-module appearing as
the direct summand above given by $\cS^{\lambda^{(i)} }$. 
\end{defn}

In the language of \cite{CEF}, this FI-module is generated in degree $\max (|\lambda^{(i)}|)$.  
Typically the degree in which an FI-module is generated is strictly lower than the degree in which stability first occurs, discussed shortly.   Now we define quasi-free generation more formally.

Fix an integer $n_0\ge 1$. For any $\fS_{n_0}$-module $W$, define a sequence of $\fS_n$-modules $\{M_n(W)\}$ as follows.
\[ M_n(W)=\begin{cases} \Ind_{\fS_{n_0}\times \fS_{n-n_0}}^{\fS_n}\,(W\otimes \mathbbm{1}_{n-n_0}), & n\ge n_0\\
0 & \text{otherwise}.
\end{cases}
\]
Note that  $\{ M_n(W)\} $ is an FI-module.  
Since direct sums of FI-modules are FI-modules, one may also speak of 
$\{ M_n(
\oplus_{\lambda }c_{\lambda } \cS^{\lambda })\} $ as an FI-module  where the sum can be over partitions of more than one positive integer.

In this case, we call the module $W$
the {\em essential part} of $M_n(W)$.  More generally, the module $\oplus_{\lambda }c_\lambda \cS^{\lambda }$ is defined to be the {\em essential part} of $M_n(\oplus_{\lambda }c_{\lambda } \cS^{\lambda })$.  

For instance, the essential part of  
$\{ M_n (5\,\cS^{(3,1)}\oplus 7\cS^{(4,2)})\} $ is  $5\,\cS^{(3,1)}\oplus  7\cS^{(4,2)}$.

\begin{defn}\label{quasi-free}
Call an FI-module that is a finite direct sum of FI-modules of the form $\{ M_n(\cS^{\lambda })\} $ a \emph{ quasi-freely generated}  FI-module.
\end{defn}

It makes sense to speak of the essential part of any quasi-freely generated FI-module.

\begin{rk}
    We make the convention in Definition ~\ref{quasi-free} that quasi-freely generated FI-modules are necessarily finitely generated.  One need not require finite generation, but  we make this choice   to avoid having to repeatedly say ``quasi-freely generated, finitely generated".
\end{rk}

Next we recall and slightly reframe results we will need for proving our representation stability bounds and also for proving that they are sharp.  The key tool will be Lemma 2.2 of \cite{Hersh-Reiner}.  This is an enhancement of Lemma~\ref{Hemmer's-lemma}.  We will rephrase  Lemma 2.2  
from \cite{Hersh-Reiner} in the language of FI-modules
in order to give  some important consequences
that we will also need.

\begin{lem}[{\cite[Lemma 2.2]{Hersh-Reiner}}] 
\label{lem:quasi-free-implies-increasing-mult}
The following statements hold.
\begin{enumerate}
\item (Injectivity property)
For any partition $\mu$, let $\mu^+$ denote the partition $(\mu_1+1, \mu_2,\ldots)$, and let $(\oplus_{\mu\vdash n} \cS^{\mu })^{(+1)}$ denote $\oplus_{\mu\vdash n} \cS^{\mu^+}$. Then the sequence $\{M_n(\cS^\mu)\}$ has the following property:  \[M_n(\cS^\mu)^{(+1)} \text{ is an $\fS_{n+1}$-submodule of } M_{n+1}(\cS^{\mu }).\]   
\item (Sharp stability bound) For any 
true (i.e.\! non-virtual) symmetric group module $W=\oplus_\mu c_\mu \cS^\mu$, 
the sequence $\{M_n(W)\}$ stabilizes sharply at $n=\max\{|\mu|+\mu_1: c_\mu \ne 0\}$.

\item (Sharpness for irreducibles) In particular when $W$ is a single irreducible $\cS^\mu$,  the sequence $\{M_n(\cS^\mu)\}$ stabilizes sharply at $n=|\mu|+\mu_1$.  
\end{enumerate}
\end{lem}

In the language of FI-modules, the results above imply the following.

\begin{prop}\label{prop:quasi-freely-gen-immplies-sharp-max}
If $V$ is a quasi-freely generated FI-module, then $V$ stabilizes sharply at $\max \{ |\lambda |+\lambda_1 : \cS^{\lambda } \text{ is a generator for its quasi-free structure} \} $.
\end{prop}

\begin{proof}
    This is essentially a restatement of Lemma~\ref{lem:quasi-free-implies-increasing-mult}, part (2).  It follows from 
the  monotonicity property of the coefficients that is 
guaranteed to hold within any direct summand $\{ M_n(S^{\mu }) \} $  by Lemma~\ref{lem:quasi-free-implies-increasing-mult}, Part (1)).
\end{proof}

The monotonicity guaranteed for $\{ M_n(\cS^{\lambda })\} $ by part (1) of 
Lemma   ~\ref{lem:quasi-free-implies-increasing-mult} shows that examples such as the following cannot occur in the context of  quasi-freely generated FI-modules.

\begin{ex}\label{ex:even-odd-problem-stab}
Consider the sequence of $\fS_n$-modules defined by $U_n=3\,\mathbbm{1}_{\fS_n}$ for $n\ge 1$. Thus each $U_n$ is simply three copies of the trivial representation.
Define $\fS_n$-modules 
\[V_n= \begin{cases} \mathbbm{1}_{\fS_n}, & 
n=1,3\\
2\, \mathbbm{1}_{\fS_n}, &
n= 2,4\\
\mathbbm{1}_{\fS_n}, & n\ge 5.
\end{cases}
\qquad \text{ and } \qquad 
W_n= \begin{cases} 2\,\mathbbm{1}_{\fS_n}, & 
n=1,3\\
 \mathbbm{1}_{\fS_n}, &
n= 2,4\\
 2\, \mathbbm{1}_{\fS_n}, & n\ge 5.
\end{cases}
\]
 Then clearly $U_n=V_n\oplus W_n$. Also the sequences $\{V_n\}$, $\{W_n\}$   stabilize sharply  at $n=5$, but their sum $\{U_n\}$ stabilizes sharply earlier, at $n=1$.
\end{ex}
 
We will also make substantial use of the next result.
It  is an 
immediate consequence of  Item (1) of Lemma~\ref{lem:quasi-free-implies-increasing-mult}. 

\begin{cor}\label{cor:FI-submodules}  Let 
$\{V_n\}, \{W_n\}$ be  quasi-freely generated FI-modules. Then the FI-module $\{V_n\oplus W_n\}$ cannot stabilize earlier than either $\{V_n\}$ or 
$\{W_n\}$.
\end{cor}

As our final topic in this section of preliminaries, we verify that the two main examples considered  in this  paper are both  FI-modules and both   satisfy the first two requirements for uniform representation stability.  We will verify the third requirement for uniform representation stability later in the paper.  

\begin{prop}\label{get-FI-modules}
If a family $\{ P_i\} $ of posets has $P_i=B_i$ for all $i$ or  $P_i=\Pi_i$ for all $i$ then both $\{ \beta_S(P_n)\} $ and $\{ W\!H_S(P_n) \} $ are FI-modules. 
\end{prop}

\begin{proof}
First we check that any injection $i$ of sets from $[m]$ to $[n]$  induces an isomorphism from $P_m$ to a lower interval $[\hat{0},u]$ of rank ${\rm rk}(P_m)$  in $P_n$ in each case under consideration.  
For the Boolean lattices, $u$ is the subset of $\{1,2,\dots ,n\} $ 
consisting of $\{ i(1),i(2),\dots ,i(m)\}$.  For the partition lattices, $u$ is the set partition having the letters $\{ i(1),i(2),\dots ,i(m)\} $ in a single block and all other letters in singleton blocks.  
For instance, $i: \{ 1,2\} \rightarrow \{ 1,2,3,4 \} $ with $i(1)=3$ and $i(2)=2$ sends $B_2$ to the interval $[\hat{0},\{ 3,2\} ]$ in $B_4$ and sends $\Pi_2 $ to the interval $[\hat{0}, 32|1|4]$ in $\Pi_4$.  Particularly important is the observation that these injections preserve poset rank, and hence send maximal chains in a rank-selected subposet $P_m^S$ to maximal chains in $P_n^S$.

It is straightforward to confirm that  every homology cycle in $P_m^S$ 
is mapped by each such injection of posets to a homology cycle in $P_n^S$;
since the homology of an open interval in each of these posets (and in each rank-selected subposet) is concentrated in top degree, by virtue of shellability of these posets,  none of the cycles can be boundaries.  
This 
injective mapping of cycles to cycles gives maps on homology that satisfy the necessary relations to be functorial by virtue of how they are defined.  In this manner, one gets 
the maps needed here for rank-selected homology. 
Similarly,  one may also deduce that rank-selected Whitney homology for a fixed rank set $S$ for either family of posets is also  an FI-module.
\end{proof}

Let $\phi_n :P_n\rightarrow P_{n+1}$ denote  the special case of the injection in the proof of Proposition ~\ref{get-FI-modules}, that is induced by the injection $i:\{ 1,2,\dots ,n\} \rightarrow \{ 1,2, \dots ,n+1\} $ having $i(j)=j$ for $j=1,2,\dots ,n$.  
Using this map $\phi_n$, the next result follows directly from the proof of Proposition ~\ref{get-FI-modules}.  The bounds appearing in the next corollary 
come from the necessary condition that the rank of $P_n$ must be at least 
$\max S + 1$ for $\beta_S(P_n)$ to be well-defined.

\begin{cor}\label{first-uniform}
The first requirement  for uniform representation stability (see Definition ~\ref{uniform-rep-stability}) is  satisfied by $\{ W\!H_S(P_n), \phi_n \} $ and by $\{ \beta_S(P_n),\phi_n\} $ for $P_n=B_n$ when $n\ge \max S+1$ and for $P_n=\Pi_n$ when $n\ge \max S+2$.  
\end{cor}

With a little more work, we can also deduce the following from the proof of Proposition ~\ref{get-FI-modules}. 
In what follows,  $\phi_n$ is induced by an arbitrary injection $\iota:\{1,2,\ldots,n\} \rightarrow \{1,2,\ldots,n+1\}$.
\begin{cor}\label{second-uniform}
The second requirement for uniform representation stability is satisfied by 
$\{ W\!H_S(P_n),\phi_n\} $ for $P_n=B_n$ when $n\ge \max S$, and for  $P_n=\Pi_n$ when  
$n\ge 2\max S$.
\end{cor}

\begin{proof}
  Recall the definition of the rank-selected Whitney homology $W\!H_S(P_n)$ as a direct sum \[\bigoplus_{x\in P_n} \tilde{H}_{|S|-2}([\hat 0,x]^{S\setminus \{\max S\}})\]
  over elements $x\in P_n$ of rank $\max S$.
  Note that each $x\in P_n$ of rank $\max S$  has at most $2\max S$ letters in nontrivial blocks of the set partition $x$ when $P_n=\Pi_n$ and at most $\max S$ letters in the set $x$ when $P_n= B_n$, regardless of how large $n$ is.  This implies that each $x\in P_{i+1}$ of rank $\max S$ may be obtained as $\pi (\phi_i (u))$ for some $u\in P_i$ and some $\pi \in \fS_{i+1}$ provided we have $i\ge 2\max S$ in the case that $P_i=\Pi_i$ or we have $i\ge \max S$ in the case that $P_i = B_i$.  
  
  Notice that this implies the stronger statement that each such interval $[\hat{0},x]$ for $P_{i+1}= B_{i+1}$ with $i\ge \max S$ or for $P_{i+1} =\Pi_{i+1}$ with $i\ge 2\max S$ is isomorphic to  an interval $[\hat{0},\pi (\phi_i(u))]$ for this same  $u\in P_i$  and $\pi \in \fS_{i+1}$.  This ensures that every element of $W\!H_S(P_n)$ is a direct sum of elements of the $\fS_{i+1}$-orbit of $\phi_i(W\!H_S(P_i))$ for $i\ge 2\max S$ when $P_i=\Pi_i$ and for $i\ge \max S$ when $P_i=B_i$.  On the other hand, $\phi_i(W\!H_S(P_i))$ is a subspace of the vector space $W\!H_S(P_{i+1})$ and  is sent to a subspace of $W\!H_S(P_{i+1})$ by $\fS_{i+1}$, 
  implying that $\sigma \phi_i(W\!H_S(P_i))$ is contained in $W\!H_S(P_{i+1})$ for each $\sigma \in \fS_{i+1}$ and hence 
  the span of the  $\fS_{i+1}$-orbit of $\phi_i(W\!H_S(P_i))$ is contained in  $W\!H_S(P_{i+1})$ for all $i\ge 2\max S$ when $P_i=\Pi_i$ and for all $i\ge \max S $ when $P_i=B_i$. Since we already showed that $W\!H_{i+1}$ is contained in the span of the $\fS_{i+1}$-orbit of $\phi_i(W\!H_S(P_i))$, the result follows.  
\end{proof}

We use now use the fact that the second condition for uniform representation stability holds for Whitney homology to deduce the much more subtle result that it also holds for rank-selected homology.  

\begin{prop}\label{second-beta}
    The second condition for uniform representation stability is satisfied for $\beta_S(B_n)$ and $\beta_S(\Pi_n)$.
\end{prop}

\begin{proof}
 We start with the fact from Proposition~\ref{prop:alt-sum-expression-S}  that $$W\!H_S(P_n) \cong \beta_S(P_n) \oplus \beta_{S\setminus \{ \max S\} } (P_n).$$ 
 Underlying this is the vector space isomorphism 
 \begin{equation}\label{eqn:unif-stab2} W\!H_S(P_n) \cong \tilde{H}_{|S|-1}(P_{n}^S)\oplus \tilde{H}_{|S|-2}(P_{n}^{S\setminus \max S}).
 \end{equation}
We apply $\phi_n$ to each part of this statement about vector spaces, and then take the span of the $\fS_{n+1}$-orbit of each part of~\eqref{eqn:unif-stab2}.   This expresses  the  span of the $\fS_{n+1}$-orbit of 
$\phi_n(W\!H_S(P_n))$ as the direct sum of the span of the $\fS_{n+1}$-orbit of $\phi_n(\tilde{H}_{|S|-1|}(P_n^S))$ and  the span of the $\fS_{n+1}$-orbit of $\phi_n(\tilde{H}_{|S|-2}(P_n^{S\setminus \max S}))$.  We already verified in Corollary~\ref{second-uniform}  that the left hand side, namely
the  span of the $\fS_{n+1}$-orbit of 
$\phi_n(W\!H_S(P_n))$, is equal  to $W\!H_S(P_{n+1})$.   For the right hand side, we have vector space containments  within the vector spaces  $\tilde{H}_{|S|-1}(P_{n+1}^S)$ and $\tilde{H}_{|S|-2}(P_{n+1}^{S\setminus \max S})$ that we wish to prove are isomorphisms of vector spaces.  However,  we also  know that 
$$W\!H_S(P_{n+1}) \cong \tilde{H}_{|S|-1}(P_{n+1}^S)\oplus \tilde{H}_{|S|-2}(P_{n+1}^{S\setminus \max S}).$$   Thus, by  dimensionality considerations   the vector space
containments  for rank-selected homology must also be isomorphisms. 
\end{proof}

\begin{rk}
    An analogous result to Proposition ~\ref{get-FI-modules} also holds for $\{ \mathcal{B}_n(q)\} $ for fixed $q$, with the category of finite sets $\{ {\bf n} := \{ 1,2,\dots ,n\}\}  $ replaced by the category of finite dimensional vector spaces $\{ \FF_q^n|n\ge 1\} $ over a finite field $\FF_q$ and the endomorphism groups $\{ \fS_n \} $ replaced by general linear groups over $\FF_q$.   In this setting, the element $u$ appearing in the proof of Proposition ~\ref{get-FI-modules} is the subspace of $\FF_q^n $ in which the last $n-m$ coordinates are set to 0, while the first $m$ coordinates may vary freely within $\FF_q$.  
    While these are not FI-modules, again one expects a strong analogy with $B_n$, and indeed we will give a sharp stability bound in this case in Section ~\ref{Boolean-section}.  It is straightforward to verify the general linear group analogues of Corollaries ~\ref{first-uniform}, ~\ref{second-uniform} and~\ref{second-beta} for $\{ \mathcal{B}_n(q)\} $.
\end{rk} 
   
   We will verify the third requirement for uniform representation stability for $\{ W\!H_S (B_n) ,\phi_n\} $ and for $\{ \beta_S(B_n),\phi_n\} $
    with stable range $n\ge 2\max S - |S| +1$ in both cases  in Theorem ~\ref{thm:Bool-stability}.  
    We will verify the  third requirement  for uniform representation stability for  $\{ \beta_S (\Pi_n)\} $ and $\{ W\!H_S(\Pi_n)\} $ 
    with stable range  $n\ge 4\max S - |S| + 1$ in both cases  in Theorem  ~\ref{main-partition-stability-theorem}.

\section{Rank-selected homology of Boolean lattices}\label{Boolean-section}

In this section, we show in the case of the Boolean lattice $B_n$ how to use an EL-labeling  
to produce a basis  for the top homology of the rank-selected subposet $B_n^S$ for any rank set $S$.   This will serve as a model for our more general results regarding geometric lattices in the next section.

\subsection{The Boolean lattice}

Proposition~\ref{Boolean-rank-selection} is   well known to experts in this area. 
We next give a proof  we discovered for this result that is not the usual proof found in the literature, leaving some of the details to be verified in greater generality in the next section, once more machinery is introduced.    We include this proof now  (with a few details deferred until later)   because it will show, in a substantially  simplified setting, the general approach we use for the partition lattice and other geometric lattices later in the paper. 

\begin{prop}[{\cite{So}, \cite[Theorem 4.3]{Stanley-aspects}, \cite[Theorem 3.4.4 and Exercise 3.4.5]{WachsPosetTop2007}}]
\label{Boolean-rank-selection}
Let $S = \{ s_1 < s_2 < \cdots < s_r \} \subseteq \{ 1,2,\dots ,n-1\} $.
The rank-selected homology $\beta_S(B_n)$ for the Boolean lattice on $n$ letters   has $\fS_n$-module structure given by the Specht module of ribbon shape $\Rib(s_1, s_2-s_1, s_3-s_2,\dots ,s_r-s_{r-1},n-s_r )$.
\end{prop}

\begin{proof}
 One way to see this result holds is by using the standard EL-labeling for  a Boolean lattice which labels each cover relation $S\prec S \cup \{ i\} $ with the label $i$. From this one obtains a homology basis for the rank-selected homology $\beta_S (B_n)$, with the generators of the homology basis corresponding in a completely natural way  
 to the polytabloids given by the fillings of the  ribbon shape $\Rib(s_1, s_2-s_1, s_3-s_2,\dots ,s_r-s_{r-1},n-s_r) $ that are strictly increasing in rows (left to right) and columns (top to bottom).   This correspondence is via
 the $\fS_n$-equivariant map ${\bar f}_{chain}$ defined  as  a quotient map of the map  $f_{chain}$, introduced next.  The map $f_{chain}$ sends each permutation $\pi_1\dots \pi_n \in \fS_n$ to the unique maximal chain
having $\pi $ as its label sequence, i.e., the chain
 \[\hat{0} \prec \{ \pi_1 \} \prec \{ \pi_1,\pi_2\} \prec \cdots \prec \{ \pi_1, \ldots,\pi_n\}.  \] 
Next we replace $\fS_n$ by the quotient group 
$\fS_n /(\fS_{s_1}\times \fS_{s_2-s_1}\times \cdots \times \fS_{n-s_r})$ where we caution readers that in this proof (and throughout this paper)   $\fS_{s_1}\times \fS_{s_2-s}\times \cdots \times \fS_{n-s_r}$ denotes the subgroup of 
$\fS_n$ which permutes the letters appearing in the leftmost $s_1$ positions in one-line notation amongst themselves and likewise permutes the letters appearing in positions $\{ s_1+1,\dots ,s_2\} $ amongst themselves and so on (in other words,  we regard $\fS_n $ as permuting positions rather than values); once we identify permutations with fillings of our ribbon later in this proof, this subgroup  will be exactly the row stabilizer group for the ribbon.  
Observe that two permutations are in the same equivalence class in
$\fS_n /(\fS_{s_1}\times \fS_{s_2-s_1}\times \cdots \times \fS_{n-s_r})$ if and only if they map to maximal chains in $B_n$ whose restrictions to rank set $S$ yield identical maximal chains in $B_n^S$.  This allows us to derive from $f_{chain}$ a quotient  map ${\bar f}_{chain}$  sending each  element of 
$\fS_n /(\fS_{s_1}\times \fS_{s_2-s_1}\times \cdots \times \fS_{n-s_r})$ to  a maximal chain in $B_n^S$.    One may easily 
observe that this quotient map is in fact a bijection.  

But ${\bar f}_{chain}$ may alternatively be viewed as a map sending each 
tabloid $\{ T\} $ of shape 
$\Rib (s_1,s_2-s_1,\dots ,n-s_r)$ having  distinct entries $\{ 1,2,\dots ,n\} $ to a maximal chain in $B_n^S$ by identifying these tabloids bijectively with the elements of $\fS_n/(\fS_{s_1}\times \fS_{s_2-s_1}\times \cdots \times \fS_{n-s_r})$; this identification of tabloids with elements of the quotient group arises by regarding  each 
$\pi \in S_n $ as a filling of the shape 
 $\Rib(s_1,s_2-s_1,\dots ,s_r-s_{r-1},n-s_r)$ by inserting the sequence 
 $\pi_1,\pi_2,\dots ,\pi_n$ of entries  from left to right and bottom to top in the shape $\Rib (s_1,s_2-s_1,\dots ,n-s_r)$.

 Treating ${\bar f}_{chain}$ as a map from tabloids to maximal chains in $B_n^S$, one may easily check that ${\bar f}_{chain}$ sends each basis vector $v_T$ of the Specht module of shape $\Rib(s_1,s_2-s_1,\dots ,n-s_r)$ to a linear combination of maximal chains in $B_n^S$ whose boundary is 0 (as is proven in Proposition~\ref{do-have-cycles} in the more general setting  of geometric lattices).   Example~\ref{ex:hom-cycle-Boolean} illustrates these ideas. 

 Shellability theory allows one to deduce that the homology facets in the  shelling for $B_n^S$  induced by the EL-labeling for $B_n$, are  exactly the maximal chains of $B_n^S$ obtainable by applying  the map $f_{chain}$ 
 to a  standard Young tableaux of shape $\Rib(s_1,s_2-s_1,\dots ,n-s_r)$ and then restricting the resulting maximal chain of $B_n$ to rank set $S$. 
 Using the fact that each cycle in our proposed homology basis has a homology facet in its support together with the fact  that each homology facet appears in the support of  exactly one of these  cycles,  it can be shown (via the reasoning  in the proof of  Theorem ~\ref{thm:homology-basis-EL-lab})  that these cycles indexed by the standard Young tableaux of shape $\Rib (s_1,s_2-s_1,\dots ,n-s_r)$ comprise a homology basis.
\end{proof}

 \begin{ex}\label{ex:hom-cycle-Boolean} Consider the rank set $S=\{2,5\}$ in $P=B_8$. 
Let $\pi$ be the permutation (in one line notation) 34167258, corresponding to the maximal chain 
\[f_{chain}(\pi)=(\hat 0\overset{3}{<}3\overset{\bf 4}{<}{\bf 34}\overset{1}{<}134\overset{6}{<}1346\overset{\bf 7}{<}
{\bf 13467}\overset{2}{<}123467\overset{5}{<}1234567\overset{8}{<}\hat 1).\]

This maximal chain in $B_8$ restricts to the maximal chain $\gamma=(\emptyset<\{3,4\}<\{1,3,4,6,7\}< \hat 0)$ in the rank-selected subposet  $B_8^S$, and maps bijectively to the ribbon filling of shape $\Rib(2, 3, 3)$ (with row lengths determined by $S$) given by 
\[F=\tableau{& & &2 &5 & 8\\ &1&6 &7\\3&4}
\]
By definition of the polytabloid $v_F$, we have 
\[v_F=\{F\}-\{(1,4) F\}-\{(2,7) F\} +\{(1,4)(2,7) F\}
\]
Hence $\bar f_{chain}(v_F)$ equals the following linear combination of maximal chains in $B_8^S$:
\[(\hat 0<34<13467<\hat 1) -(\hat 0<31<13467<\hat 1)
-(\hat 0<34<13462<\hat 1)+(\hat 0<31<13462<\hat 1)
\]
It is easy to check that the boundary is zero, i.e., it is a homology cycle, a fact proved 
more generally for arbitrary geometric lattices in   Proposition~\ref{do-have-cycles}. 
\end{ex}

Proposition~\ref{Boolean-rank-selection} completely determines the rank-selected homology of the Boolean lattice.   
Our next task is to establish the  \emph{sharp} stability bound of $2\max S - |S| + 1$ for $\{ \beta_S(B_n)\}$ as a consequence of this result.

In Section 7 of ~\cite{Church-Farb}, Church and Farb proved 
for the Boolean lattice $B_n$ that the rank-selected homology  modules $\{ \beta_S(B_n) \} $ are multiplicity  representation stable. 
See Definition~\ref{def-of-mult-stability} and Example~\ref{example-of-mult-stability} for the definition of multiplicity representation stability and an example that exhibits multiplicity representation stability but which is not representation stable.  

In \cite{Church-Farb}, the authors proved more specifically 
that the coefficient of the irreducible  $\cS^{(n-|\lambda|, \lambda)}$ within $\beta_S(B_n)$ takes the same value for all 
$n\ge \max (S) + |\lambda |$,
thereby establishing multiplicity representation stability.
No upper bound was given for 
$|\lambda |$, so this work from \cite{Church-Farb} had left open the question of whether $\{ \beta_S(B_n)\} $ was representation stable.
Corollary~\ref{second-uniform} 
implies $|\lambda |\le \max (S) $ for all $\cS^{(n-|\lambda|, \lambda)}$ 
appearing in the stabilized  formula  for $\beta_S(B_n)$.  Thus, we have 
established representation stability for $\{ \beta_S(B_n) \} $ by  giving a stability bound of $2\max S$.

Next we present one of the main results of this section by sharpening this bound.  That is, we will  
deduce a sharp stability bound of $2\max S - |S| + 1$ in Theorem~\ref{thm:Bool-stability}.

Recall the definition of the rank-selected Whitney homology $W\!H_S(P)$ from Equation  \eqref{eqn:WH-S-defn}. 
Proposition~\ref{prop:alt-sum-expression-S} and Proposition~\ref{prop:WHS-bound-equals-betaS-bound-anyP-AGAIN} allow us to deduce the following equivalence. 

\begin{cor}\label{cor:WHBoolS-bound-equals-betaS-bound} 
The rank-selected Whitney homology module 
$W\!H_S(B_n)$ stabilizes sharply at $2\max S - |S| + 1$ for all $S$ if and only if the rank-selected homology module $\beta_S(B_n)$ stabilizes sharply at $2\max S -|S| + 1$ for all $S$. 
\end{cor}
\begin{proof} Take $P_n=B_n$ and $f(S)=2\max S-|S|+1$ 
in Proposition~\ref{prop:WHS-bound-equals-betaS-bound-anyP-AGAIN}.
\end{proof}

\begin{thm}\label{thm:Bool-stability} Let $S\subseteq \{1,2,\ldots,n-1\}$.  Then  the rank-selected Whitney homology $W\!H_S(B_n)$, and hence the rank-selected homology $\beta_S(B_n)$,  
stabilizes sharply at $n=2\max S-|S|+1$. 
\end{thm}

\begin{proof} Let $S=\{s_1<s_2<\cdots<s_r=\max S\}$.   We begin by observing that if $u$ is a subset of $B_n$ of size $\max S$,  the interval $[\hat 0, u]$ 
is isomorphic to the Boolean lattice  $B_{\max S}$. Thus the $\fS_n$-module structure of the rank-selected Whitney homology $W\!H_S(B_n)$  is given by the induced module 
\begin{equation}\label{eqn:WHS-Bool-module}
    \Ind_{\fS_{n-\max S}\times \fS_{\max S}}^{\fS_n} \left(\one_{\fS_{n-\max S}}\otimes \widehat{W\!H}_S(B_n)\right)
\end{equation}
where $\widehat{W\!H}_S(B_n)$ is the $\fS_{\max S}$-module coinciding with the rank-selected homology module 
$\beta_{S\setminus \{\max S\}}(B_{\max S})$.
By Proposition~\ref{Boolean-rank-selection}, this in turn is the  Specht module for $\fS_{\max S}$ corresponding to the ribbon $\Rib(s_1, s_2-s_1,\ldots, {\max S}-s_{r-1})$. 

 By examining Littlewood-Richardson fillings \cite{Macdonald}, we see that this ribbon representation decomposes into irreducibles  indexed by partitions $\lambda$ such that $\lambda_1$   is bounded above by $\max S - |S| + 1$, with at least one irreducible representation having precisely this first row length. 

 From~\eqref{eqn:WHS-Bool-module}, the Frobenius characteristic of the $\fS_n$-module $W\!H_S(B_n)$ is $h_{n-{\max S}} f$, where $f$ is the ribbon Schur function for the Specht module in the preceding paragraph, and therefore has degree $\max S$.   From Lemma~\ref{lem:quasi-free-implies-increasing-mult}, Part 2, we conclude that $W\!H_S(B_n)$ stabilizes sharply at ${\max S}+\max S - |S| + 1=2\max S-|S|+1$, as claimed.   

 The statement for $\beta_S(B_n)$ now follows from Corollary~\ref{cor:WHBoolS-bound-equals-betaS-bound}.
\end{proof}

Theorem~\ref{thm:Bool-stability} in fact establishes the following result. 
We can also prove this result directly using the Littlewood-Richardson rule.
\begin{cor}\label{cor:ribbonSpechts-sharp-rep-stab}   For a fixed subset $S$, the ribbon $\fS_n$-Specht modules indexed by $\Rib(s_1, s_2-s_1,\ldots, {\max S}-s_{r-1}, n-\max S)$ are representation stable, stabilizing sharply at $n=2\max S-|S|+1.$
\end{cor}

\begin{cor}\label{cor:Boolean-FI} The rank-selected Whitney homology  $\{ W\!H_S(B_n)\} $ and the rank-selected homology $\{ \beta_S(B_n)\} $ are both finitely generated FI-modules with FI degree exactly  $\max S$ and sharp stability bound of $2\max S - |S| + 1$. 
Moreover, $\{ W\!H_S(B_n)\} $ is a quasi-freely generated FI-module. 
\end{cor} 
\begin{proof} For the  Whitney homology statement, 
combine the sharp stability bound of $2\max S - |S|+1$ for $\{ W\!H_S(B_n)\} $ proved in Theorem~\ref{thm:Bool-stability} with Proposition ~\ref{prop:WHS-bound-equals-betaS-bound-anyP-AGAIN} and Corollaries ~\ref{first-uniform} and~\ref{second-uniform}.
    For the rank-selected homology statement, 
    combine Theorem~\ref{thm:Bool-stability} with  Corollaries~\ref{first-uniform}  and~\ref{second-beta}.

For the last statement, we note that 
the usage of $\widehat{W\!H}_S(B_n)$ in 
Equation~\eqref{eqn:WHS-Bool-module} of the proof of Theorem~\ref{thm:Bool-stability}
 establishes, by Definition~\ref{quasi-free}, quasi-freeness of $\{ W\!H_S(B_n)\} $.
\end{proof}

\begin{cor}\label{cor:chain-stab-Bool}  The rank-selected modules of chains $\alpha_S(B_n)$ stabilize sharply at $2\max S$.  
\end{cor}
\begin{proof}  This follows from Corollary~\ref{cor:chain-module-Stab} applied to $P_n=B_n$ with $f(S)=2\max S-|S|+1$, and Theorem~\ref{thm:Bool-stability}.
\end{proof}
Corollary~\ref{cor:chain-stab-Bool} implies that for fixed $\mu$ and $S$, the Kostka numbers $K_{\mu,\lambda(S)}$ stabilize at $2\max S$. This is because $\ch \alpha_S(B_n)= \sum_{\mu\vdash n} K_{\mu, \lambda(S)}\, s_\mu=h_{\lambda(S)}$, where $\lambda(S)$ is  the partition of $n$ obtained by rearranging the elements   $\{s_1, s_2-s_1,\ldots, \max S-s_{r-1}, n-\max S\}$ in decreasing order. 

Our final result in this section  answers a question asked by Colin Crowley about the subspace lattice (personal communication).  
Let  $\cB_n(q)$ be 
 the lattice of subspaces of an $n$-dimensional vector space over the finite field $\mathbb{F}_q$ with $q$ elements;  it is a geometric lattice  of rank $n$ \cite[Example 3.10.2]{EC1}.  The rank function is given by vector space dimension. The general linear group $GL_n(q)$ is a group of automorphisms for $\cB_n(q)$.  In \cite[Section 5]{Stanley-aspects}, Stanley determines the structure  of the rank-selected homology $GL_n(q)$-modules.  Let $\{\lambda\}_q$ be the $GL_n(q)$-irreducible indexed by the partition $\lambda$. Then 
we have the following.
\begin{prop}\label{prop:subspace-lattice--homology-q-analogue-Boolean} Let $S$ be any set of ranks of $B_n$, and let $\lambda$ be a partition of $n$. Then the multiplicity of the $\fS_n$-irreducible indexed by $\lambda$ in the rank-selected homology $\beta_S(B_n)$ 
coincides with the multiplicity of the $GL_n(q)$-irreducible $\{\lambda\}_q$ indexed by $\lambda$ in the rank-selected homology $\beta_S(\cB_n(q))$.  Hence  the rank-selected homology modules for $\cB_n(q)$ stabilize sharply at $n= 2\max S-|S|+1$.
\end{prop}
\begin{proof} Stanley  \cite[Theorem 5.1]{Stanley-aspects} shows that the multiplicity of the $GL_n(q)$-irreducible $\{\lambda\}_q$ indexed by $\lambda$ in the rank-selected homology $\beta_S(\cB_n(q))$ equals the number of standard Young tableaux of shape $\lambda$ with descent set $S$.  But it is well known  (again see, e.g,  \cite{Stanley-aspects}) that this is also 
the multiplicity of the irreducible Specht module  $\cS^\lambda$ in the Specht module indexed by the ribbon $\Rib(S)$, and the result follows using Proposition~\ref{Boolean-rank-selection}.   
The sharp stability bound is now immediate from Theorem~\ref{thm:Bool-stability}.
\end{proof}

It seems natural to ask whether this same approach may be used to analyze representation stability for other families of posets with $\fS_n$-actions.  In the case of the partition lattice, we show in the remainder of this paper that the answer is yes, once certain obstacles are overcome.  
First we describe some of the issues that arise.

One key challenge in the case of the partition lattice,  and other geometric lattices, is that the analogues of  $\beta_S(B_n)$ (and the closely related $\widehat{W\!H}_S(B_n)$)
are no longer  Specht modules of ribbon shape.  We will, however, construct a sort of ``ribbon'' homology basis which does enable us to push the above viewpoint further.
Arbitrary geometric lattices (including the partition lattice in particular) also present one other challenge:
a decomposition of the rank-selected homology 
into (skew) Specht modules
is rarely known, 
unlike in the Boolean case where ribbon Specht modules are involved.  Hence tools such as Littlewood-Richardson fillings 
and symmetric functions 
are no longer available  to deduce an upper bound on $\lambda_1$  
for the irreducibles $\cS^{\lambda }$ appearing in the decomposition.

To get past these issues, 
in Section ~\ref{geometric-ribbon-section} 
we will introduce new ribbon bases for rank-selected homology and Whitney homology of geometric lattices (and for the  essential part of Whitney homology)  that are strikingly similar to the polytabloid bases for Specht modules of ribbon shape. 
We will then be able to use the property of Young symmetrizers described next  to get our desired  upper bound of $\lambda_1 +|\lambda |$ for irreducible representations $\cS^{\lambda }$ appearing in  
$\widehat{W\!H}_S(P)$ for geometric lattices $P$ with automorphism group containing $\fS_n$, so in particular for 
$\Pi_n$ and $B_n$. 
We describe this property in the simplified setting of the Boolean lattice $B_n$, where one may use 
Theorem ~\ref{Boolean-rank-selection} to express
$\widehat{W\!H}_S(B_n)$  
as a Specht module of ribbon shape $\Rib (s_1,s_2-s_1,\dots ,s_r-s_{r-1})$ for  $S = \{ s_1 < \cdots < s_r\} $. 

Given a standard Young tableau $T$ of shape $\lambda \vdash s_r$ 
and a filling $F$ of the ribbon shape 
$\Rib (s_1,s_2-s_1,\dots ,s_r-s_{r-1})$, if 
$\lambda_1$ is strictly larger than the number of columns in $\Rib (s_1,s_2-s_1,\dots ,s_r-s_{r-1})$ 
then there must be some $1\le i < j\le s_r$ 
with $i$ and $j$ both appearing in the first row of $T$ and also both appearing in the same column of $F$.  Lemma~\ref{lem:Youngsym-annihil-Fulton} then implies that $a_T v_F=0$.  
Since this holds for all standard fillings $F$ of the shape $\Rib (s_1,s_2-s_1,\dots ,s_r-s_{r-1})$, 
it holds for all of the elements of a basis for 
$\cS^{\Rib (s_1,s_2-s_1,\dots ,s_r-s_{r-1})}$, 
implying $b_Ta_T\cdot \cS^{\Rib (s_1,s_2-s_1,\dots ,s_r-s_{r-1})}=0$.   
This in turn implies that no irreducible representation $\cS^{\lambda }$ having $\lambda_1 $ larger than the number of columns in 
$\Rib (s_1,s_2-s_1,\dots ,s_r-s_{r-1})$, namely greater than $s_r - r + 1$,   
appears with positive multiplicity in 
$\widehat{W\!H}_S(B_n)$.   
This upper bound on $\lambda_1$ allows the usage of Proposition~\ref{prop:quasi-freely-gen-immplies-sharp-max} to deduce a stability upper bound of $s_r + (s_r-r+1) = 2\max S - |S| + 1$ for $W\!H_S(B_n)$.
See Example~\ref{ex:symm-annihilates-asis-elt-Boolean-larger-ex}. 
\begin{ex}\label{ex:symm-annihilates-asis-elt-Boolean-larger-ex}  Consider  the rank set $S=\{2,3,4,5,8,9\}$ in $B_{12}$.  Then $\widehat{W\!H_S}(B_{12})$ is an $\fS_9$-module.  Also let $\lambda=(5,2,2)$, a partition of $\max S=9$.

Let  $T=\begin{tableau}{1 & {\bf 2} &{\bf 4} & 6 &9 \\3& 5  \\7 &8 }\end{tableau}$ and 
$F=\begin{tableau}{& & &6\\&{\bf 2} &5 &9\\&3\\& {\bf 4}\\&7\\ 1&8}
\end{tableau}.$

Then the polytabloid $v_{F}$ corresponds $\fS_9$-equivariantly to a basis element for the homology module
$\widehat{W\!H}_S(B_{12})$, since $F$ is a standard filling of the ribbon shape $\Rib(2,1,1,1,3,1)$  for the rank set $S$.  
  Notice that $\lambda_1=5$ is greater than the number of columns of $F$, which is 4. Hence  any  standard Young tableau of shape $\lambda$  must have at least two entries in the first row that also belong to the same column in $F$; in  this example  the entries are $2,4$ in $T$.
  The existence of entries $2,4$ appearing in the same row of $T$ and the same column of $F$ will ensure that applying the Young symmetrizer given by $T$ to the ribbon basis element given by $F$ must yield 0, as we explain in more detail shortly.  Such calculations will guarantee that the irreducible of shape $(5,2,2)$ occurs with multiplicity 0 in $\widehat{W\!H}_S(B_{12})$.

  Let $H$ be the subgroup generated by the transposition $(2,4)$, and let $G$ be the stabilizer of the column of $F$ containing 2 and 4. 
One may use this pair of entries $2,4$ to show that $a_Tv_{F}=0$, as we now explain.  The row symmetrizer of $T$ factors into 
   $a_T= (\sum_{\sigma\in \mathrm{Row}_T/ H} \sigma) (1+ (2,4)),$  where $\mathrm{Row}_T/ H$ is a complete set of distinct left  coset representatives of $H$, including the identity element  in $G$,  in $\mathrm{Row}_T.$

   The column symmetrizer $b_{F}$ can similarly be factored into $b_{F}= (1-(2,4)) \sum_{\sigma\in G / H} \mathrm{sgn}(\sigma) \sigma$, where now $G/H$ is taken to be a complete set of  right  coset representatives of $H$ in $G$, including the identity element. 
The polytabloid $v_{F}$ can then be written as follows: 
 $ v_{F} = \{  (1-(2,4)) \sum_{\sigma\in G / H} \mathrm{sgn}(\sigma) \sigma \cdot F\}$. 
 Finally  the fact that  $(1+ (2,4)) (1- (2,4))=0$ forces $a_T \cdot  v_{F}=0$, and hence $b_Ta_T\cdot v_{F}=0$, as desired. 
\end{ex}

In Section ~\ref{geometric-ribbons} we will construct  the ribbon bases needed to 
use Young symmetrizers in the manner discussed above to analyze stability for the partition lattice. 
In Section~\ref{Young-section}, we develop an analogue of Lemma~\ref{lem:Youngsym-annihil-Fulton}, namely a key property of Young symmetrizers, 
that will be applicable to these 
ribbon bases. 
The pair of entries $2,4$ appearing in the same row of $T$ and the same column of $F$ in our example above will be replaced when we turn to the partition lattice $\Pi_n$ in 
Section~\ref{Young-section}  by a quadruple 
$i,j,k,l\in \{ 1,\dots ,n\} $ all appearing in the same row of $T$, where there are atoms $(i,j)$ and $(k,l)$ both appearing  in the same column of $F$.  Proving the existence of these quadruples will be a particularly challenging part of our proof of a sharp stability bound for the rank-selected homology of the partition lattice later in the paper. 

\section{Ribbon bases for rank-selected homology of geometric lattices}\label{geometric-ribbon-section}

In this section, we establish a new ribbon basis for the rank-selected homology and the rank-selected Whitney homology of any geometric lattice.  In the case of the partition lattice $\Pi_n$, these bases give rise to $\fS_n$-modules which  may be regarded as graphical analogues of Specht modules --  the atoms of $\Pi_n$ are in natural bijection with the edges in the complete graph on $n$ labeled vertices, reflecting the fact that  $\Pi_n$ is the lattice of flats of the graphic matroid given by a complete graph on $n$ vertices. 
We go on to show that, in the case of the partition lattices, Young symmetrizers act on these $\fS_n$-modules in much the same way as they do for traditional Specht modules.

\subsection{Construction of ribbon bases 
}\label{geometric-ribbons}

Motivated by the ribbon-shaped Specht module  structure for the rank-selected homology of Boolean lattices discussed  earlier in the paper, we now construct ``ribbon bases'' for all geometric lattices.
In the case of the partition lattice, this ribbon basis  
 will allow us to deduce sharp  representation stability results later in the paper.
 
 We note that our homology basis construction is somewhat
 reminiscent of the homology basis for the $d$-divisible partition lattice constructed by Wachs in ~\cite{Wachs-divisible}, in that $G$-equivariant bases  are built using 
 (generalized) polytabloids;
 however, our construction is necessarily 
 quite different from the construction 
 in \cite{Wachs-divisible} so as 
 to apply  to rank-selected homology and do so for  
 arbitrary 
 geometric lattices.

In Lemma 7.6.2 and Theorem 7.6.3  in \cite{Bj-hom-matroid}, Bj\"orner shows that certain edge labelings for geometric lattices (referred to  as minimal labelings in \cite{Davidson-Hersh}) are EL-labelings. 
This notion first appears in \cite{Stanley-AdmissibleLattices}, and is defined next.
\begin{defn}\label{def:minlabel}  Let $P$ be a geometric lattice, and for each $x\in P$  let $A(x)$ denote the set of atoms $a$ satisfying $a\le x$ in $P$. Choose  any total order $a_1 < a_2 < \cdots < a_r$ on the atoms of  $P$, and then label each cover relation $u\prec v$ with the label  $\min A(v)\setminus A(u)$, i.e., with the smallest atom that is less than or equal to $v$ but not less than or equal to $u$. This edge labeling $\lambda $ of $P$ is known as the {\em minimal labeling} of $P$ given by the  atom ordering $a_1 < a_2 < \cdots < a_r$.
\end{defn}
An example of a minimal labeling is  the EL-labeling used in the proof of Proposition~\ref{Boolean-rank-selection} for the Boolean lattice. We will be especially interested in the case of the partition lattice $\Pi_n$.   In that case the atoms naturally correspond to pairs $\{ i,j\} $ with 
$1\le i<j\le n$.  Any atom ordering for a geometric lattice was shown by Bj\"orner to give rise to a minimal labeling that is an EL-labeling.  Examples of posets which do not admit minimal labelings at all include weak order (which is not atomic) and Bruhat order (which is not a lattice).

The maps defined next will be important to how we define our homology bases, allowing us to transfer ideas from Specht module theory into the realm of poset topology for geometric lattices.
We will rely heavily on  the notion of a {\it basis of a matroid of rank $n$}. By this we mean a set of $n$ atoms in a geometric lattice $P$ of rank $n$ such that the join of this set of atoms is the maximal element $\hat{1}$ in $P$; this terminology is motivated by the fact that every geometric lattice has a matroid associated to it, whose ground set 
is exactly the set of atoms of the geometric lattice, and whose {\it independent sets} 
are the sets $\{ a_{i_1},\dots ,a_{i_r}\} $ of atoms such that  the join $a_{i_1}\vee \cdots \vee a_{i_r}$ has  rank  $r$.  We refer the reader to \cite{Bj-hom-matroid} for more details about the connection between  matroids and geometric lattices. 

\begin{defn}
Given any filling $F$ of a ribbon shape with atoms, define the \emph{reading word} of $F$ to be the word $(F_1,\dots ,F_n)$ in which $F_i$ is the atom appearing as the entry in the $i$-th box of the ribbon shape,  as we read the entries  from left to right and bottom to top.    
\end{defn}

Now we define the  map $f_{chain}$ sending ribbon fillings to maximal chains in geometric lattices.

\begin{defn}\label{def:fchain}
The map $f_{chain}$ is a surjective map from the space of  fillings of $\Rib (S)$ with independent sets of atoms of  $P$, of size ${\rm rk}(P)$,  
to the space   of maximal   chains in $P$.
The map $f_{chain}$ sends such a  filling $F$ with reading word  $(a_{i_1},a_{i_2},\dots ,a_{i_{{\rm rk} (P)}})$  to the maximal chain 
$(\hat{0}\prec a_{i_1} \prec a_{i_1} \vee a_{i_2} \prec \cdots \prec a_{i_1}\vee \cdots \vee a_{i_{{\rm rk}(P)}}=\hat{1}) $ in $P$.  

 This map is extended linearly to all  linear combinations of such fillings.  In other words, $f_{chain}$ is a surjective map from the space of fillings with matroid bases to the space of maximal chains of $P$. 
    \end{defn}

\begin{ex}\label{ex:ribbon-filling-ordered-atoms} Let $R$ be the ribbon $\Rib(3, 1,3,1).$
\[F=\tableau{&&&& a_8 \\  & &a_5 & a_6 &a_7\\ & & a_4\\ a_1 & a_2 & a_3}\longmapsto (\hat{0} \prec a_1 \prec a_1 \vee a_2 \prec \cdots \prec a_1 \vee \cdots \vee a_8=\hat{1})=f_{chain}(F).\]
\end{ex}
 
As in Section~\ref{subsec:YoungSym}, every such ribbon filling $R_{\cA}$ in turn produces a \emph{ribbon tabloid} $\{ R_{\cA} \}$ obtained by forgetting the order of letters  within each row, and then  a \emph{ribbon polytabloid} $v_{R_{\cA} }$ defined exactly as in~\eqref{eqn:polytabloid}, by applying the column stabilizer to the tabloid $\{R_{\cA}\}$.

Next we define a sort of one-sided inverse to $f_{chain}$, namely a  map sending maximal chains in $P$ 
to ribbon fillings.

\begin{defn}
Given any maximal chain $M$ in a geometric lattice $P$ with EL-labeling $\lambda $, let $f_{rib}(M)$ be the ribbon filling  whose reading word is the label sequence 
$\lambda(M)$.  
\end{defn}

In order to study the rank-selected poset $P^S$, we will make extensive use of the quotient map ${\bar f}_{chain}$ defined next.

    \begin{defn}\label{def:fbarchain}
    For any $S = \{ s_1<s_2<\cdots <s_k\} \subseteq \{ 1,2,\dots ,n-1\} $,
    let ${\bar f}_{chain}$ denote the quotient map obtained from $f_{chain}$ by replacing the space of ribbon fillings $R_{\cA}$ for the ribbon shape $R=\Rib(s_1,s_2-s_1,\dots ,n-s_k)$ to which $f_{chain} $ applies, with the space of corresponding ribbon tabloids $\{R_{\cA}\}$.    For the image of this quotient map, the space of maximal chains in $P$ is replaced with the space of maximal chains in $P^S$.

    The map ${\bar f}_{chain}$ sends any tabloid $\{ F\} $ given by a filling $F$ with reading word $(F_1,\dots ,F_n)$ to the maximal chain $F_1\vee \cdots \vee F_{s_1} < 
    F_1 \vee \cdots \vee F_{s_2} < \cdots < F_1 \vee \cdots  \vee  F_{s_k}$ in $P^S$.  Thus, ${\bar f}_{chain}$ is a map from the space of tabloids of shape $R$ given by fillings with bases of the matroid to the space of maximal chains in $P^S$. 
    \end{defn}

\begin{ex}\label{ex:f-bar-chain}
Let $P$ be a poset of rank 8, and let $S$ be the rank set $\{3,4,7\}$,
giving the ribbon $\Rib(3,1,3,1)$ and its filling $R_{\cA}$  in  Example~\ref{ex:ribbon-filling-ordered-atoms}.  Then the image under ${\bar f}_{chain}$ of the   tabloid $\{R_{\cA}\}$ is the following  chain in the rank-selected subposet $P^S$:
\[(\hat{0} <\underbrace{a_1 \vee a_2\vee a_3}_{\mathrm{Row \, 1}} < \underbrace{a_1 \vee a_2\vee a_3\vee a_4}_{\mathrm{Rows\, 1 \text{ and } 2}}  <\underbrace{a_1 \vee a_2\vee a_3\vee a_4\vee a_5\vee a_6\vee a_7}_{\mathrm{Rows\, 1, 2 \text{ and } 3}} <  \underbrace{a_1 \vee \cdots \vee a_8}_{\mathrm{Rows\, 1, 2, 3\text{ and } 4}}=\hat{1})
 \]    
\end{ex}

In order to better understand the map  ${\bar f}_{chain}$, 
it will be helpful to recall well-known maps from shellability theory (see Theorem~\ref{thm:rank-sel-hom-facets}) for passing back and forth between maximal chains in $P$ and maximal chains  in $P^S$, as well as recalling  the map sending tableaux to tabloids.

\begin{defn}\label{def:ffirst}
Given an EL-labeling 
for a graded, bounded  poset $P$ and a subset $S$ of the set of ranks in $\overline{P}$, 
let ${\rm res}_S$ denote the map sending each maximal chain $\gamma $ in $P$ to the  maximal chain in $P^S$ obtained  by restricting $\gamma $ to the rank set $S$, namely taking the subchain of $\gamma $ consisting exactly  of ranks in $S$.  

Let $f_{first}$ denote the map  (also defined earlier in Theorem ~\ref{thm:rank-sel-hom-facets}) 
sending each maximal chain $\beta  $ in $P^S$ to the  
maximal chain in $P$ containing $\beta $ which has   lexicographically earliest label sequence.    This will be used to characterize shelling homology facets in Proposition ~\ref{homology-facets-using-standard-ribbons}.
\end{defn}

Denote by $\tau$ the map 
which takes a Young tableau $T$ (of arbitrary shape) to its tabloid:  $\tau(T)=\{T\}$.

Given a set $S= \{ s_1,s_2,\dots ,s_k\}\subseteq \{ 1,2,\dots ,n-1\} $ of ranks with $s_1<s_2 < \cdots < s_k$, 
it will be  convenient in what follows 
to use the shorthand
$\Rib (S)$ for $\Rib (s_1,s_2-s_1,\dots ,n-s_k)$.

The following straightforward identity will be useful later in this section for proving that the proposed cycles in our homology bases are indeed cycles:  
    \begin{equation}\label{eqn:fchain-fbarchain}
    \bar f_{chain} \circ \tau=\mathrm{res}_S\circ f_{chain}.     
    \end{equation}

More precisely, let $\cM(P)$ denote the set of maximal chains of a  
geometric lattice $P$, and let $\cF(\Rib(S))$ denote the set of standard fillings of the ribbon shape $\Rib(S)$ corresponding to a rank set $S$.   Finally let $\{\cF(\Rib(S))\}$ denote the set of tabloids obtained from the fillings $\cF(\Rib(S))$.  Then according to Definitions~\ref{def:fchain} and~\ref{def:fbarchain}, we have 
\[f_{chain}:\cF(\Rib(S))\rightarrow \cM(P),\quad
{\bar f}_{chain}:\{\cF(\Rib(S))\}\rightarrow \cM(P^S).\]

\begin{ex}\label{ex:maps-ffirst-frib} Let $P$ be the  Boolean lattice $B_5$,  and let $S$ be the rank set $\{2\}$.
Consider the maximal chain 
 \[M^S=\hat 0< \{1,3\}<\hat 1 \text{ in } P^S.\]
  Using  the standard EL-labeling described in the proof of Proposition~\ref{Boolean-rank-selection}, the lexcographically earliest maximal chain $M$ in $P$ containing $M^S$ is
\[M=f_{first}(M^S)=\emptyset \overset{1}{<} \{1\}\overset{3}{<}\{1,3\}\overset{2}{<}\{1,3,2\}  \overset{4}{<}\{1,3,2,4\}  \overset{5}{<}\hat 1.\]
Then $f_{rib}(M)=f_{rib}(f_{first}(M^S))=\tableau{& 2 & 4 & 5\\1 &3}$.

Notice that $f_{chain}(f_{rib}(M))=f_{first}(M^S)=f_{first}({\rm res}_S(M)). $
\end{ex}

To describe  the elements of our homology basis for $\beta_S (P)$ and then also for $WH_S(P)$, we need one more ingredient: a notion of ordered basis introduced in Definition ~\ref{NBC-def}
that is slightly stronger than  the well-known concept  of an NBC basis for a matroid.   This will  characterize  
the types of label sequences that are  produced by minimal labelings (see Definition~\ref{def:minlabel}). 

\begin{defn}
An {\it NBC independent set} of a geometric lattice $M$  with atom  ordering $a_1 < a_2 < \cdots < a_n$ is a set 
$\{ a_{i_1},\dots ,a_{i_l}\} $ of atoms such that:
\begin{enumerate}
    \item 
$\vee_{k=1}^l a_{i_k}$ has rank $l$.
\item 
If an atom  $a\not \in \{ a_{i_1},\dots ,a_{i_l}\} $  satisfies $a\le \vee_{k=1}^l a_{i_k}$, then $a$ comes later in our atom ordering than at least one element of  $\{ a_{i_1},\dots ,a_{i_l}\} $.
\end{enumerate}
\end{defn}

Recall from Definition~\ref{def:minlabel} that $A(u)$ is the set of atoms  satisfying $a\le_P u$.  

\begin{defn}\label{NBC-def}
An {\it $\minNBC$ basis} 
 for $M$ is an ordered  $N\!BC $ independent set 
$(a_{i_1},\dots ,a_{i_l})$ of $M$ with $l=rk(M)$ such that 
each $a_{i_k}$ for $k=1,2,\dots ,l$ satisfies 
$a_{i_k}= \min A(\vee_{s=1}^k a_{i_s})\setminus A(\vee_{s=1}^{k-1}a_{i_s})$.
\end{defn}

\begin{rk}\label{NBC-matroid-view}
For those who prefer the language of matroid theory, this minimality requirement for $a_{i_k}$  in the definition of  $\minNBC$ basis is  equivalent to requiring 
for each $j=1,2,\dots ,l$ that   $\{ a_{i_j},\dots ,a_{i_l}\} $ is an NBC independent set for the matroid contraction in which we contract 
$a_{i_1},\dots ,a_{i_{j-1}}$; this is the matroid whose associated geometric lattice is the interval
$[J,
\hat{1}]$ in $M$ 
for 
$J = \vee_{k=1}^{j-1}a_{i_k}$. 
In taking this matroid contraction, we make the convention that we still use atoms of the original lattice, so now, for instance,  if two atoms $a_1,a_2$ have $a_1\vee J = a_2\vee J$, we regard $\{ a_1,a_2\} $ as a circuit in our matroid contraction so that $a_1 < a_2$ makes $a_2$ by itself a broken circuit; this notion of broken circuits of size 1 is 
exactly what forces the choice of 
$\min A(v)\setminus A(u)$ as the atom used to label
$u\prec v$ (as in the minimal labeling).
\end{rk}

\begin{rk}
The restriction of the map $f_{chain}$ to the set of fillings using $\minNBC$-bases of a matroid is a bijection from the set of $\minNBC$-bases to the set of maximal chains in the associated geometric lattice, by virtue of $f_{rib}$ serving as a one-sided inverse to $f_{chain}$, with the fillings whose reading words are  $\minNBC$-bases being exactly the fillings obtained by applying $f_{rib}$ to the maximal chains in $P$. 
\end{rk}

We are now prepared to describe our homology bases for 
$\beta_S(P)$ and $W\!H_S(P)$ for any geometric lattice $P$.  Our choices in defining these bases 
are motivated by the  characterization in 
Proposition~\ref{homology-facets-using-standard-ribbons}  
of homology facets in the shelling for $P^S$ as the image under
${\bar f}_{chain} $ of the standard fillings of a ribbon with fillings
whose reading words are  $\minNBC$-bases.  

\begin{defn}\label{beta-basis-def}
To construct a homology basis  
$\mathcal{B}^S_{rib}(P)$ for the top 
homology of $P^S$ for a 
geometric lattice $P$, we first specify a total order $a_1,\dots ,a_n$
on the atoms of $P$.  

We  let 
$$\mathcal{B}^S_{rib}(P) = \{ {\bar f}_{chain}(v_F) \mid F 
\text{ is a standard NBC$^+$ filling of } 
\Rib(s_1,s_2-s_1,\dots ,n-s_k)\},$$
a set described in more detail next.

The elements of  $\mathcal{B}^S_{rib}$ 
are indexed  by the standard fillings of the ribbon shape $\Rib(s_1,s_2-s_1,\dots ,n-s_k)$ with ordered matroid bases $(a_{i_1},\dots ,a_{i_j})$  that are $\minNBC$  bases with respect to our atom ordering.  From such an 
$\minNBC$-basis, we obtain a filling of 
$\Rib(s_1,s_2-s_1,\dots ,n-s_k)$ by proceeding from left to right and bottom to top through $\Rib (s_1,s_2-s_1,\dots ,n-s_k)$, putting the $i$-th atom in the 
$\minNBC$-basis in the $i$-th box encountered.
The standard fillings obtained in this way  are exactly the ribbon fillings sent by the map ${\bar f}_{chain}$ to the homology facets in the  shelling for $P^S$ given by the minimal labeling for $P$ induced by our atom ordering. 
See Theorem  \ref{thm:rank-sel-hom-facets} for this shelling for $P^S$.

We obtain from each of these standard fillings $F$ of 
$\Rib(s_1,s_2-s_1,\dots ,n-s_k)$ an element of our homology basis as follows. 
The first step is to consider the polytabloid associated to $F$, defined in~\eqref{eqn:polytabloid}, i.e., 
the  alternating sum 
$$v_F := \sum_{\sigma \in \mathrm{Col}_F} \mathrm{sgn} (\sigma ) \{ \sigma F\}$$ of tabloids $\{\sigma F\}$ where $\sigma $ permutes the column entries in $F$.  (Recall that a tabloid is a filling where we forget the order of the entries in each row.) 
Then  we apply the map
${\bar f}_{chain}$ to each such $v_F$ to obtain an alternating sum 
${\bar f}_{chain}(v_{F})$ of maximal chains in $P^S$.
These  alternating sums comprise 
$\mathcal{B}_{rib}^S(P)$.
\end{defn}

\begin{ex}\label{ex:ribbon-hom-basis-polytabloid}
Returning to Example~\ref{ex:maps-ffirst-frib}, 
we see that the filling $F=f_{rib}(M)$ of $\Rib(2,3)$ for the rank set $S=\{2\}$ of $B_5$ gives the difference of tabloids below:
    \[v_F=\left\{\tableau{& \textit{2} & 4 & 5\\1 &\textit{3}} \right\}- \left\{\tableau{& \textit{3} & 4 & 5\\1 &\textit{2}} \right\},\]
    because $\Col_F$ is the 2-element subgroup generated by the transposition $(2,3)$.
\end{ex}

\begin{ex}\label{ex:hom-Pi-1} Consider the rank set $S=\{2,5\}$ in $P=\Pi_8$. Let $\gamma$ be the maximal chain in $P^S$ defined by 
$\gamma=(\hat 0<|12|56|<|12356|78| < \hat 1),$ 
where we have omitted the blocks of size 1 for convenience.
Then 
\[f_{first}(\gamma)=\hat 0\overset{12}{<}|12|\overset{\bf 56}{<}{\bf |12|56|}\overset{13}{<}|123|56|\overset{15}{<}|12356|\overset{\bf 78}{<}
{\bf |12356|78|}\overset{14}{<}|123456|78|\overset{17}{<}\hat 1,\]
and 
\[F=f_{rib}( f_{first}(\gamma) )=\tableau{& & &14 &17 \\ &13&15 & {\bf 78}\\12& {\bf 56}}
\]
is a standard $\minNBC$ filling of the ribbon shape $\Rib(2, 3, 2)$.  The column stabilizer $\Col_F$ of the filling $F$ is the Klein-4 group of four permutations, of the two pairs of boxes containing the set of atoms $\{56, 13\}$ and $\{78, 14\}$ respectively. The permuted fillings are as follows:
\[F_1=\tableau{& & &14 &17 \\ & {\bf 56}  &15 & {\bf 78}\\12&   13} \qquad
F_2=\tableau{& & & {\bf 78} &17 \\ &13&15 & 14 \\12& {\bf 56}}\qquad
F_3=\tableau{& & & {\bf 78} &17 \\ &{\bf 56}  &15 & 14 \\12& 13 }.\]

Then 
by definition of the polytabloid $v_F$, we have 
\[v_F=\{F\}-\{ F_1\}-\{ F_2\} +\{ F_3\}.
\]
Hence the map $\bar f_{chain}$ sends $v_F$ to 
\[
\scalebox{0.75}{\ensuremath{
(\hat 0<|12|56|<|12356|78| < \hat 1) 
-(\hat 0<|123|< |12356|78|< \hat 1)
-(\hat 0<|12|56|< |123456| < \hat 1)
+(\hat 0<|123|<|123456| < \hat 1)
}}
\]
\end{ex}

\begin{defn}\label{Whitney-basis}
Consider  a geometric lattice $P$ with atom ordering $a_1,\dots ,a_n$ giving rise to a minimal labeling,  
and consider any subset $S$ of the nonzero  ranks of 
$P$. 
We construct  a  basis $\mathcal{W}^S_{rib}(P)$
for $W\!H_S(P)$ for our choice of atom ordering as follows.  Consider each $u\in P$ of rank $\max S$ separately, and use our construction in Definition ~\ref{beta-basis-def} for $\beta_S$ to obtain a ribbon homology basis for  $\beta_{S\setminus \max S}[\hat{0},u]$.  Take the union of these bases as our ribbon basis for $W\!H_S(P)$.
\end{defn}

\begin{defn}
We likewise derive as follows  a {\em ribbon basis} for $\widehat{W\!H}_S(P_n)$ for any sequence $P_1,P_2,\dots $ of posets where $\fS_n$ is a subgroup of  the automorphism group of $P_n$.  Choose one representative element $u$ of rank $\max S$ from each $\fS_n$-orbit of elements of rank $\max S$.  Take the  union of all of these chosen orbit representatives $u$ of the ribbon bases for $\beta_{S\setminus \max S}([\hat{0},u]$, and let this serve as the ribbon basis for $\widehat{W\!H}_S(P_n)$.     
\end{defn}

\begin{ex}\label{OS-example}
The elements of 
$\mathcal{W}_{rib}^{\{ 1,2,\dots ,k\} }(P)$  
for $1\le k\le {\rm rk}(P)$ 
are exactly those ${\bar f}_{chain}(v_T)$ such that $T$ is a standard filling of the ribbon $\Rib (1,1,\dots ,1)$ having $k$ boxes in a single column in which the boxes are filled from bottom to top  with any ordered NBC independent set $(a_{i_1},\dots ,a_{i_k})$ satisfying  $a_{i_1} > a_{i_2} > \cdots > a_{i_k}$; this is because any NBC independent set listed in descending order is an $\minNBC $-independent set, by results in \cite{Bj-hom-matroid} characterizing the descending chains in a minimal labeling.  These standard fillings  are in bijection with the maximal chains in $P$ that have descent set exactly  $\{ 1,2,\dots ,k\} $ via the map  sending such a filling to the maximal chain which includes 
$\hat{0}\prec a_{i_1} \prec a_{i_1}\vee a_{i_2} \prec \cdots \prec
a_{i_1}\vee \cdots \vee a_{i_k}$, and then proceeds upward from $a_{i_1}\vee\cdots\vee a_{i_k}$ to $\hat{1}$ along  the unique saturated chain on $[a_{i_1}\vee\cdots \vee a_{i_k},\hat{1}]$ having ascending labels under our minimal labeling.  \end{ex}

While this description in Example ~\ref{OS-example} may seem a bit different from how we constructed $\mathcal{W}_{rib}^{\{ 1,2,\dots ,k\} }(P)$ in Definition~\ref{Whitney-basis}, it is well known that the descending chains for a minimal labeling restricted to the geometric lattice $[\hat{0},u]$ have as their label sequences exactly  the  NBC independent sets $\{ a_{i_1},\dots ,a_{i_k}\} $ satisfying 
$a_{i_1}\vee \cdots \vee a_{i_k} = u$ with elements listed in descending order; moreover, by definition of minimal labeling the  distinct descending chains of $[\hat{0},u]$ are  labeled by distinct NBC independent sets. 
This viewpoint allows us to use our ribbon basis for $\beta_{\{ 1,2,\dots ,k-1\} }[\hat{0},u]$ to  give exactly  the desired basis elements given by  fillings having $a_{i_1}\vee\cdots \vee a_{i_k} = u$.  

In the remainder of this section, we will prove for any geometric lattice $P$  that 
$\mathcal{B}^S_{rib}(P)$ is indeed 
a basis for $\beta_S(P)$, and that $\mathcal{W}_S^{rib}(P)$ is indeed a basis for $W\!H_S(P)$.
First we rephrase the results of  
Theorem ~\ref{thm:rank-sel-hom-facets} and Proposition
~\ref{folklore-homology-facets}. Later in this section, using the 
characterization of homology facets for $P^S$ given in Proposition ~\ref{homology-facets-using-standard-ribbons} below,  we prove that our proposed homology bases 
$\mathcal{B}_{rib}^S(P)$ and $\mathcal{W}_{rib}^S(P)$ are indeed bases. 

\begin{prop}\label{describe-homology-facets}
Let $\lambda $ be an   EL-labeling for a graded poset $P$. Then $\lambda $ induces a shelling for $P^S$ in which the homology facets are exactly those facets given by maximal chains  of $P^S$ that are contained in maximal chains  of $P$
whose associated ribbon fillings (via the map $f_{rib}$) are standard fillings of $\Rib (s_1,s_2-s_1,\dots ,n-s_k)$.
\end{prop}

\begin{prop}\label{homology-facets-using-standard-ribbons}
Consider a geometric lattice $P$ with atom ordering $a_1,\dots ,a_n$ giving rise to a minimal labeling, 
and consider any subset $S$ of the set of ranks in $\overline{P}$.   Then the  homology facets in the induced shelling for $P^S$ are exactly those maximal chains in $P^S$ that are sent by 
$f_{rib}\circ f_{first}$ (composing functions right to left)  to standard fillings $F$ of $\Rib(S)$ whose  reading words 
comprise  $N\!BC^+$ bases for $P$.
\end{prop}

\begin{proof}
    The characterization of  homology facets  is justified by combining  Proposition ~\ref{describe-homology-facets} with Definition ~\ref{NBC-def}, using the fact that our minimal labeling is an EL-labeling, hence has a unique ascending chain in each interval of $P$. 
    
    More specifically, given any homology facet in $P^S$, note that we indeed obtain a standard filling by applying $f_{first}$ followed by $f_{rib}$ to it.  Conversely, given any standard filling of $\Rib(S)$  whose reading word is an $\minNBC$ basis, observe that this will be the label sequence on some maximal chain $\gamma $ in $P$ by virtue of how minimal labelings are defined.
    Moreover, the fact that the entries are increasing left to right in each row implies that this label sequence has ascending labels between each pair of elements in the maximal chain ${\rm res}_S (\gamma )$ in $P^S$.  But since there is a unique ascending chain in each interval of $P$,  and since this ascending chain is lexicographically earliest in that interval, we may conclude  that $\gamma = f_{first} ({\rm res}_S (\gamma ))$, and hence  is in  ${\rm im}(f_{first})$.  The fact that entries are increasing down columns guarantees that the maximal chain 
    ${\rm res}_S (\gamma )$ in $P^S$ 
    is indeed a homology facet in the shelling for $P^S$, by virtue of $\gamma $ having descents at all ranks in $S$.  
\end{proof}

Next we prove that the elements of our proposed homology basis $\mathcal{B}_{rib}^S(P)$ are cycles.

\begin{defn}\label{def:d-i}
    In what follows, let $d_i$ denote the part of the boundary map deleting the $i$-th element from a maximal chain in $P^S$.  Thus, 
    $$d_i(u_1<\cdots <u_{i-1}<u_i<u_{i+1}<\cdots < u_k) = (-1)^{i-1}(u_1<\cdots <u_{i-1}<u_{i+1}<\cdots <u_k)$$  and  $d = \sum_{i=1}^kd_i$.  
\end{defn}

\begin{ex}\label{ex:hom-cycle-Pi} Consider Example~\ref{ex:hom-Pi-1} again. 
Recall that $\bar{f}_{chain}(v_F)$ equals 
\[
\scalebox{0.75}{\ensuremath{
(\hat 0<|12|56|<|12356|78| < \hat 1) 
-(\hat 0<|123|<|123456| < \hat 1)
-(\hat 0<|12|56|<|12356|78| < \hat 1)
+(\hat 0<|123|<|123456| < \hat 1)
}}
\]
The component $d_2$ of the boundary map on $P^S$ removes the elements at rank 5, 
and 
changes the signs on both  chains.  But these chains already have opposite signs, so 
$d_2(\bar f_{chain}(v_F))=0$.  One can check also that 
$d_1(\bar f_{chain}(v_F))=0$.
\end{ex}

 \begin{defn}\label{filling-boundary}
Let us also define a boundary map $d^{fill}$ directly  on the fillings of a ribbon shape $\Rib (s_1,s_2-s_1,\dots ,n-s_r)$ as $\sum_{i=1}^{r+1}(-1)^{i-1}d_i^{fill}$ where $d_i^{fill}$ sends the ribbon filling $F$ to the unique ribbon filling of the shape 
$\Rib (s_1,s_2-s_1,\dots ,s_{i-1}-s_{i-2},s_{i+1}-s_{i-1},\dots ,n-s_r)$ having the same reading word as $F$.  In other words, $d_i^{fill}$ merges the $i$-th and $(i+1)$-st rows  (indexing rows from bottom to top)   into a single row,  by appending the $(i+1)$-st row  to the immediate right of the $i$-th row.  
\end{defn}

We illustrate next how $d^{fill}$ acts on a filling $F$ in comparison with how $d$ acts on $\bar{f}_{chain}(\{F\})$ for the same filling $F$,  in Example~\ref{ex:d-and-d-fill}.

\begin{ex}\label{ex:d-and-d-fill}  
Let $\{a_1, a_2,\ldots, a_7\}$ be an independent set of atoms in a geometric lattice of rank 7.
Let $S$ be the rank-set $S=\{2,5\}$, and consider the filling of the ribbon $\Rib(S)=\Rib(2,3,2)$ given by 
 $F=\tableau{& & &a_6 &a_7 \\ &a_3&a_4 & {\bf a_5}\\a_1& {\bf a_2}}.
$

Then $\bar f_{chain}(\{F\}) = (\hat 0< a_1 a_2< a_1 a_2\cdots a_5< \hat 1), $ 
where we have suppressed the join symbol for clarity, and written $a_1 a_2$ for $a_1\vee a_2$ etc.  It follows from Definition~\ref{def:d-i} that 
\[d_1(\bar f_{chain}(\{F\}))= (\hat 0< a_1\cdots a_5<\hat 1).\]

Also, from Definition~\ref{filling-boundary}, $d_1^{fill}(F)=\tableau{& && &a_6 &a_7\\a_1& a_2 & a_3 & a_4 & a_5}\ {\longleftrightarrow}\ 
(\hat 0< a_1\cdots a_5<\hat 1).$

Similarly $d_2(\bar f_{chain}(\{F\})) = (\hat 0< a_1 a_2< \hat 1), $ 
and
\[d_2^{fill}(F)=\tableau{&   a_3 & a_4 & a_5 &a_6 &a_7\\a_1& a_2}\ \longleftrightarrow\  
(\hat 0< a_1 a_2<\hat 1).\]
 This example illustrates the principle behind the next lemma.
\end{ex}

 \begin{lem}\label{d-on-fillings}
Given any filling $F$ of $\Rib (s_1,s_2-s_1,\dots ,n-s_r)$, we have  
$$d_l({\bar f}_{chain}(v_F))= 
({\rm res}_{S\setminus \{ s_i\} } \circ f_{chain} )\,\left(\sum_{\sigma \in \Col_F} \sgn (\sigma )\, \sigma(  d_l^{fill} F) \right).$$
 \end{lem}

 \begin{proof}
We begin by using 
Equation~\eqref{eqn:fchain-fbarchain} to observe that 
\[{\bar f}_{chain}(v_F) = {\rm res}_S(f_{chain}(\sum_{\sigma \in \Col_F} \sgn(\sigma ))\, \sigma F). \]  Next we observe that $d_l^{fill}$ describes the impact of the boundary on the fillings rather than on the chains, and as such  commutes  with ${\rm res}_* \circ f_{chain}$, by which we mean that
$$d_l \circ {\rm res}_S \circ f_{chain } = {\rm res}_{S\setminus \{ s_l \}} \circ f_{chain} \circ d_l^{fill}.$$
Note that  ${\rm res}_{S}$ is  replaced by 
${\rm res}_{S\setminus \{ s_l\} }$ when the boundary map is applied before the map ${\rm res}_*$ is applied, since the boundary $d_l^{fill}$ merges the $l$-th and $(l+1)$-st rows, and this has the impact of eliminating the chain rank $s_l$ from $S$. 
Finally, we use the fact  that $d_l^{fill}$ commutes with $\fS_n$ 
to complete the proof. 
 \end{proof}
 
\begin{rk}
    To prove that the elements of our proposed homology bases are indeed cycles, it will be helpful to define column and row stabilizers in a non-standard way that yields the same basis elements ${\bar f}_{chain}(v_F)$ for $v_F =  \sum_{\sigma \in \Col_F} {\rm sgn}(\sigma )\{ \sigma F\} $.
    
    Let $\Rib (S) = \Rib (s_1,s_2-s_1,\dots ,n-s_r)$ be the shape of the ribbon having filling $F$.  Let $(F_1,\dots ,F_n)$ be the reading word of $F$.  Consider the action of $\fS_n$ permuting positions rather than values in $(F_1,F_2,\dots ,F_n)$, so for instance 
    $$(j,j+1)(F_1,\dots ,F_n) = (F_1,\dots ,F_{j-1},F_{j+1},F_j,\dots ,F_n).$$  Let $\Col_{\Rib (S)}$ (resp. $\Row_{\Rib (S)}$) be the subgroup of $\fS_n$ which permutes elements that are in the same column (resp. row) as each other, using this action on positions rather than values.  
    
We will rely heavily on the observation that $$ 
\sum_{\sigma \in \Col_F} {\rm sgn}(\sigma )(\sigma F) =
    \sum_{\sigma \in \Col_{\Rib (S)}} {\rm sgn}(\sigma )(\sigma F) $$ where $\fS_n$ acts on values on the left side and acts on positions on the right side.  This allows us to rewrite $v_F$ as a signed sum using the action on positions, provided that we only pass to tabloids after applying the group action. 
\end{rk}

\begin{prop}\label{do-have-cycles}
Consider a geometric lattice $P$ with atom ordering $a_1,\dots ,a_n$, giving rise to a minimal labeling, 
and consider any subset $S=\{ s_1<\cdots <s_r\} $ of the set of ranks in $\overline{P}$.   Then for each filling $F$ of the ribbon
shape $\Rib (S)$ 
with the elements of a basis for the matroid of $P$, the alternating sum
${\bar f}_{chain}(v_F)$ of maximal chains in $P^S$ is a cycle.  That is,  $d({\bar f}_{chain}(v_F))=0$.  In particular, 
each  element 
${\bar f}_{chain}(v_F)$ of $\mathcal{B}^S_{rib}(P)$  
is a cycle. 
\end{prop}

\begin{proof}
We will show for each $l=1,2,\dots ,r$ that
$d_l({\bar f}_{chain}(v_F))=0$.  We do this by grouping summands  in ${\bar f}_{chain}(v_F)$ in pairs that cancel with each other in  $d_l({\bar f}_{chain}(v_F))$.  

For each $s_l\in S$, consider the portion $d_l$ of the boundary map for $P^S$.
This map $d_l$ eliminates the $l$-th element from each maximal chain 
in $P^S$.  
Note that 
for  $j:= s_l \in S$, 
the $l$-th element in the  maximal chain 
$\mathrm{res}_S (f_{chain}(\sigma F))$ in $P^S$ appearing as a 
summand in ${\bar f}_{chain}(v_F)$
is also the 
$j$-th element of the maximal chain  $f_{chain}(\sigma F)$
in $P$. 

 By Lemma ~\ref{d-on-fillings}, we may  rewrite $d_l({\bar f}_{chain}(v_F))$ as 
$$({\rm res}_{S \setminus \{ s_l\} }\circ f_{chain})\left(\sum_{\sigma\in \Col_F} \sgn(\sigma )\, \sigma ( d_l^{fill} F)\right).  $$  Equation~\eqref{eqn:fchain-fbarchain} allows us to  further rewrite this as 
$$({\bar f}_{chain}\circ \tau )\left( \sum_{\sigma \in \Col_F}{\rm sgn}(\sigma )\,\sigma (d_l^{fill}  F)\right)$$
and prove that this equals 0 to deduce that 
$d_l({\bar f}_{chain}(v_F))=0$.

Observe  that 
the
boundary map $d_l^{fill}$ 
(as in Definition ~\ref{filling-boundary}) applied to $F$ 
moves the $(j+1)$-st box in the reading word order on the boxes of $\Rib (S)$   from its position directly above the $j$-th box 
to the position directly to its right instead.  
Importantly for what follows, observe that the transposition   $(j,j+1)$ belongs to $ \Col_{\Rib (S)}$ and also to  $\Row_{d_l^{fill} (\Rib (S))}$,  where 
$d_l^{fill}(\Rib (S)) := \Rib (S\setminus \{ s_l\} )$. 
For this latter claim, we are using the fact that  $d_l^{fill}(\Rib (S))$ denotes the shape obtained from $\Rib (S) $ by merging the $l$-th and $(l+1)$-st rows, thereby putting the $j$-th and $(j+1)$-st boxes into  a single row.

These features of the action of $(j,j+1)$ will allow us to 
pair up the fillings 
appearing in the sum 
$$\sum_{\sigma \in \Col_{\Rib (S)}}{\rm sgn} (\sigma )\, \sigma (d_l^{fill} F).$$
Then we show that each pair contributes 0 to this sum and hence  that 
\begin{equation}\label{eqn:cyc-prf-1}
(\tau \circ d_l^{fill}) \left(\sum_{\sigma \in \Col_{\Rib (S)}} {\rm sgn} (\sigma) \, \sigma F \right)=0.
\end{equation} 

Specifically, we pair the summand given by any 
filling $\sigma F$ having reading word $$(F_{\sigma (1)},F_{\sigma (2)},\dots ,F_{\sigma (n)})$$ with the summand given by the filling $(j,j+1)\cdot \sigma F$ having reading word $$(F_{\sigma (1)},\dots ,F_{\sigma (j-1)},F_{\sigma (j+1)},F_{\sigma (j)},F_{\sigma (j+2)},\dots ,F_{\sigma (n)}).$$ Note that these two summands have opposite signs since  they differ by a transposition.   Since $\tau $ maps these two fillings to the same tabloid, because $(j,j+1)$ belongs to ${\rm Row}_{d_l^{fill}(\Rib (S))}$, this establishes~\eqref{eqn:cyc-prf-1} 
by virtue of each such pair contributing 0 to the sum. 
Finally we apply 
${\bar f}_{chain}$ to both sides to deduce that $d_l$ sends our proposed cycle ${\bar f}_{chain}(v_F)$ to 0.

Since this holds 
for $l=1,2,\dots ,|S|$ and 
our boundary map $d$ satisfies 
$d= \sum_{l=1}^{|S|} d_l$, it follows that 
$d({\bar f}_{chain}(v_F))=0$.
\end{proof}

Given how $\mathcal{W}_{rib}^S(P)$ is constructed, the next result follows immediately from Proposition ~\ref{do-have-cycles}.

\begin{cor}
The elements of $\mathcal{W}_{rib}^S(P)$ are cycles.
\end{cor}

\begin{thm}\label{thm:homology-basis-EL-lab}
Given a geometric lattice of rank $n$ with an atom ordering $a_1<a_2<\cdots < a_k$ giving rise to a  minimal labeling, and a subset $S=\{s_1<s_2<\cdots<s_k\}$ of ranks, consider  
the set $\mathcal{B}^S_{rib}(P)$ of cycles  ${\bar f}_{chain}(v_{T})$ given by the standard fillings $T$ of ribbon shape $\Rib(s_1,s_2-s_2,\dots ,n-s_k)$ with $\minNBC$ bases.
Then $\mathcal{B}^S_{rib}(P)$ is  a basis for $\beta_S(P)$.
\end{thm}

\begin{proof}
In Proposition ~\ref{do-have-cycles}, we prove that the elements of $\mathcal{B}^S_{rib}(P)$ are cycles.  To prove they are a basis for homology, we will rely on 
two facts observed by  Wachs in  ~\cite{Wachs-cohomology}, the latter of which appears as Proposition 1.1 in that paper. 

The first fact is  that the homology facets $F_1,\dots ,F_m$  of a shelling of a simplicial complex  give rise to a cohomology basis
$C_1,\dots ,C_m$, by letting $C_i$ be the dual of the homology facet $F_i$ for $i=1,\dots ,m$.
The second result asserts that any collection $H_1,\dots ,H_m$ of cycles with the property that $\langle H_i,C_j\rangle = \delta_{i,j}$ for such a shelling-based cohomology basis is itself a homology basis.  

We now explain why  the cycles in 
$\mathcal{B}^S_{rib}(P)$ have this relationship with the cohomology basis for $P^S$ comprised of the duals to the homology facets described in Proposition ~\ref{homology-facets-using-standard-ribbons}.
Each homology facet $F_i$ corresponds to a standard filling $T_i$ of our ribbon shape with atoms of $P$ comprising an $\minNBC$ basis, by Proposition ~\ref{homology-facets-using-standard-ribbons}.  None of the  summands  in $v_{T_i}$ other than the term given by $T_i$ itself  are standard fillings of $\Rib(S)$ since the others are each obtained by permuting elements in the columns of our ribbon filling.  
This shows that  
$\langle {\bar f}_{chain}(v_{T_i}),C_{T_j}\rangle = 0$ for $i\ne j$, while  
$\langle {\bar f}_{chain}(v_{T_i}),C_{T_i}) \rangle =1$ 
by virtue of how ${\bar f}_{chain}(v_{T_i})$ is constructed.  Thus, we may use \cite[Proposition 1.1]{Wachs-cohomology} to deduce that $\mathcal{B}^S_{rib}(P)$ is a homology basis.
\end{proof}

From this we derive an analogous result for rank-selected Whitney homology. 

\begin{thm}\label{thm:min-label-gives-NBC-homology-basis}
Given a minimal labeling for a geometric lattice $P$, the set $\mathcal{W}^S_{rib}(P)$ is a homology basis  
for $W\!H_S(P)$.
\end{thm}

\begin{proof}
Our method of  proof 
that $\mathcal{B}^S_{rib}(P)$ is a basis for 
$\beta_S(P)$ carries over to show likewise that  
$\mathcal{W}^S_{rib}(P)$ is a homology basis for 
$W\!H_S(P)$ by virtue of how $W\!H_S(P)$ is defined, noting that every interval $[\hat{0},u]$  
in a geometric lattice $P$ is itself a geometric lattice whose atoms  are a subset of the set of atoms of $P$.
\end{proof}

\begin{ex}\label{ex:ribbon-basis-WhS}  Consider the rank set $S=\{2,4\}$ and the rank 4 element  $u=|128|67|45|3|$ in $\Pi_8$.  A ribbon basis element for the homology of $W\!H_S(\Pi_8)$ coming from the interval $[\hat 0, u]^S$ 
is given by the following standard $\minNBC$ filling of $\Rib(2,4)$: 
\[F=\tableau{&45 & 18\\ 12 & 67}.\] 
This yields  the polytabloid 
 $v_F=\left\{\tableau{&45 & 18\\ 12 & 67}\right\}-\left\{ \tableau{&67 & 18\\ 12 & 45}\right\}$.  
 Then $\bar{f}_{chain}$ maps $v_F$ to the difference of chains 
 \[(\hat 0<|12|67|<|128|67|45|=u)- (\hat 0<|12|45|<|128|67|45|=u);\]
 it is a cycle because the boundary map $d_1$ sends the linear combination to 0.
\end{ex}

\subsection{Ribbon generators are $G$-equivariant and sit inside  Boolean sublattices} 

Next we examine  how our ribbon bases interact with a group action, focusing especially on the case of the partition lattice $\Pi_n$.
The atoms in $\Pi_n$ are in bijection with the  ordered pairs $(i,j)$ with $1\le i < j\le n$, allowing us to use these ordered pairs as our ribbon filling labels.
The $\fS_n$-action on $\{ 1,2,\dots ,n\} $ by permuting values also gives a permutation action on the set of ordered pairs $(i,j)$ with $1\le i< j \le n$.

Next is a property of Specht modules that $\beta_S(P)$ and $W\!H_S(P)$ share when $P$ is a geometric lattice with an action of a  group $G$, so that $\beta_S(P$ and $W\!H_S(P)$ are both $G$-modules.  

\begin{prop}\label{Specht-like-action}
Let $P$ by any geometric lattice which is a  $G$-poset.
For each $g\in G$, 
the elements ${\bar f}_{chain}(v_F)$ of $\mathcal{B}_{rib}^S(P)$ (resp.\! $\mathcal{W}_{rib}^S(P))$ satisfy
$g ({\bar f}_{chain}(v_F)) = {\bar f}_{chain}(v_{g F})$ for each filling $F$ of $\Rib (S)$.
\end{prop}

\begin{proof}
This follows directly 
from the  definition for  ${\bar f}_{chain}(v_{gF})$.  
\end{proof}

\begin{rk}
    In the case of $\Pi_n$, the $\fS_n$-action permuting the values $1,2,\dots ,n$ sends each   
    ribbon filling  $F$ giving rise to an element of our homology basis $\mathcal{B}^S_{rib}(P)$ 
    to another ribbon filling $\sigma F$ for each $\sigma \in \fS_n$.  However,  ${\bar f}_{chain}(v_{\sigma F})$
        often is not in  $\mathcal{B}^S_{rib}(P)$.  
    One reason for this is that the polytabloid $v_{\sigma F}$ 
    might not have a standard filling among its summands.   
    A second issue is that the reading word for $\sigma F$ might not be 
    the label sequence of any maximal chain in $\Pi_n$ 
    given by our minimal labeling, i.e. might not be an 
    $\minNBC $ basis. (Let $F$ be the filling of Example~\ref{ex:hom-Pi-1}, and consider  the filling $\sigma F$ for $\sigma=(3,6)$.  Then the polytabloid $v_{\sigma F}$ is such an example.)  
    \end{rk}
    
Our next result further describes  the structure of our ribbon 
homology basis elements  for all geometric lattices.   It helps explain why geometric lattices are particularly well suited to the construction of ribbon bases that will turn out to behave quite similarly to  the polytabloid bases for Specht modules.  Essentially, it shows that each basis element lives inside a Boolean sublattice. 
First is a lemma we will need.

\begin{lem}\label{swap-chain}
    Consider any geometric lattice $P$ of rank $r$, any  minimal labeling $\lambda $ for $P$, and any maximal chain 
    $\hat{0}\prec u_1\prec u_2 \prec \cdots \prec u_{r-1}\prec \hat{1}$ in $P$.  Let $(\lambda (\hat{0},u_1),\lambda (u_1,u_2),\dots ,\lambda (u_{r-1},\hat{1}))$ denote its label sequence under the minimal labeling given by a fixed atom ordering.  Then  the $r!$ sequences obtained by the $\fS_r$-action permuting the positions of the labels in
    $$(\lambda (\hat{0},u_1),\lambda (u_1,u_2),\dots ,\lambda (u_{r-1},\hat{1}))$$ map to distinct maximal chains in $P$ via the map $$(a_{i_1},a_{i_2},\dots ,a_{i_r})\rightarrow (\hat{0}\prec a_{i_1}\prec a_{i_1}\vee a_{i_2}\prec \cdots \prec a_{i_1}\vee\cdots \vee a_{i_r})$$
\end{lem}

\begin{proof}
This is immediate from the fact that $\{ a_{i_1},\dots ,a_{i_r}\} $ is an independent set, implying  that its  distinct subsets have distinct joins.  
\end{proof}

\begin{prop}
For $P$ a geometric lattice and ${\bar f}_{chain}(v_F)$ any element of the basis $\mathcal{B}_{rib}^S(P)$, the summands appearing in ${\bar f}_{chain}(v_F)$ are distinct maximal chains in $P^S$.
\end{prop}

\begin{proof}
Distinct summands are given by distinct permutations  in $\Col_{\Rib (S)}$.  Lemma~\ref{swap-chain} ensures   that the map $f_{chain}$ sends $\sigma (F)$ and $\sigma ' (F)$ to distinct maximal chains in $P$ whenever $\sigma $ and $\sigma ' $ are distinct permutations in $\fS_r$.  We claim that 
the map 
${\rm res}_S $ then sends  $f_{chain}(\sigma (F))$ and $f_{chain}(\sigma ' (F))$ to distinct maximal chains in $P^S$,  provided that $\sigma $ and $\sigma ' $ are distinct permutations in  $\Col_{\Rib (S)}$.  This claim follows because the column group for $\Rib (S)$ intersects the  row group for $\Rib (S)$ in only the identity permutation.  Combined, this yields the desired distinctness of summands in 
${\bar f}_{chain}(v_F)$.  
\end{proof} 

\subsection{Isomorphism of $W\!H_{\{ 1,2,\dots ,i\} }(P)$ with the $i$-th graded piece of the Orlik-Solomon algebra}
\label{OS-connection}
    It is natural to ask in the case of rank set $S=\{ 1,2,\dots ,i\} $  
    for any $i$ satisfying $1\le i\le {\rm rk}(P)$   
    how our 
    generators are related to the generators of the $i$-th graded piece of the Orlik-Solomon algebra.  We will be especially interested in the case of the type A braid arrangement.  In that case, the $i$-th cohomology group of the complement of this complex hyperplane arrangement is well-known to be $\fS_n$-equivariantly isomorphic to   $W\!H_{\{ 1,2,\dots ,i\}}(\Pi_n)$ (see  Theorem 1.7 of \cite{Sundaram} for this isomorphism).  Indeed, such an isomorphism (without the $\fS_n$-equivariant structure) holds more  generally for the intersection poset of any complex hyperplane arrangement, and it holds $G$-equivariantly for any such intersection poset having  automorphism group $G$.

    Our homology basis for $W\!H_{\{ 1,2,\dots ,i\} }(P)$  
    for $1\le i\le {\rm rk}(P)$ 
    is indexed by  the  NBC independent sets of size $i$ in the matroid associated to the geometric lattice $P$;
    this is explained in Example ~\ref{OS-example}.
    The Orlik-Solomon algebra of a geometric lattice has a monomial basis for its $i$-th graded piece that is also indexed by the NBC independent sets of size $i$.  
    
    Recall that the Orlik-Solomon algebra of a matroid   is an exterior algebra  $\mathcal{E}$ generated by the atoms of the geometric lattice, with relations given by the circuits of the matroid.  Let $e_a$ denote the generator indexed by the atom $a$, and let $*$ denote the product in this algebra.  Recall that a circuit in a matroid  is a dependent set of atoms, namely a set  $\{ a_{i_1},\dots a_{i_s}\} $ of atoms with $\rank (a_{i_1}\vee \cdots \vee a_{i_s}) <s$,  with the further minimality property  that removing  any one of its  elements gives an independent set.  
    For each circuit $\{ a_{i_1},\dots ,a_{i_s}\} $ with $a_{i_1} < \cdots < a_{i_s}$, there is a relation
    $$\sum_{l=1}^r (-1)^{i-1} e_{a_{i_1}}*\cdots *\hat{e}_{a_{i_l}} *\cdots *e_{a_{i_s}} = 0.$$  
    These relations give a way of expressing  each broken circuit (i.e. each independent set obtainable  by removing from a   circuit its smallest element)  of size $i$ as a linear combination of NBC independent sets of size $i$, namely independent sets containing no broken circuits.  Thus,  the NBC independent sets of size $i$ generate the $i$-th graded piece of the Orlik-Solomon algebra. 
   
    The fact that NBC independent sets of size $i$ index both bases provides  a natural correspondence for any geometric lattice $P$  between 
    the  generators $$\{ {\bar f}_{chain}(v_F) :  F \text{ is standard filling with an NBC set} \} $$ in the basis $\mathcal{W}^{1,2,\dots ,i}_{rib}(P)$ and the monomials 
    serving  as a basis for  the $i$-th graded piece of the Orlik-Solomon algebra of this geometric lattice.  In fact, there is a natural map  from the $i$-th graded piece of the Orlik-Solomon algebra given by $P$  to $W\!H_{\{ 1,2,\dots ,i\} }(P)$ sending any monomial $e_{a_{j_1}}\cdots e_{a_{j_i}}$ of degree $i$ to $\bar{f}_{chain}(v_F)$ for the filling $F$ having reading word $(a_{j_1},\dots ,a_{j_i})$.  
    Call this the OS-to-Whitney correspondence.
    
    Now consider the case of a geometric lattice which is a $G$-poset. 
    By Proposition ~\ref{Specht-like-action},  
    $g  ({\bar f}_{chain}(v_F)) = {\bar f}_{chain}(v_{gF})$ for each $g\in G$  and each ribbon filling $F$ with entries an independent set of size $i$.
    The monomials $e_F:= e_{F_1}*e_{F_2}*\cdots *e_{F_i}$ of degree $i$ in the Orlik-Solomon algebra 
    likewise satisfy $g(e_F) = e_{gF}$.  
    Thus, the OS-to-Whitney correspondence 
    is  $G$-equivariant.
    
For any geometric lattice $P$, 
the Orlik-Solomon algebra  relations given by matroid circuits  are sent by  the  OS-to-Whitney correspondence
to relations which  hold  
in $W\!H_{\{ 1,2,\dots ,i\} }(P)$.  This follows from the next result, which was essentially  proven (in somewhat different language) in Section 3 of \cite{Orlik-Solomon}.   
We include our own proof below.

\begin{prop}\label{relations-to-relations}
For each circuit $\{ a_{i_1} ,\dots ,a_{i_j}\} $ in our geometric lattice $P$ with $a_{i_1}>a_{i_2}>\cdots > a_{i_j}$,
there is  a relation 
$$\sum_{s=1}^j (-1)^s {\bar f}_{chain} (v_{F(a_{i_1},\dots ,\hat{a}_{i_s},\dots ,a_{i_j})})= 0$$
in $W\!H_{\{ 1,2,\dots ,j-1\}} (P)$ where $F(a_{i_1},\dots ,\hat{a}_{i_s},\dots ,a_{i_j})$ denotes the filling of the ribbon shape consisting of a single column with $j-1$ boxes with reading word  $(a_{i_1},\dots ,\hat{a}_{i_s},\dots ,a_{i_j})$.  
\end{prop}

\begin{proof}
It helps to note that for any circuit 
$\{ a_{i_1},\dots ,a_{i_j}\} $, the map 
$$(a_{i_1},\dots ,\hat{a}_{i_s},\dots ,a_{i_j})\rightarrow 
< \cdots  < a_{i_1}\vee \cdots \vee a_{i_{s-1}}\vee a_{i_{s+1}} \vee \cdots \vee a_{i_j}$$ yields a chain whose largest element $a_{i_1}\vee \cdots \vee a_{i_{s-1}} \vee a_{i_{s+1}} \vee \cdots \vee a_{i_j}$ is the same for all choices of $s$.  Thus, 
the map $\overline{f}_{chain}$ effectively forgets what entry is in the top box in $\Rib (\{ 1,2,\dots , j-1\} ) = \Rib (1,1,\dots ,1)$.
In other words,  the map
$$(a_{k_1}, \dots ,a_{k_{j-1}})\rightarrow a_{k_1} < a_{k_1}\vee a_{k_2} < \cdots < a_{k_1}\vee \cdots  \vee a_{k_{j-1}}$$ 
only depends on $(a_{k_1},\dots ,a_{k_{j-2}})$, provided that  $\{ a_{k_1},\dots ,a_{k_{j-1}}\} \subset \{ a_{i_1} , \dots ,a_{i_j}\}  $. 

Using this forgetful property, we will show for the  alternating sum giving rise to our proposed relation that each chain appearing in the alternating sum  appears exactly  twice with opposite signs.  The first important fact we will use to check this  is  that the map sending $(a_{k_1},\dots ,a_{k_{j-1}})$ to $a_{k_1} < a_{k_1}\vee a_{k_2} < \cdots < a_{k_1} \vee a_{k_2}\vee \cdots \vee a_{k_{j-1}}$ is a 2 to 1 map, by virtue of there being exactly  two elements $a$ in the circuit with the property that $a \vee (a_{k_1}\vee \cdots \vee a_{k_{j-2}}) = a_{i_1}\vee \cdots \vee a_{i_j}$, hence two choices for $a_{k_{j-1}}$.  These two choices  
are the two  elements $a_{i_s}$ and $a_{i_t}$
in $\{ a_{i_1},\dots ,a_{i_j}\} \setminus \{ a_{k_1},\dots ,a_{k_{j-2}}\} $.  These two choices give rise to a pair of chains that  will  cancel each other out.  

What remains to show is that these chains $a_{k_1} < \cdots < a_{k_1}\vee \cdots \vee a_{k_{j-2}} < (a_{k_1} \vee \cdots \vee a_{k_{j-2}})\vee a_{i_s} $ and $ a_{k_1} < \cdots < a_{k_1}\vee \cdots \vee a_{k_{j-2}} <  (a_{k_1}\vee \cdots \vee a_{k_{j-2}})\vee a_{i_t}$ appear with opposite signs.

Let  us assume 
$s<t$.  Note that the  sign for the chain with  $a_{k_{j-1}} = a_{i_s}$  (resp. $a_{k_{j-1}}= a_{i_t}$)  is $(-1)^{(s-1)+({\rm revinv}(\pi_s))}$ (resp. $(-1)^{(t-1)+({\rm revinv}(\pi_t))}$)  where we define ${\rm revinv}(\pi_s)$ to be  the number of adjacent transpositions one must apply to $\pi_s := (a_{k_1},\dots ,a_{k_{j-2}}, a_{i_s})$ to obtain a subsequence of $(a_{i_1},\dots ,a_{i_j})$.    Thus, showing we get opposite signs amounts to showing  that 
$$(-1)^{(s-1)+({\rm revinv}(\pi_s))} \text{ and } (-1)^{(t-1)+({\rm revinv}(\pi_t))}$$ have opposite signs.  But $(t-1)-(s-1)=t-s$ while 
$({\rm revinv}(\pi_s))-({\rm revinv}(\pi_t))= t-s-1$, as we now explain.   The ${\rm revinv}$ discrepancy  calculation above can be verified by  
comparing inversion sets for $\pi_s$ and $\pi_t$ and getting $t-s-1$ discrepancies resulting from  the $t-s-1$ values $a_{i_l}$ for $l=s+1,\dots ,t-1$ that each satisfy $a_{i_s} > a_{i_l} > a_{i_t}$;  each of these letters  forms an inversion with $a_{i_s}$ when $a_{i_s}$  is put in the last position  but does not form an inversion with $a_{i_t}$ when it is in the last position.  
The result follows.
\end{proof}

These relations may easily be generalized 
to relations which suffice to straighten any ${\bar f}_{chain}(v_F)$ in which the reading word for $F$ contains a broken circuit, and indeed results in Section 3 of 
\cite{Orlik-Solomon} imply this more general statement.  For example, given a dependent set $\{ a_4 > a_3 > a_2 > a_1\} $ of atoms in which $\{ a_4,a_2,a_1 \} $ is a circuit while $\{ a_4,a_3,a_2\} $ is an independent set, one may use reasoning as in the  proof of Proposition ~\ref{relations-to-relations} above to verify  the relation
$${\bar f}_{chain}(v_{F(a_4,a_3,a_2)}) - {\bar f}_{chain}(v_{F(a_4,a_3,a_1)}) + {\bar f}_{chain}(v_{F(a_2,a_3,a_1)}) = 0.$$  More  generally, for any $\{ a_{i_1}> a_{i_2} > \cdots a_{i_s}\} $  containing a circuit $ \{ a_{i_{j_1}} > \cdots > a_{i_{j_t}}\} $, there is a relation obtained by taking an alternating sum over the ways to omit one element belonging to the specified circuit  while holding  fixed  the positions of all the atoms not involved in the circuit and arranging the atoms from  the circuit in descending order in their allotted positions.  

Thus, the OS-to-Whitney correspondence
sends every relation in the Orlik-Solomon algebra to a relation among our ribbon generators, by Proposition ~\ref{relations-to-relations} and its slight generalization discussed above.
We proved in Theorem ~\ref{thm:homology-basis-EL-lab} that the elements of our ribbon basis are linearly independent, guaranteeing that there cannot be  any further relations in 
$W\!H_{\{ 1,2,\dots ,i\}} (P)$ not already implied by these relations.
From this one  may deduce 
the following:

\begin{thm}\label{OS-isomorphism}
Given any geometric lattice $P$, the map sending the monomial generators of the $i$-th graded piece of the Orlik-Solomon algebra of $P$ (namely the degree $i$ monomials consisting of NBC sets with variables listed in decreasing order)   to the corresponding elements of 
$\mathcal{W}_{\{ 1,2,\dots ,i\} }(P)$ 
is a vector space isomorphism.  Moreover, if $G$ is an   automorphism group of $P$, then this isomorphism is $G$-equivariant.    
\end{thm}

After proving 
Theorem ~\ref{OS-isomorphism}, we 
discovered that it is very closely related to Theorem 3.7 in \cite{Orlik-Solomon}.  Our  shellability viewpoint significantly simplifies one part of the proof from \cite{Orlik-Solomon}, 
namely their argument that (the Orlik-Solomon version of)
the ribbon generators cannot satisfy any  relations not implied by those appearing in Proposition ~\ref{relations-to-relations} and its slight generalization discussed afterwards.  

\subsection{Useful property of  Young symmetrizers applied to the ribbon basis for $\Pi_n$}
\label{Young-section}

\begin{lem}\label{lem:YoungSym-annihilate-ribbon}  
Let $\Rib (S)$  
be a ribbon  corresponding to some rank set $S$, and let $F$ be a filling of $\Rib (S)$ that produces  the homology basis element $v_{F}$.   Suppose the following three 
conditions hold for  four distinct letters $b,c,d,e$.
\begin{enumerate}
   \item $F$ has two boxes in the same column, such that one box has the label $\{b,c\} $  and the other with the label 
$\{ d,e\} $;
\item 
none of the letters $b,c,d,e$ appear within the labels on any other boxes of $F$;
\item there is a standard tableau $T$ of shape $\lambda\vdash n$ such that all four letters $b,c,d,e$ appear in the first row of $T$.
\end{enumerate}
 Then 
 the Young symmetrizer $b_T a_T$ sends $v_{F}$ to 0.
\end{lem}

\begin{proof}   
Condition (3) guarantees that  the row symmetrizer $a_T$ has a right factor $(1+(b,d)(c,e))$, since the product of transpositions $(b,d)(c,e)$ fixes the first row of $T$.  Explicitly, let $H$ be the subgroup of the row stabilizer group $\mathrm{Row}_T$ generated by $(b,d)(c,e)$, and let $\mathrm{Row}_T/H$ be a complete set of distinct left coset representatives (including the identity element) of $H$ in $\mathrm{Row}_T$. Then we have 
\[ a_T=\left(\sum_{\sigma\in \mathrm{Row}_T/H} \sigma\right) (1+(b,d)(c,e)).
\]
Conditions (1) and (2) guarantee that $(b,d)(c,e)$ swaps the labels $\{b,c\}$ and $\{d,e\}$ in $F$ while leaving all other labels in $F$ unchanged, and hence sends  $v_{F}$ to its negative $-v_{F}$.
Hence $(1+(b,d)(c,e))$ applied to the homology basis element $v_{F}$ gives 0.   It follows from the above factorization of $a_T$ that $a_T v_{F}=0$, and hence $b_T a_T v_{F}=0$, as claimed.
See Figure~\ref{fig:symmetrizer-annihilation1}.
\end{proof}

\begin{figure}
    \centering
    \[ \text{Let } T=\tableau{1 & 2 & 3 &4},\]
    
    \vskip 5pt
    Then $a_T=\sum_{\sigma\in S_4} \sigma= w\cdot (1+(13)(24))  \text{ for some $w$ in the group algebra}.$
\[F=\tableau{\underline{12}\\ \underline{34}} \qquad 
(13)(24) F= \tableau{\underline{34}\\ \underline{12}} \qquad\qquad 
v_{F}= \bigg\{ \tableau{\underline{12} \\ \underline{34}}\bigg\}-\bigg\{ \tableau{\underline{34}\\ \underline{12}}\bigg\} \qquad
(13)(24)\cdot v_{F}= - v_{F}\]
Hence $a_T v_{F} = w(1+(13)(24)) v_{F}=0.$
    \caption{A small example of a Young symmetrizer annihilating a ribbon homology element.  The ribbon entries  are an NBC independent set in $\Pi_4$.} 
    \label{fig:symmetrizer-annihilation1}
\end{figure}

\section{Stability results  for rank-selected Whitney homology and rank-selected homology in $\Pi_n$}\label{Pi-n-section} 

In this section, we prove the sharp stability bound of 
$4\max S - |S| + 1$  of Conjecture 11.3 in \cite{Hersh-Reiner} for the rank-selected homology of the partition lattice, namely for $\beta_S(\Pi_n)$.    In order to do this, we prove the sharp stability bound of $4\max S - |S| + 1$ for $W\!H_S(\Pi_n)$, then invoke Corollary ~\ref{cor:WHPiS-bound-equals-betaS-bound} below to deduce Conjecture 11.3 of \cite{Hersh-Reiner}. 
In Section ~\ref{sharp-section}, we will  prove  that 
$W\!H_S(\Pi_n)$ cannot stabilize earlier than $4\max S - |S| +1$. Then we lay the groundwork for proving stability and also verify our desired stability bound for $W\!H_S(\Pi_n)$ in special cases  such as rank sets $S$ with $|S|\le 4$, as a warm-up for the general case.   Finally, this   sharp stability bound for 
$W\!H_S(\Pi_n)$ is proven  in Section ~\ref{H-R-Conj-proved}.
We will rely upon  the fact that $\Pi_n$ is a geometric lattice of rank $n-1$, enabling the usage of our ribbon bases for geometric lattices.

Exactly as for the Boolean lattice, Proposition~\ref{prop:alt-sum-expression-S} and Proposition~\ref{prop:WHS-bound-equals-betaS-bound-anyP-AGAIN} allow us to deduce the following equivalence.
\begin{cor}\label{cor:WHPiS-bound-equals-betaS-bound}
$W\!H_S(\Pi_n)$ stabilizes sharply at $4\max S - |S| + 1$ for all $S$ if and only if $\beta_S(\Pi_n)$ stabilizes at $4\max S -|S| + 1$ for all $S$.
\end{cor}
\begin{proof}  Take $P_n=\Pi_n$ and $f(S)=4\max S -|S|+1$ 
in Proposition~\ref{prop:WHS-bound-equals-betaS-bound-anyP-AGAIN}. \end{proof}
In view of this corollary, the remainder of the section will focus on stability of $W\!H_S(\Pi_n) $.
The sharp stability bound of $4\max S - |S| + 1$ for $W\!H_S(\Pi_n)$ was proved  in the special case of  consecutive rank sets $S=\{1,2\ldots, i\}$ in \cite[Corollary 5.4]{Hersh-Reiner}, thereby establishing  Conjecture 11.3 of \cite{Hersh-Reiner}  
for these  rank sets. 

Next we 
handle singleton sets $S=\{i\}$ before turning to general $S$.  Our proof for $S=\{ i\} $ also serves as a warm-up for our techniques, and provides motivation for Definition~\ref{def:essential-part} in the next section. 

\begin{lem}\label{lem:sharp-bound-singletonS} For $S=\{i\}$, $W\!H_S (\Pi_n)$ stabilizes sharply at $n=4i$, and hence so does $\beta_S(\Pi_n)$.
\end{lem}
\begin{proof} 
The last statement is a consequence of the fact that $\beta_{\{i\}}(\Pi_n)=W\!H_{\{i\}} -\one_{\fS_n}$.

For $S = \{ i\} $, the Whitney homology is a sum of permutation modules, appearing as induced wreath products of trivial representations.  First we observe that if $n\ge 2i$, $\Pi_n$ contains the partition with $i$ blocks of size 2 and $n-2i$  blocks of size 1 at rank $i$.  The $\fS_n$-orbit of these partitions has Frobenius characteristic $h_i[h_2] h_{n-2i}$, which stabilizes sharply  at $4i$ by Lemma~\ref{plethysm-first-row-bound}, since the trivial representation must occur in the permutation module with Frobenius characteristic $h_i[h_2]$. Hence we may assume $n\ge 2i$.  A calculation will show that the contribution of the remaining permutation modules in the Whitney homology stabilizes for  $n\ge 4i$.

More precisely, in terms of Frobenius characteristics, for $S=\{i\}$ and $n\ge 2i$ we have the  expression 
\[\ch W\!H_S(\Pi_n)=\sum_{r=1}^i \sum_{\substack{\nu\vdash n \\ \ell(\nu)=n-i}} h_{n-i-r}[h_1]\, h_{m_2}[h_2]\cdots h_{m_j}[h_j] \cdots,\]
for integer partitions $\nu$ of $n$ with $n-i-r $ parts equal to 1, and $m_j$ parts equal to $j$ if $j\ge 2$.   In particular, $\sum_{j\ge 2} j m_j=i+r$, and $\sum_{j\ge 2} m_j=r$. We have $i\ge r$ because $i+r=\sum_{j\ge 2} j m_j\ge 2 (\sum_{j\ge 2} m_j  ) =2r$.

Thus the Frobenius characteristic of $W\!H_S(\Pi_n)$ consists of summands of the form $h_{n-i-r} f$ where $\deg f= i+r$.  It is clear from the preceding paragraph that the largest first row that can appear in the Schur expansion of $f=h_{m_2}[h_2]h_{m_3}[h_3]\cdots$ is its degree $i+r$. 

By Part (2) of Lemma~\ref{lem:quasi-free-implies-increasing-mult}, the summand corresponding to a fixed $r$  stabilizes  sharply for $n-i-r\ge i+r$, i.e. for $n\ge 2(i+r)$. But  $2i+2r\le 4i$ and moreover, the bound $4i$ is achieved when $r=i$  and $\nu=(2^i, 1^{n-2i})$.  
Hence we obtain the sharp stability bound of $4i=4\max S-|S| +1$ in this case.
\end{proof}

\subsection{$W\!H_S(\Pi_n)$ cannot stabilize earlier than $4\max S - |S| + 1$}\label{sharp-section}

For a set partition $u\in \Pi_n$ and an integer partition $\lambda$ of $n$, we say $u$ has 
\emph{type} $\lambda$ if the block sizes of $u$ are given by the parts $\lambda_1, \lambda_2, \ldots$ of $\lambda$.  If $\lambda$ has $m_i$ parts equal to $i$, we will sometimes refer to the type of $u$ as $(m_1, m_2, \ldots).$

Again, let $S$ be any subset of nontrivial ranks. 
For the partition lattice $\Pi_n$, we  have the following finer decomposition of the $\fS_n$-module $W\!H_S(\Pi_n)$, according to the type of a set partition.

\begin{defn}\label{def:WHS-submodules}
Define $W\!H_{S,\lambda }(\Pi_n)$ to be the $\fS_n$-submodule  of $W\!H_S(\Pi_n)$ 
obtained by summing over only  those set partitions $u$ of rank $\max S$ and type $\lambda $.  Then 
\begin{equation}\label{def:WH-S-lambda} W\!H_S(\Pi_n)=\bigoplus_{\lambda\vdash n} W\!H_{S, \lambda}(\Pi_n)
\end{equation}
and 
\begin{equation}\label{eqn:WH-S-lambda-u}
W\!H_{S, \lambda}(\Pi_n)=\bigoplus_{\substack{u \text{ of type } \lambda\\ \rank(u)=\max S}}
\tilde{H}_{|S|-2}({[\hat 0, u]}^S). 
\end{equation}
\end{defn}
\begin{defn}\label{def:essential-part}
Now let $k\le n$ and let  $\mu$ be an integer partition of $k$,  having no part of size 1.  Then $(\mu, 1^{n-k})$ is a partition of $n$ with $n-k$ parts equal to 1.  Define the \emph{essential part} of the $\fS_n$-module 
$W\!H_{S, \, (\mu, 1^{n-k} ) }(\Pi_n) $ to be the unique 
$\fS_k$-module $\widehat{W\!H}_{S, \, (\mu, 1^{n-k} ) }(\Pi_n)$  satisfying 
$$W\!H_{S,\, (\mu, 1^{n-k} ) }(\Pi_n) = \Ind_{\fS_{n-k}\times \fS_{k}}^{\fS_n}\left({\one}_{\fS_{n-k}} \otimes \widehat{W\!H}_{S, \, (\mu, 1^{n-k} ) }(\Pi_n)\right).$$ 
Thus one has an isomorphism of  $\fS_k$-modules $\widehat{W\!H}_{S, \, (\mu, 1^{n-k} ) }(\Pi_n)\simeq W\!H_{S,\,\mu}(\Pi_k)$.
\end{defn}
\begin{rk}\label{WH-is-quasi-free}
Note that by  definition, each component 
$\{ W\!H_{S, \lambda}(\Pi_n) \} $ of the rank-selected Whitney homology $\{ W\!H_S(\Pi_n)\} $ is a quasi-freely generated FI-module. 
\end{rk}
This definition is motivated by the following consequence, which we will use  frequently. 
\begin{lem}\label{lem:Stab-bd-from-ess-part}  Let  $\mu$ be a partition of size $|\mu|$, 
with no part of size 1.  Then 
$W\!H_{S, \, (\mu, 1^{n-|\mu|})}(\Pi_n) $ has sharp representation stability at  
\[   n =  \max \{|\lambda|+\lambda_1:\cS^\lambda \text{ appears in } \widehat{W\!H}_{S, \, (\mu, 1^{n-|\mu|})}(\Pi_n)\}.    \] 
\end{lem}
\begin{proof} Passing to Frobenius characteristics, we have 
\[\ch W\!H_{S, (\mu, 1^{n-|\mu|})}(\Pi_n)=h_{n-|\mu|}\cdot \ch \widehat{W\!H}_{S,\, (\mu, 1^{n-|\mu|})}(\Pi_n).\]
The claim now follows from 
Lemma~\ref{lem:quasi-free-implies-increasing-mult},
noting that $|\mu|=|\lambda|$ since the essential part $\widehat{W\!H}_{S,\, (\mu, 1^{n-|\mu|})}(\Pi_n)$ is an $\fS_{|\mu|}$-module.  
\end{proof}

\begin{prop}\label{rank-selection-first-row-bound}
There exists an occurrence of the Specht module $\cS^{\lambda }$ appearing in $\widehat{W\!H_S}(\Pi_n)$ for a partition $\lambda$ having $\lambda_1 = 2\max S - |S| + 1$.
\end{prop}

\begin{proof}  Let $S=\{s_1<s_2<\cdots<s_r=i\}$. 
We consider those elements $u$ of rank $i= \max S$ which  consist of $i$ blocks of size $2$ and $n-2i$ blocks of size 1.  Restricting to the sum $\bigoplus_{\mathrm{type}(u) = (2^i,1^{n-2i})} \tilde{H}_{r-2}([\hat{0},u]^{S\setminus\{i\}})$ over the  intervals with such elements $u$ of rank $i$ gives an $\fS_n$-submodule 
$W\!H_{S,\, (2^i,1^{n-2i})}(\Pi_n) $   of 
$W\!H_S(\Pi_n)$.  From Definition~\ref{def:essential-part},  the essential part  $\widehat{W\!H}_{S, \, (2^i, 1^{n-2i})}(\Pi_n)$ satisfies $$W\!H_{S, \, (2^i, 1^{n-2i})}(\Pi_n) = \Ind_{\fS_{n-2i}\times \fS_{2i}}^{\fS_n}\left({\one}_{\fS_{n-2i}} \otimes \widehat{W\!H}_{S,\, (2^i,  1^{n-2i})}(\Pi_n)\right).$$  
First observe that the stabilizer of a partition with $i$ parts each of size 2 is the wreath product  of symmetric groups 
$\fS_i[\fS_2]$.
From Definition~\ref{def:essential-part}, $\widehat{W\!H}_{S, \, (2^i,1^{n-2i})}(\Pi_n)$ is isomorphic to the $\fS_{2i}$-module ${W\!H}_{S,\, (2^i)}(\Pi_{2i})$.
We claim  that this in turn is 
the $\fS_{2i}$-module induced from the $\fS_i[\fS_2]$-module   given by the wreath product representation   
\begin{equation}\label{eqn:ess-part-char-all2s}
\chi (s_1, s_2-s_1,s_3-s_2,\dots, i-s_{r-1} )[{\one }_{\fS_2}]\end{equation}
where $\chi (a_1,a_2,\dots )$ denotes the representation given by the Specht module of ribbon shape $\Rib(a_1, a_2,\ldots)$.

The interval $[\hat 0, u]$ in $\Pi_n$ is isomorphic to the Boolean lattice 
$B_i$, where $i=\max S$. Hence
the outer expression $\chi (s_1,s_2-s_1,\dots, i-s_{r-1} )$ in~\eqref{eqn:ess-part-char-all2s} arises from  the $\fS_i$-module structure for the rank-selected homology $\beta_S(B_i)$, as explained in Lemma 
~\ref{Boolean-rank-selection}.    
Also $\fS_2$ acts trivially on $\Pi_2$, giving the inner term of ${\one}_{\fS_2}$ in~\eqref{eqn:ess-part-char-all2s}.
 Exactly as in the proof of Theorem~\ref{thm:Bool-stability}, 
the Frobenius characteristic of the $\fS_{2i}$-module induced from the $\fS_i[\fS_2]$-module~\eqref{eqn:ess-part-char-all2s} is  $f[h_2]$, where $f$ is the Frobenius characteristic of 
the rank-selected Boolean homology $\beta_S(B_i)$. Since $f$ has degree  $i=\max S$,  Lemma ~\ref{plethysm-first-row-bound} now gives the sharp first row length bound of $(\max S - |S| + 1) + \max S$, as desired.
\end{proof}

\begin{prop}\label{all-twos}
For each $S$, $W\!H_{S,\, (2^i, 1^{n-2i})}(\Pi_n)$ stabilizes sharply at $n=4\max S - |S| + 1$.
\end{prop}
\begin{proof}
This follows from the preceding Proposition~\ref{rank-selection-first-row-bound} and Lemma~\ref{lem:Stab-bd-from-ess-part}.
\end{proof}

\begin{cor}\label{rank-selection-sharpened-stability}
For each $S$, $W\!H_S(\Pi_n)$ cannot stabilize earlier than $n = 4\max S - |S| + 1$.  
\end{cor}
\begin{proof}The bound follows from Proposition~\ref{all-twos}, since $W\!H_{S, \, (2^i, 1^{n-2i})}(\Pi_n)$ is a submodule of $W\!H_S(\Pi_n)$, as we now explain.  
By definition (see Definition~\ref{def:essential-part}),  the rank-selected Whitney homology $W\!H_S $ of the partition lattice is a quasi-freely generated  
FI-module in which each component $W\!H_{S,\lambda }$ is itself a quasi-freely generated  FI-module.  This ensures, by  Lemma~\ref{lem:quasi-free-implies-increasing-mult} and 
Corollary~\ref{cor:FI-submodules},  that stability cannot occur earlier than where it occurs in any submodule $W\!H_{S,\lambda }$.
\end{proof}

\subsection{An inequality yielding the conjectured stability bound for small $|S|$}\label{small-S-section}

Some inequalities used in proving  stability bounds for  the 
Whitney homology $W\!H_S(\Pi_n)$ are collected in this section. Lemma ~\ref{lem:Tricia-magic-inequality}  gives an especially useful bound for small $|S|$. 

\begin{lem}\label{lem:equiv-Statement} 
Let $\mu$ be a partition of size at most $n$, with no part of size 1. Let $\ell(\mu)$ be the number of parts of $\mu$, 
and let $m_i(\mu)$ be the number of parts of $\mu$ that are equal to $i$.   Consider $W\!H_{S, \,( \mu, 1^{n-|\mu|})}(\Pi_n)$.
Then 
\begin{enumerate}
    \item $\max S=|\mu|-\ell(\mu)=\sum_{i\ge 2} (i-1) m_i(\mu)$.
    \item For each Specht module $\cS^{\lambda }$ appearing in the essential part  of $W\!H_{S,\,( \mu, 1^{n-|\mu|})}(\Pi_n)$,  we have 
$|\lambda | = 2\max S - \sum_{i\ge 3}  (i-2) m_i(\mu).$
\end{enumerate}
\end{lem}

\begin{proof}  Let $u$ be a set partition of type 
$(\mu, 1^{n-|\mu|})$, so that $u$ has $\ell(\mu) + n-|\mu|$ blocks.

\noindent 
Item (1) follows from the fact that 
$\max S=\rank\, u= n-(\ell(\mu)+n-|\mu|)=|\mu|-\ell(\mu)$, and thus 
$\max S=|\mu|-\sum_{i\ge 2} m_i(\mu)$.  
Now use  $|\mu|=\sum_{i\ge 2} i m_i(\mu)$.

\noindent
For item (2), since the essential part is an $\fS_{|\mu|}$-module, we have $|\lambda|=|\mu|$. from above. But also $|\mu|=\sum_{i \ge 2} i m_i(\mu)=\max S+\sum_{i \ge 2}  m_i(\mu) $, the last inequality following from Item (1).  Now rewrite  $\sum_{i \ge 2}  m_i(\mu)=\sum_{i \ge 2}  (i-1)m_i(\mu)-\sum_{i \ge 2} (i-2) m_i(\mu)$, and use Item (1) again to obtain  the result.
\end{proof}

\begin{lem}\label{lem:Stab-bound} 
To prove a stability upper bound of $4\max S - |S|+1$ for $W\!H_S (\Pi_n)$, it suffices to prove that $\lambda_1 \le 2\max S - |S| + 1 + \sum_{i\ge 3}(i-2)m_i(\mu)$ for each  $\cS^{\lambda }$ appearing in the essential part of  $W\!H_{S, \, (\mu, 1^{n-|\mu|})}(\Pi_n)$ for each set partition type 
$(\mu ,1^{n-|\mu |})$ at rank $\max S$.
\end{lem}
\begin{proof}
Using Item (2) of Lemma~\ref{lem:equiv-Statement}, we see that 
the stated first row bound is equivalent to the inequality
\[|\lambda | +\lambda_1\le  4\max S - |S| +1.   \]
By Lemma~\ref{lem:Stab-bd-from-ess-part}, this implies 
$W\!H_S (\Pi_n)$ is representation stable for $n\ge  4\max S - |S| +1$.
\end{proof}

The next lemma  allows our desired stability bound  to be deduced readily for small values of $|S|$.  Indeed we deduce the desired bound for $|S|\le 4 $ quite easily via this approach.  However, it seems increasingly complicated to work with this bound as $|S|$ grows, so instead we use a different approach in the next section to deduce our desired stability bound.  

\begin{lem}\label{lem:Tricia-magic-inequality}
Let $\mu$ be an integer partition with $m_i$ parts of size $i$. 
If  $2\sum_{i\ge 3}(i-2)m_i \ge |S|-1$, then the 
component ${W\!H}_{S,\, (\mu, 1^{n-|\mu|})}$ of $W\!H_{S}$  
satisfies the desired stability upper bound of 
$4\max S - |S| + 1$.
\end{lem}

\begin{proof}
Consider any $\cS^{\lambda }$ in the essential part  $\widehat{W\!H}_{S,\, ( \mu, 1^{n-|\mu|})}$ 
of $W\!H_S$.  
Then we 
have
$|\lambda | = |\mu|=2\max S - \sum_{i\ge 3} (i-2)m_i$ by Item (2) of Lemma ~\ref{lem:equiv-Statement}.  Hence the inequality in the statement is equivalent to 
\[4\max S -2|\lambda|\ge |S|-1. \]
Since $\lambda_1\le |\lambda|$, the lemma follows from Lemma~\ref{lem:Stab-bd-from-ess-part}.  \end{proof}

Below we handle cases with small $S$ to give some amount of  intuition for our approach. 

\begin{cor}\label{cor:sizeSis3}
If $|S| \le 3$, then the conjectured stability upper bound of $4\max S - |S| +1$ is always satisfied.
\end{cor}

\begin{proof}
For $|S|\le 3$, the previous lemma says that the desired stability upper bound is satisfied unless  $2\sum_{i\ge 3}(i-2)m_i  < |S| - 1\le 2$.  Thus, to complete the result we only need to handle the case with $2\sum_{i\ge 3}(i-2)m_i < 2$.  In this case we have $m_i=0 $ for all $i\ge 3$.  But we already handled the case when all parts have size 1 or 2 in Proposition~\ref{all-twos}.
\end{proof}

\begin{lem}\label{lem:sizeSis4}
If $|S| = 4$, then the conjectured stability upper bound of 
$4\max S - |S| + 1$ holds. 
\end{lem}

\begin{proof}  
We need only  consider the case when the inequality of  Lemma~\ref{lem:Tricia-magic-inequality} does not hold. That is, we  focus on the case when $2\sum_{i\ge 3} (i-2)m_i<3$.  This forces $m_3\le 1$ and $m_j=0$ for $j\ge 4$. By Proposition~\ref{all-twos}, we can assume  $\mu=(2^{m_2},3)$ for some $m_2\ge 1$ (since $|S|=4$).
Then our desired first row bound of $2\max S - |S| + 1 + \sum_{i\ge 3}(i-2)m_i$ (from  Lemma~\ref{lem:Stab-bound}) simplifies to 
$2\max S - |S| + 1 + 1 = 2(m_2 + 2) - 2  = 2m_2 + 2 $, 
using Item (1) of Lemma~\ref{lem:equiv-Statement} for the first equality.

As before, let $\cS^\lambda$ be an irreducible appearing in the essential part $W\!H_{S,\, ( \mu, 1^{n-|\mu|})}$, where $\mu=(2^{m_2},3)$.  Then $|\lambda|=2m_2+3$, and our desired first row bound reduces to the inequality $\lambda_1\le |\lambda|-1$.
Equivalently, we wish to rule out the possibility that $\lambda$ consists of a single row.

 Consider a ribbon for elements of our homology basis for the rank-selected interval $(\hat 0, u)^S$, where $u$ has type $(2^{m_2},3)$. 
 Since $|S|=4$, our ribbon can have 3 columns each with 2 boxes, or   a column of size 3 and one of size 2, or   a single column of size 4.
 \begin{center}
 $\ribbon{0/0, 1/0, 2/0, 2/1, 3/1, 4/1, 4/2, 5/2, 5/3, 6/3, 7/3}$
 $\ribbon{0/0, 1/0, 2/0, 2/1, 2/2, 3/2, 4/2, 4/3, 5/3, 6/3, 7/3}$
 $\ribbon{0/0, 1/0, 2/0, 2/1, 2/2, 2/3, 3/3, 4/3, 5/3, 6/3, 7/3}$\end{center}
 
 Assume without loss of generality that the block of size 3 in $u$ is $\{1,2,3\}$.   Then in the labeled ribbon $T'$, at most two boxes in a column will have  labels coming from the part of size 3, the allowed pairs of EL labels  being 
 12, 13 or 12, 23.  This forces some column to have two boxes both with labels of the form $i j$ where $\{i,j\}$ is a block of size 2.  That is, some column has two adjacent boxes $B_1, B_2$, say, labeled $i_1, j_1$ and $i_2 j_2$, where the set $\{i_1, i_2, j_1, j_2\}$ has size 4, because these letters come from two distinct blocks each of size 2. Moreover, because all blocks of $u$ other than $\{1,2,3\}$ have size 2, these letters do not appear in any other box of the ribbon.

All four of these letters appear in the first row of the Young symmetrizer of shape $(\lambda_1)=\lambda$, so the row stabilizer contains $(1+\tau)$ as a factor, where $\tau=(i_1, i_2)(j_1, j_2)$ is the product of two transpositions whose only effect on $T'$ is to  swap the labels in the boxes $B_1, B_2$. Swapping labels in adjacent boxes of a column multiplies our homology basis element $v_{T'}$ by $(-1)$, and hence, by Lemma~\ref{lem:YoungSym-annihilate-ribbon}, this forces the Young symmetrizer corresponding to any tableau $T$ of shape $\lambda$ to give 0 when applied to our generator.
\end{proof}

\subsection{Proof of Hersh-Reiner sharp stability conjecture for $\Pi_n$}\label{H-R-Conj-proved}

Consider the rank set $S = \{ s_1,\dots ,s_r \} $ with $1\le s_1 < s_2 < \cdots < s_r  = \max S$. 
The main thing we will need to do in order to 
prove the conjecture is to  
show that no irreducible $S^{\lambda }$ appearing in $\widehat{W\!H}_S(\Pi_n)$ has first row length $\lambda_1$ strictly greater than $4\max S - |S| + 1 - |\lambda |$.  
The plan is to prove this fact  
for each possible type of set partition $u$ of rank 
$\max S $, in other words for each component $\widehat{W\!H}_{S,\lambda }(\Pi_n)$ in the Whitney homology; recall that the type $(m_1(u),m_2(u),\dots ,m_n(u))$ of an element $u\in \Pi_n$ is the vector in which $m_i(u)$ counts the number of blocks of size $i$ in $u$. 
We will accomplish this by using Theorem~\ref{thm:irrep-detection} in combination with Lemma
~\ref{how-many-ambiguous}, together leading to our proof of the conjecture in Theorem ~\ref{main-partition-stability-theorem}.
 
Our approach is as follows. We  apply any Young 
symmetrizer having more than the allowed number of  boxes in the first row to the module  $\widehat{W\!H}_{S,\, (\mu ,1^{n-|\mu |} )}(\Pi_n)$,  using 
the ribbon basis in Section ~\ref{geometric-ribbons}, 
and show that the result is 0.  
The key is to show that, for each element of our ribbon basis, some column of our ribbon filling 
must have two ``swappable boxes'' in it, allowing us to apply 
Lemma~\ref{lem:YoungSym-annihilate-ribbon}.  
The proof of existence of this pair of swappable boxes is quite delicate, and is carried out in 
Lemma~\ref{how-many-ambiguous}.
A good warm-up is the case with $|S|=4$, handled in Lemma ~\ref{lem:sizeSis4}. 

\begin{defn}\label{def:SYT-swappable-ambig-letter}
    Given an SYT $T$ and an element $u\in \Pi_n$, the \emph{swappable pairs of letters} in $T$ with respect to $u$ are those $i,j \in T$ such that $i$ and $j$ both appear in the first row of $T$ and comprise a block of size 2 in $u$.  A letter $i\in T$ is  \emph{ambiguous} if it does not belong to any swappable pair.  A pair of letters $i,j\in T$ is an \emph{ambiguous pair}  if either $i$ or $j$ is ambiguous.
\end{defn}

\begin{ex} 
Let $T=\tableau{1 & 2 &{\bf 4}& 6 &{\bf 9}\\
                3 & 5\\
                7 & 8}
                $
and let $u=|1|27|{\bf 49}|356|8|$.  Then  $4,9$ is the only swappable pair for $u$, and all other letters are ambiguous.  
\end{ex}

\begin{rk}\label{rk:ambiguous-letter}
    A letter $i$ will be ambiguous if and only if either (a) $i$ appears in a block of size larger than 2 in $u$ or (b)  $i$ appears in a block $\{i,j\}$ of size 2 in $u$ such that $i,j$ are not both in the first row of $T$.
\end{rk}

\begin{defn}\label{def:filling-swappable-ambiguous-box}
Fix a SYT $T$ and a partition $u$ in $\Pi_n$. 
A box in a ribbon with filling $F$ is a \emph{swappable box for the filling $F$} if it contains a swappable pair of letters with respect to  $T$ and $u$.   
 It is an \emph{ambiguous box for $F$} if it contains an ambiguous pair. For clarity in the proofs below, we further distinguish between ambiguous boxes of type (a) or type (b), according to  the distinction made in Remark~\ref{rk:ambiguous-letter}.
\end{defn}

The filling $F$ in 
Example~\ref{ex:ribbon-basis-WhS}, reproduced below,
\[F=\tableau{&45 & 18\\ 12 & 67},\]
has a pair of 
boxes containing the atoms $|45|$ and $|67|$  which are 
swappable boxes 
with respect to the set partition $u = 123|45|67|8$.
On the other hand, the filling 
$F $ in Example~\ref{ex:hom-Pi-1} has no swappable boxes. 

\begin{lem}\label{pigeonhole}
Suppose a ribbon shape  has $j$ boxes which each have a box directly below them in the ribbon, and suppose at most $j-1$ boxes in the ribbon are ambiguous for the  filling $F$.
Then some column has at least two swappable boxes for $F$.
\end{lem}

\begin{proof}
Let $n$ be the number of boxes in our ribbon shape $R$.  Then $R$ must have exactly $n-j$ columns.  It also has at least $n-j+1$ swappable boxes.  Thus, some column must have two or more swappable boxes by the pigeonhole principle.
\end{proof}

It may be helpful to review the statement of Lemma~\ref{lem:Stab-bd-from-ess-part} in order to put the hypotheses below in context.
Let $S=\{s_1<s_2<\cdots<s_r=\max S\}$ be any rank set. Denote by $\Rib_{W\!H}(S)$ the ribbon shape associated to $S$ for the Whitney homology module $W\!H_S(\Pi_n)$. Thus $\Rib_{W\!H}(S)=\Rib(s_1, s_2-s_1,\ldots, \max S-s_{r-1})$, and $|\Rib_{W\!H}(S)|=\max S$. 
\begin{lem}\label{how-many-ambiguous}
Consider a rank set $S$ with associated ribbon $\Rib_{W\!H}(S)$.  Also consider $u\in \Pi_n$ of rank $\max S$ for $n\ge 4\max S - |S| + 1$ 
and a standard tableau $T$ 
of shape $\lambda $ with $\lambda_1 + |\lambda | > 4\max S - |S| + 1$.

Then each saturated chain from $\hat{0}$ to $u$ 
has label sequence giving rise to a filling $F$ of $\Rib_{W\!H}(S)$ such that at most $|S|-2$ boxes in $\Rib_{W\!H}(S)$ 
are ambiguous for the filling $F$.  
\end{lem}

\begin{proof} 
Suppose $u$ has type $(\mu, 1^{n-|\mu|})$. 
Recall that 
$m_i(u)$ denotes the number of parts of size $i$ in $u$.  First we make two important definitions. 
\begin{equation}\label{lambda(u)}
|\lambda (u)|  := \sum_{i\ge 2}im_i(u).
\end{equation}
\begin{equation}\label{lambda1(u)}
\lambda_1(u) := 4\max S - |S| + 2 - |\lambda (u)|.
\end{equation}
Note that $|\lambda (u)|=|\mu|$.

The notation 
$|\lambda (u)|$ has been deliberately chosen in order to   designate   
 the size of $\lambda$ for 
any $S^{\lambda }$  appearing in the essential part of $W\!H_{S,\mu ,1^{n-|\mu |} }(\Pi_n)$,  where $(\mu ,1^{n-|\mu |})$ is the type of $u$. 
Defining  $\lambda_1(u)$ as a function of $|\lambda  (u)|$ rather than as the first row length will be crucial to our method of proof. 

  Let   $\rho \vdash |\lambda (u)|$ be any partition of $|\lambda (u)|$ having $\rho_1 = \lambda_1(u)$,  and let  $T $ be any standard filling of shape $\rho $.  Our task is to prove that there must be at least  $2(\max S - |S| + 2)$ boxes in the first row of $T$ whose entries come in pairs,  comprising  $\max S - |S| + 2$ parts of size 2 in $u$.
 
It is worth noting that for some choices of $u$, no $\rho \vdash |\lambda (u)|$ with $\rho_1=\lambda_1(u)$ exists.  This transpires when we have $|\lambda (u) | < \lambda_1(u)$, rendering the condition to be checked for those $u$ vacuous. An example of  this is the case when $u$ has few or no blocks of size 2.    

Our proof will be by induction on $N(u) := \sum_{i\ge 3}(i-1)m_i(u)$ where $m_i(u)$ is the number of parts of size $i$ in $u$.  Some readers might find it helpful to note that this quantity  $N(u)$  
equals the number of ambiguous boxes for
$F$ of type (a) for the ribbon filling $F$ given by any saturated chain from $\hat{0}$ to $u$.   This is because, 
in our saturated chain from $\hat{0}$ to $u$,  each block $C$ of $u$ of size $i\ge 3$ is created by  merging sub-blocks (of $C$)  $i-1$ times.  Moreover, these $i-1$ merge steps are labeled by the minimal labeling with $i-1$ atoms that appear in ambiguous boxes of type (a) in $F$.

For the remainder of this proof, we assume the statements  of three lemmas,  Lemma~\ref{bounding-first-row-pairs}, \ref{base-case-for-induction} and \ref{swappable-number-cannot-go-down}, 
whose proofs  appear after the current lemma. 

 For any $u$, under the hypotheses of the present lemma,  Lemma ~\ref{bounding-first-row-pairs} asserts  the following fact for the filling  $F$  with respect to the SYT $T$ of shape $\rho $ :
 \begin{equation}\label{eqn:lower-bound-swappable-halfS(u)}
    \text{The number of swappable boxes for  $F$  with respect to $T$  is at least }  
    \frac{1}{2}\Sw(u),
 \end{equation}
 where $\Sw(u)$ is the statistic defined by  
\begin{equation}\label{eqn: defn S(u)}
\Sw(u):= \lambda_1 (u) - \sum_{i\ge 3}im_i(u) - (|\lambda (u)|  - \lambda_1(u)).
\end{equation} 
 Next we invoke    
Lemma ~\ref{base-case-for-induction}.  This  shows that  for each $u$ satisfying $N(u)=0$, we have  
\begin{equation}\label{eqn:lower-bound-for-S(u)} 
\frac{1}{2}\Sw(u)\ge \max S - |S| + 2. 
\end{equation}

Combining \eqref{eqn:lower-bound-swappable-halfS(u)}  and \eqref{eqn:lower-bound-for-S(u)} gives the claimed  lower bound on the number of swappable boxes for the filling $F$  in  the base case  $N(u)=0$ of our induction.

In our inductive step, we will prove our result for all $u$ having $N(u)
= N+1>0$, assuming the result holds  for all $u'$ having
$N(u')<N+1$.  
For each $u$ having $N(u)=N+1$, we will rely on the desired result already holding  for a very particular
$u'$ with $N(u')<N+1$ in order to deduce the result for $u$,  so let us now describe which such 
$u'$ to use for each $u$ with $N(u)>0$.

\begin{itemize}
\item (Step 1)
First, we split some block $B$ of $u$ of size at least 3 into a pair of blocks, one of size $|B|-1$ and the other of size 1; we are guaranteed that  $u$ will have  such a block $B$ by the positivity of $N(u)=N+1$,  
which implies that not all blocks of $u$ are of size 1 or 2.   If $u$ has more than one such block $B$, it does not matter which one we use.    Denote by $B'$ the block of size $|B|-1$ in $u'$ that is obtained from $B$ by splitting off a block of size 1 from $B$. 
\item (Step 2)
Second, we merge two blocks of size 1 from $u$ into a single block of size 2 in $u'$; $u$ is indeed guaranteed to have at least two blocks of size 1 by virtue of how large our desired stability bound is, 
 since  $\max |\lambda | =2\max S$ and $n\ge 4\max S - |S| + 1$  together ensure that every $u$ of rank $\max S$ has at least $2\max S - |S| + 1$ blocks of size 1. 
It does not matter which two blocks of size 1 are merged in this step.  
\end{itemize}

Regardless of our choices, this pair of modifications to $u$, producing $u'$,  will ensure $\rank (u) = \rank (u')$ since $u$ and $u'$ both have the same number of parts.   
Passing from $u$ to $u'$ also reduces the value of the statistic  $N(u) = \sum_{i\ge 3}(i-1)m_i(u)$ to a strictly smaller value $N(u') = \sum_{i\ge 3}(i-1)m_i(u')$; this is because merging two blocks of size 1 in $u$ has no impact on this stastistic, while splitting the block of size $|B|$ reduces this statistic  by 1 (if $|B|>3$) or by 2 (if $|B|=3$).  Thus,  inducting on this statistic allows us to 
assume the result for   $u'$ and use it to deduce the result for $u$, as described next. 

 Lemma~\ref{swappable-number-cannot-go-down} shows that, 
 for $u'$ related to $u$ as above, we have the inequality 
 \begin{equation}\label{eqn:swappable-number-cannot-go-down} 
     \Sw(u)\ge \Sw(u').
 \end{equation}

Combining inequality \eqref{eqn:lower-bound-for-S(u)}, which provides the base case, and inequality \eqref{eqn:swappable-number-cannot-go-down},  which is the inductive step,  we obtain 
$\frac{1}{2}\Sw(u) \ge \max S - |S| + 2$ for all $u\in \Pi_n$ of rank $\max S $.  Finally \eqref{eqn:lower-bound-swappable-halfS(u)}   gives us the requisite lower bound on the number of swappable boxes for any  $u\in \Pi_n$ of rank 
$\max S$ for any  shape $\lambda =(\lambda_1,\lambda_2,\dots )  \vdash |\lambda (u)|$ satisfying 
$|\lambda | + \lambda_1 \ge 4\max S - |S| + 2$. \end{proof}

The next three lemmas, already invoked in the proof of Lemma~\ref{how-many-ambiguous}, establish the crucial properties of the statistic $\Sw(u)$, whose definition we reproduce  below.
$$\Sw(u) = \lambda_1(u) - (|\lambda (u)| - \lambda_1(u)) - \sum_{i\ge 3}im_i(u)$$
where $|\lambda (u) | := \sum_{i\ge 2}im_i(u)$ and $\lambda_1(u) := 4\max S - |S| + 2-|\lambda (u)|$.

\begin{lem}\label{bounding-first-row-pairs}
Let $T$ be  any standard Young tableau  of shape $\lambda = (\lambda_1,\lambda_2,\dots )$ 
with $|\lambda |:= \sum_{i\ge 2}im_i(u)$  and  $\lambda_1 \ge 4\max S - |S| + 2 - |\lambda | $. Then 
a lower bound on the number of swappable boxes with respect to $T$, for any ribbon filling $T'$ arising from a saturated chain from $\hat{0}$ to $u$,   is $\frac{1}{2}\Sw(u)$.
\end{lem}

\begin{proof}
Let $L$ be the set 
consisting of the $k$ leftmost  boxes in the first row of $T$, for some $k\ge 1$. 
Consider the ordered pairs $(i,j)$ of letters $i\ne j$ 
in $L$ 
such that 
$\{i,j\}$ is a block of size 2 in $u$. We refer to such letters $i$ and $j$ 
as being
 {\em first-row-pairable 
in $L$}.   
In particular,  when $L=T$, the number of swappable boxes in the ribbon filling $T'$ with respect to $T$ is  exactly half the number of first-row-pairable letters in $T$. 

To prove this lemma, it suffices to show that $\Sw(u)$ is a lower bound on the number of first-row-pairable letters in $T$. 

The hypothesis $\lambda_1 + |\lambda | \ge 4\max S - |S| + 2$  guarantees that we have at least $\lambda_1(u) := 4\max S - |S| + 2 - |\lambda (u)|$ letters in the first row of $T$.
Now specialize the initial segment  $L$ to be the subset of the first row of $T$ consisting of the leftmost $\lambda_1(u)$ letters in that first row. 
See Figure~\ref{fig:L-inital-segment-of-T}.
\begin{figure}
    \centering
    $\tableau{X &\ldots  &\ldots &X & \  &\ &\ \\
             \ &\ &\ &\ &\ \\ 
             \ &\ &\  }$
    \caption{A subset $L$ of the first row of $T$, indicated by the boxes labelled $X$, as in the proof of Lemma~\ref{bounding-first-row-pairs}.}
    \label{fig:L-inital-segment-of-T}
\end{figure}

The letters in  $L$ which are not first-row-pairable in $L$ 
are of the following two types: 
\begin{enumerate}
\item[Type I:] those  
in a block of $u$ of size larger than 2, or
\item[Type II:] those 
in a block of size 2 in $u$, but paired with some letter not in  $L$. 
\end{enumerate}
Let $A_1$ and $A_2$ be the number of letters in $L$ of Type I and Type II, respectively. Then 
$A_1\le \sum_{i\ge 3}im_i(u)$, since this sum is the number of letters in $T$ appearing in blocks of size larger than 2.

Also $A_2\le |\lambda(u)| - \lambda_1 (u)$ since the latter expression is the total number of letters in $T$ that do not appear  in $L$.  
Hence the number of first-row-pairable letters in $L$ 
is 
\[ \lambda_1(u)-A_1-A_2\ge \lambda_1(u) - \sum_{i\ge 3}im_i(u) - ( |\lambda(u)| - \lambda_1 (u)).\] 
The right-hand side is precisely $\Sw(u)$.
Since the number of first-row-pairable letters in all of $T$ is at least as large as the number of first-row-pairable letters in $L$, we are done.
\end{proof}

\begin{lem}\label{base-case-for-induction}
Under the hypotheses of Lemma ~\ref{how-many-ambiguous}, each $u$ of rank $\max S$ in $\Pi_n$ satisfying $N(u)=0$ has 
$\frac{1}{2} \Sw(u)\ge \max S - |S| + 2$.
\end{lem}

\begin{proof}
 In the  $N(u)=0$ case, we have   $\sum_{i \ge 3}(i-1)m_i =0$, which is equivalent to saying that all parts of $u$ have size at most 2.  Thus,  all letters   in the first row of $T$ appear in  boxes of the ribbon $\Rib_{W\!H}(S)$  that are not ambiguous boxes of type (a) with respect to $T$.  In this case, 
we have $|\lambda (u)| =2 m_2(u)= 2\max S$, implying 
$$\lambda_1(u) =  4\max S - |S| + 2 - |\lambda (u)| = 2\max S - |S| + 2.$$   Also note that $\sum_{i\ge 3}im_i(u)=0$ since $N(u)=0$ implies $m_i(u)=0$ for all $i>2$. 
By definition of $\Sw(u)$, we have 
\begin{equation*}
\begin{split}
\frac{1}{2} \Sw(u) 
    &=  \frac{1}{2}\left( 2\lambda_1(u) - |\lambda (u)| -\sum_{i\ge 3} im_i\right)
\end{split}
\end{equation*}  
But when $N(u)=0$, this  
equals 
\begin{align*}
    & \frac{1}{2} (4\max S - 2|S| + 4 - 2\max S)\\
&= \max S - |S| +2,
\end{align*}
establishing the claim of the lemma for $N(u)=0$.
\end{proof}

\begin{lem}\label{swappable-number-cannot-go-down}
Under the hypotheses of  Lemma ~\ref{how-many-ambiguous}, the two elements  $u', u\in \Pi_n$ constructed in the proof of Lemma ~\ref{how-many-ambiguous} with 
$N(u') < N(u)$ also satisfy $\Sw(u)\ge \Sw(u').$  More precisely, $\Sw(u)=\Sw(u')$ when $|B'|=2$, whereas  $\Sw(u)=\Sw(u')+2$ when $|B'|>2$.
\end{lem}

\begin{proof}  
We need to examine the blocks $B'$ in $u'$ and $B$ in $u$ where $B'$ is obtained by splitting off a single element from $B$.  There are two cases, depending on whether $|B'|=2$ or $|B'|>2$.  In either case, note that $|\lambda (u)| = |\lambda (u')| -1$, because the reduction from $u$ to $u'$ results in a net loss of one block of size 1. 

Because
$\lambda_1(u)+|\lambda(u)|$ equals a constant, namely $4\max S-|S|+2$, this forces  
$\lambda_1 (u) = \lambda_1 (u') + 1$.  Thus, in either case we have
$$\lambda_1(u) - (|\lambda (u)| - \lambda_1(u)) = \lambda_1(u') - (|\lambda (u')| - \lambda_1(u')) +3.$$     In the first case, we have 
$\sum_{i\ge 3}im_i(u) = (\sum_{i\ge 3}im_i(u')) + 3$, giving
\begin{align*}
    \Sw(u) 
    &= \left(\lambda_1(u')-(|\lambda (u')|-\lambda_1(u')) + 3\right) - \left(\sum_{i\ge 3}im_i(u') +3\right)\\ 
    &=  \Sw(u').
\end{align*}
In the second case, we have 
$\sum_{i\ge 3}im_i(u) = \sum_{i\ge 3}im_i(u') +1$, giving
\begin{align*}
    \Sw(u) 
    &= \left(\lambda_1(u') - (|\lambda (u')|-\lambda_1(u')+3\right) - \left(\sum_{i\ge 3}im_i(u') + 1\right)\\
    &=     \Sw(u') + 2.
    \end{align*}
    In either case, $\Sw(u)\ge \Sw(u')$,
    as desired.
\end{proof}

Next we  use Lemma~\ref{lem:YoungSym-annihilate-ribbon} 
to get our desired upper bound on the length of the first row for $\cS^{\lambda}$ appearing in 
$\widehat{W\!H}_S(\Pi_n)$.

\begin{lem}\label{pulled-together}
Let $T$ be a standard Young tableaux of shape $\lambda $ with $\lambda_1 > 2\max S - |S| + 1 + \sum_{i\ge 3}(i-2)m_i$.  Then each $v_{T'}$ in our homology basis has $b_Ta_Tv_{T'} = 0$.  Thus, the multiplicity of $\cS^{\lambda}$ within the essential part of 
$W\!H_S(\Pi_n)$ is 0.
\end{lem}

\begin{proof}
Here we use the existence of a column of $T'$ having two swappable boxes with respect to the standard tableau $T$, as we deduce by combining Lemmas ~\ref{pigeonhole} and ~\ref{how-many-ambiguous} above.
\end{proof}

\begin{thm}\label{main-partition-stability-theorem}
    A sharp stability bound of $4\max S -|S| + 1$ holds for $W\!H_S(\Pi_n)$.  Consequently, the conjectured sharp stability bound of $4 \max S - |S| + 1$ also holds for $\beta_S(\Pi_n)$.
\end{thm}
\begin{proof} Combine 
Lemma~\ref{pulled-together} and Lemma~\ref{lem:Stab-bd-from-ess-part} to get stability.  Sharpness of the bound follows from Proposition~\ref{all-twos} and Proposition~\ref{rank-selection-sharpened-stability}.  
The  second statement is a consequence of 
Corollary~\ref{cor:WHPiS-bound-equals-betaS-bound}.
\end{proof}

In exact analogy with Corollary~\ref{cor:chain-stab-Bool}, we 
deduce the following stability result for the rank-selected chains of $\Pi_n$, implicit in \cite{Hersh-Reiner}.
\begin{cor}\label{cor:chain-stab-Pin}  The rank-selected modules of chains $\alpha_S(\Pi_n)$ stabilize sharply at $4\max S$.  
\end{cor}

Our sharp stability bound also makes evident the following.

\begin{cor}
The smallest sharp stability bound for $\beta_S(\Pi_n)$ for any rank set $S$ is  $3\max S + 1$ for a fixed maximal element 
$\max  (S)$  while  $4\max S$ is the largest sharp stability bound among sets $S$ having  this same fixed maximal element $\max S$. More specifically, for any fixed choice of $\max S$, stability occurs earliest for 
$S = \{ 1,2,\dots ,\max S \} $ and  occurs latest for $S = \{ \max S\} $.  
\end{cor}

Tying this in with the theory of FI-modules, we deduce the following.

\begin{cor}
For any fixed $S$, $\{ W\!H_S(\Pi_n)\} $ is a finitely generated FI-module with FI degree of  $2\max S$ and sharp stability bound   of $4\max S - |S| + 1$.  
\end{cor}

\begin{proof}
The fact that we have a finitely generated FI-module  follows from our  decomposition   of  
$W\!H_S(\Pi_n)$ into  components $W\!H_{S,(\mu ,1^{n-|\mu |})}(\Pi_n)$, as in Definition~\ref{def:WHS-submodules},
each  generated by finitely many individual FI-module generators, 
using the description of FI-module generators appearing in~\cite[Proposition 2.6]{CEF}.  
The degree calculation of $2\max S$ comes from 
considering the component 
$W\!H_{S,(2^{\max S},1^{n-2\max S})}(\Pi_n)$ which has
$|\lambda | = 2\max S$ and our observation (see Lemma~\ref{lem:equiv-Statement}, Item (2))  that every other component 
$W\!H_{S,(\mu ,1^{n-|\mu |})}(\Pi_n)$ is generated in strictly lower degrees. 

 Alternatively this result may be deduced from our having now verified the three requirements for uniform representation stability, which  implies that $\{ W\!H_S(\Pi_n) \} $ is a finitely generated FI-module.  
\end{proof}

\begin{cor}
    The rank-selected homology $\{ \beta_S(\Pi_n)\} $ for fixed rank set $S$ is a finitely generated FI-module with FI degree of $2\max S$ and sharp stability bound of 
    $4\max S - |S| + 1.$
\end{cor}

\begin{proof}
We proved this is an FI-module in Proposition ~\ref{get-FI-modules}, so by Theorem 1.13 from \cite{CEF} it suffices to check that it satisfies the three requirements for uniform representation stability with stable range $4\max S - |S| + 1.$  The first two requirements were proven in Corollaries~\ref{first-uniform} and~\ref{second-beta}.
The requisite sharp stability bound was proven in  Theorem ~\ref{main-partition-stability-theorem}.
\end{proof}

\section{More precise 
stability bounds for individual  submodules
$W\!H_{S,\lambda }(\Pi_n)$ of $W\!H_S(\Pi_n)$}\label{precise-section}

Next we show how certain individual components of 
$W\!H_S(\Pi_n)$ stabilize earlier than our sharp bound 
for the entire $\fS_n$-module $W\!H_S(\Pi_n)$.  A key tool for doing this is the following lemma.

\begin{lem}\label{growth-bound}
    Given any $S\subseteq \{ 1,2,\dots ,n-2\} $ and any $u\in \Pi_n$ of rank $\max S$, we have  
    $\Sw(u)\ge \Sw(u_0) + 2\sum_{i\ge 4}(i-3)m_i(u)$ for any 
    $u_0\in \Pi_n$ of rank $\max S $ satisfying 
    $N(u_0)=0.$
\end{lem}

\begin{proof}
We use the fact that $u$ may be obtained from some $u_0\in \Pi_n$ satisfying $N(u_0)=0$ by a series of inductive steps $u' \rightarrow u$ of the type described in the proof of Lemma ~\ref{how-many-ambiguous}, each of which has $N(u)>N(u')$.  These steps each increase the size of a block referred to as $B'$, of size $|B'|\ge 2$ in $u'$, to a size one larger in $u$.  
By Lemma ~\ref{swappable-number-cannot-go-down}, we have 
$\Sw(u) = \Sw(u')+2$ whenever $|B'|>2$, while 
$\Sw(u)=\Sw(u')$ whenever $|B'|=2$.  Thus, it suffices to show that there must be at least $\sum_{i\ge 4}(i-3)m_i(u)$ of these inductive steps having $|B'|>2$.

Notice that $u_0$ has no blocks of size larger than 2, since  $N(u_0)=0$.  
Thus, each block of size $i\ge 4$ in $u$ requires exactly  $i-3$ of the $u'\rightarrow u$ inductive steps  that each satisfy $\Sw(u)=\Sw(u')+2$, since each such block of size $i\ge 4$ in $u$  requires a series of $i-3$ of these steps, in which the block is enlarged from size 3 to size $i$.   Thus, any path of inductive steps from $u_0$ to $u$ requires a total of  $\sum_{i\ge 3}(i-3)m_i(u)$   of the inductive steps of the type that each increase the value of $\Sw$ by 2. 
The upshot is that $\Sw(u) = \Sw(u_0) + 2\sum_{i\ge 4}(i-3)m_i(u)$.
\end{proof}

\begin{thm}
Let $u\in \Pi_n$ be any set partition of rank $\max S $ and type 
$(\mu ,1^{n-|\mu |})$.  Then the $\fS_n$-module $W\!H_{S,(\mu ,1^{n-|\mu |})}(\Pi_n)$ satisfies the stability bound
$4\max S - |S| + 1 - K(u)$ where 
$K(u)  = \sum_{i\ge 4} (i-3)m_i(u)$.
\end{thm}

\begin{proof}
First we verify  that the proof of Lemma ~\ref{bounding-first-row-pairs} can be very gently modified to show the following: Let $u$ be an element of rank $\max S $ in $\Pi_n$ and let $T$ be any standard Young tableau of shape $\lambda $ with 
$|\lambda | = \sum_{i\ge 2}im_i(u)$ and $\lambda_1\ge 4\max S - |S| + 2 - |\lambda | - K$ for some constant $K$.  Then a lower bound on the number of swappable boxes with respect to $T$, for any ribbon filling $F$ arising from a saturated chain from $\hat{0}$ to $u$, is $\frac{1}{2}\Sw(u) - K$.  The proof modification simply involves making $L$ smaller, by taking $L$ to be the leftmost $\lambda_1(u) - K$ boxes in the first row of $T$.  This has the effect of decreasing by $K$ our lower bound, thus replacing $\frac{1}{2}\Sw(u)$ with $\frac{1}{2}\Sw(u)-K$. 
That completes the justification of the modified version of Lemma ~\ref{bounding-first-row-pairs}. 

Now let $\Sw_K(u) = \Sw(u) - 2K$, so $\frac{1}{2}\Sw_K(u) = \frac{1}{2}\Sw(u) - K$.  We  use Lemma ~\ref{base-case-for-induction} to deduce that 
$\frac{1}{2}\Sw_K(u_0)\ge \max S - |S| + 2 - K$ for any $u_0$ satisfying $N(u_0)=0$.  
By Lemma ~\ref{growth-bound}, we have that 
$\Sw(u) = \Sw(u_0)+ 2\sum_{i\ge 4}(i-3)m_i(u)$ for any $u_0$ satisfying $N(u_0)=0$.  
For  $K = \sum_{i\ge 4}(i-1)m_i(u)$, 
this yields  
$$\frac{1}{2} \Sw_K(u) =  \frac{1}{2}(\Sw_K(u_0) + 2K) = \frac{1}{2}\Sw(u_0).$$
But we may combine this with the 
inequality $$ \frac{1}{2}\Sw(u_0) \ge  \max S - |S| + 2,$$   
which follows from  Lemma ~\ref{swappable-number-cannot-go-down},  to deduce the inequality $\frac{1}{2}\Sw_K(u)\ge \max S - |S| + 2$ for $K= \sum_{i\ge 4}(i-3)m_i(u)$.  
Now we combine our earlier  inequality, namely the lower bound of $\frac{1}{2}\Sw_K(u)$ on the number of swappable boxes with respect to $T$, with the  second inequality $\frac{1}{2}\Sw_K(u)\ge \max S - |S| + 2$ we have just derived.  Together these  give the desired lower bound of $\max S - |S| + 2$ 
on the number of swappable boxes with respect to $T$. 
This shows that any $\cS^{\lambda}$ satisfying 
$|\lambda | + \lambda_1 > 4\max S - |S| + 1 - K(u)$ has multiplicity 0 in $W\!H_{S,(\mu , 1^{n-|\mu |})}$, yielding the desired stability bound.  
\end{proof}

\begin{prop}\label{prop:no-parts-lessthan-k}
Let $\mu$ be an  integer partition of size at most $n$ and having no parts of size less than $k$. Then  $W\!H_{S, (\mu ,1^{n-|\mu |})}(\Pi_n)$ is representation  stable for $n\ge \frac{2k}{k-1}\max S$. 
\end{prop}

\begin{proof}
We use Lemma~\ref{lem:Stab-bd-from-ess-part}, but first 
we show that $|\mu|\le \frac{k}{k-1}\max S$.

From Item (1) of Lemma~\ref{lem:equiv-Statement}, we have 
$\max S=|\mu| -\ell(\mu)$.
Also, since every part of $\mu$ has size at least k, we have  $|\mu|\ge k \ell(\mu)$. We conclude that $\max S \ge (k-1)\ell(\mu)$, and thus 
$|\mu|=\max S+\ell(\mu)\le \frac{k}{k-1}\max S$ \ as claimed.

Now consider  the essential part $\widehat{W\!H}_{S, (\mu, 1^{n-|\mu|})}(\Pi_n)$.  By definition, this is an $\fS_{|\mu|}$-module. 
It follows that any Specht module $\cS^\lambda$ appearing in the essential part has $|\lambda|=|\mu|$, and hence 
every such $\lambda$ satisfies 
\[ k\max S\ge (k-1)|\mu|\ge \frac{k-1}{2}(|\lambda|+\lambda_1).\]
The last inequality is due to the fact that $2|\mu|=2|\lambda|\ge|\lambda| +\lambda_1 $. The conclusion follows from Lemma~\ref{lem:Stab-bd-from-ess-part}.
\end{proof}

\section{Graphical and matroidal Specht-like modules}\label{graphical-section}

One could in principle define Specht-like modules for any skew shape (or straight shape) $\lambda /\mu $ with $n$ boxes and any matroid of rank $n$ as follows.   Consider all possible  fillings of the boxes of the shape  with the elements of the ground set, requiring that the set of entries in the boxes of the shape be a basis  
for the matroid.   Now consider the polytabloids generated by these fillings and the vector space generated by all such polytabloids.  

This gives  a matroid analogue of a Specht module.    
In the language of geometric lattices, the entries in a  filling would be any set of $n$ atoms whose join is the top element in the geometric lattice. 
When this geometric lattice is  the partition lattice, the associated matroid is the graphic matroid given by a complete graph on $n$ labeled vertices, and then the fillings by matroid bases consist of collections of graph edges comprising spanning trees.  

\begin{prop} 
  Fix any ordering on the ground set of a matroid  (i.e., the  atoms of a geometric lattice).  Consider the polytabloids given by the standard fillings with respect to this atom ordering  with independent sets of atoms.  Then these polytabloids generate all of  the polytabloids resulting from fillings by independent sets.
    \end{prop}
    \begin{proof}
    For any particular independent set $\{ F_1,\dots ,F_n\} $,
    the set of  polytabloids $v_{\sigma F}$ with reading word  
    $(F_{\sigma (1)},\dots ,F_{\sigma (n)})$ for  $\sigma \in S_n$  is generated by the  smaller set of $v_{\sigma F}$ in which the filling $\sigma F$ is standard; 
    this follows directly from the  proof showing the traditional Specht modules are generated by the polytabloids given by standard fillings (see e.g. \cite{Fulton}).  
    \end{proof}

When $\lambda /\mu $ is a ribbon shape, applying the map $f_{chain}$ to  this generating set above
introduces further  relations beyond the Garnir relations (see \cite{Fulton}).  We showed in Section ~\ref{geometric-ribbon-section} that these extra relations lead to smaller bases indexed by those standard fillings whose reading words are 
$\minNBC $ independent sets.

Ribbon shapes are particularly well suited  
to the study of the rank-selected homology of a geometric lattice, hence our focus throughout this paper on ribbon shapes. However,  one could imagine uses for other shapes as well.

\begin{rk}
There is a substantial literature regarding objects known as  set-valued tableaux (see  e.g. ~\cite{Buch}, ~\cite{KMY}, ~\cite{KMY2}, ~\cite{MT}, ~\cite{PS}) 
These were first introduced by Anders Buch in \cite{Buch} as a combinatorial tool for calculating structure constants for the K-theory of the Grassmannian.  In the case of the partition lattice $\Pi_n$, our ribbon fillings with pairs $\{ i,j\} $ satisfying $1\le i<j\le n$ may also  be regarded as set-valued tableaux.
However, the established notion of standard filling  for set-valued tableaux in the context of studying K-theory of Grassmannians  \cite{Buch} 
does not coincide with the version of standardness needed for our approach to understanding the  rank-selected homology of the partition lattice.
We are not aware of any prior work organizing the set-valued tableaux of \cite{Buch} into Specht-like modules. 
\end{rk}

\section{Acknowledgments}
The authors thank  Colin Crowley, Shreeya Moghe, Scott Neville, Franco Saliola, 
and Alex Yong 
for helpful discussions, for proofreading and feedback, and for assistance with using software to calculate examples and to generate figures involving set-valued tableaux.  They especially thank 
Dan Dugger and Vic Reiner  for helping them better understand FI-modules,  Nick Proudfoot for asking how their  ribbon bases  relate to the monomial bases for the graded pieces  
of the Orlik-Solomon algebra, and Ben Elias for asking whether their theory also applies to Specht modules of shapes other than ribbon shapes.  

They are also extremely  grateful to ICERM and Brown University for providing a wonderful work environment during the Fall 2025 semester program.

\end{document}